%% file: main.tex
\documentclass[11pt, reqno]{amsart}

\usepackage{amsfonts, amssymb}
\usepackage{graphicx}
\usepackage{hyperref}
\usepackage{color}
\usepackage{parskip}
\numberwithin{equation}{section}

\newtheorem{thm}{Theorem}[section]
\newtheorem{cor}[thm]{Corollary}
\newtheorem{lem}[thm]{Lemma}

\newtheorem{prop}[thm]{Proposition}

\theoremstyle{remark}
\newtheorem{rem}[thm]{Remark}

\theoremstyle{definition}
\newtheorem{defin}[thm]{Definition}

\newcommand{\Alo}{A_{m, \mathrm{lo \wedge lo}}}

\newcommand{\Blo}{B_{m, \mathrm{lo \wedge lo}}}
\newcommand{\Bhi}{B_{m, \mathrm{hi \lor hi}}}

\newcommand{\Alol}{A_{\ell, \mathrm{lo \wedge lo}}}

\newcommand{\Nlo}{N_{\mathrm{lo}}}

\newcommand{\cA}{\mathcal{A}}
\newcommand{\cC}{\mathcal{C}}

\newcommand{\cN}{\mathcal{N}}
\newcommand{\cP}{\mathcal{P}}
\newcommand{\cT}{\mathcal{T}}

\newcommand{\oB}{\overline{B}}
\newcommand{\oG}{\overline{G}}
\newcommand{\oN}{\overline{N}}

\newcommand{\tG}{\widetilde{G}}
\newcommand{\tN}{\widetilde{N}}

\newcommand{\e}{\mathbf{e}}

\newcommand{\kx}{k_{\mathrm{max}}}
\newcommand{\kd}{k_{\mathrm{mid}}}
\newcommand{\kn}{k_{\mathrm{min}}}

\newcommand{\dx}{{\; dx}}
\newcommand{\ds}{{\; ds}}
\newcommand{\rr}{{\mathbf{R}^2}}

\newcommand{\cV}{{\epsilon}}

\newcommand{\Ec}{{E_{\mathrm{crit}}}}

\newcommand{\R}{{\mathbf{R}}}
\newcommand{\Rd}{{\mathbf{R}^d}}
\newcommand{\RR}{{\mathbf{R}^2}}

\newcommand{\nline}{{\vspace{\baselineskip}}}

\renewcommand{\Re}{\mathrm{Re}}
\renewcommand{\Im}{\mathrm{Im}}

\title[Conditionally global Schr\"odinger maps]{Conditional global regularity of Schr\"odinger maps: 
subthreshold dispersed energy}
\author[P. Smith]{Paul Smith}
\date{}
\address{University of California, Berkeley}
\email{smith@math.berkeley.edu}

\thanks{The author is supported in part by NSF grant DMS-1103877. 
This research was carried out at UCLA with their support and at UC Berkeley with NSF support.}

\begin{document}

\input{Abstract}

\maketitle

\tableofcontents

\input{Introduction}
\input{GaugeFieldEquations}
\input{FunctionSpaces}
\input{MainTheorem}
\input{LocalSmoothingBilinearStrichartz}
\input{CaloricGauge}
\input{ParabolicEstimates}

\emph{Acknowledgments.}
The author sincerely thanks Terence Tao for all of his guidance, support, and encouragement,
and thanks Daniel Tataru and Ioan Bejenaru for their encouragement and helpful
talks and discussions on wave maps and Schr\"odinger maps. 
The author would also like to thank the anonymous referee for pointing out that the proof of Lemma \ref{L:VmBound} is only valid for $\sigma < \frac{1}{2} - \delta^+$, for supplying a suggestion as to how that lemma might be extended to larger $\sigma$, and for helpful suggestions pertaining to overall exposition.

\bibliography{SMbib}
\bibliographystyle{amsplain}

\end{document}

%% file: Abstract.tex
\begin{abstract}
We consider the Schr\"odinger map initial value problem
\begin{equation*}
\begin{cases}
\partial_t \varphi &= \quad \varphi \times \Delta \varphi \\
\varphi(x, 0) &= \quad \varphi_0(x),
\end{cases}
\label{SMa}
\end{equation*}
with
$\varphi_0 : \mathbf{R}^2 \to \mathbf{S}^2 \hookrightarrow \mathbf{R}^3$
a smooth $H_Q^\infty$ map from the
Euclidean space $\mathbf{R}^2$ to the sphere $\mathbf{S}^2$
with subthreshold ($< 4 \pi$) energy.
Assuming an a priori $L^4$ boundedness condition on the solution $\varphi$,
we prove that the Schr\"odinger map system admits a unique global 
smooth solution $\varphi \in C( \mathbf{R} \to H_Q^\infty)$
provided that the initial data $\varphi_0$ is sufficiently energy-dispersed, i.e., sufficiently small
in the critical Besov space $\dot{B}^1_{2, \infty}$.
Also shown are global-in-time bounds on certain Sobolev norms of $\varphi$.
Toward these ends we establish improved local smoothing and bilinear 
Strichartz estimates,
adapting the Planchon-Vega approach
to such estimates to the nonlinear setting of Schr\"odinger maps.
\end{abstract}

%% file: Introduction.tex
\section{Introduction} \label{S:Introduction}
We consider the
Schr\"odinger map initial value problem
\begin{equation}
\begin{cases}
\partial_t \varphi &= \quad \varphi \times \Delta \varphi \\
\varphi(x, 0) &= \quad \varphi_0(x),
\end{cases}
\label{SM}
\end{equation}
with $\varphi_0 : \mathbf{R}^d \to \mathbf{S}^2 \hookrightarrow \mathbf{R}^3$.

The system (\ref{SM}) enjoys conservation of energy
\begin{equation}
E(\varphi(t)) := \frac{1}{2} \int_{\mathbf{R}^d} \lvert \partial_x \varphi(t) \rvert^2 dx
\label{Energy Def}
\end{equation}
and mass
\begin{equation*}
M(\varphi(t)) := \int_{\mathbf{R}^d} \lvert \varphi(t) - Q \rvert^2 dx,
\end{equation*}
where $Q \in \mathbf{S}^2$ is some fixed base point.
When $d = 2$, both (\ref{SM}) and (\ref{Energy Def}) are invariant
with respect to the scaling
\begin{equation}
\varphi(x, t) \to \varphi(\lambda x, \lambda^2 t), \quad \quad \lambda > 0,
\label{E:scaling}
\end{equation}
in which case
we call (\ref{SM}) \emph{energy-critical}.
In this article we restrict ourselves to the energy-critical setting.

For the physical significance of (\ref{SM}), see
\cite{ChShUh00, NaStUh03-2, PaTo91, La67}.
The system also arises naturally
from the (scalar-valued) free linear Schr\"odinger equation
\begin{equation*}
(\partial_t + i \Delta) u = 0
\end{equation*}
by replacing the target manifold $\mathbf{C}$ with the sphere
$\mathbf{S}^2 \hookrightarrow \mathbf{R}^3$, which then requires replacing $\Delta u$
with 
$(u^* \nabla)_j \partial_j u = \Delta u - \perp(\Delta u)$ and $i$ with the complex structure
$u \times \cdot$. Here $\perp$ denotes orthogonal projection onto the normal bundle,
which, for a given point $(x, t)$, is spanned by $u(x, t)$.
For more general analogues of (\ref{SM}), e.g., for K\"ahler targets
other than $\mathbf{S}^2$, see \cite{DiWa01, Mc07, NaShVeZe07}.
See also \cite{KePoVe00, KeNa05, IoKe11} for connections with
other spin systems.
The local theory for Schr\"odinger maps
is developed in \cite{SuSuBa86, ChShUh00, DiWa01, Mc07}.
For global results in the $d = 1$ setting, see \cite{ChShUh00, RoRuYaSt09}.
For $d \geq 3$, see \cite{Be08, Be08-2, BeIoKe07, IoKe06, IoKe07-2, BeIoKeTa11}.
Concerning the related
modified Schr\"odinger map system, see
\cite{Ka05, KaKo07, NaShVeZe07}.

The small-energy (take $d = 2$) theory for (\ref{SM}) is now well-understood: 
building upon previous work (e.g., see below or \cite[\S 1]{BeIoKeTa11} for a brief history), 
global wellposedness
and global-in-time bounds on certain Sobolev norms
are shown in \cite{BeIoKeTa11} given
initial data with sufficiently small energy.
The high-energy theory, however, is still very much in development.
One of the main goals is to establish
what is known as the \emph{threshold conjecture}, which
asserts that global wellposedness holds for (\ref{SM}) given
initial data with energy below a certain energy threshold,
and that finite-time blowup is possible for certain initial data with energy above this threshold.
The threshold is directly tied to the nontrivial stationary
solutions of (\ref{SM}), i.e., maps $\phi$ into $\mathbf{S}^2$ 
that satisfy
\begin{equation*}
\phi \times \Delta \phi \equiv 0
\end{equation*}
and that do not send all of $\mathbf{R}^2$ to a single point of $\mathbf{S}^2$.
Therefore we identify such stationary solutions with nontrivial
harmonic maps $\mathbf{R}^2 \to \mathbf{S}^2$,
which we refer to as \emph{solitons} for (\ref{SM}).
It turns out that
there exist no nontrivial harmonic maps into the sphere $\mathbf{S}^2$
with energy less than $4 \pi$, and that
the harmonic map given by the inverse of stereographic projection has energy precisely
equal to $4 \pi =: \Ec$.
We therefore refer to the range of energies $[0, \Ec)$ as \emph{subthreshold}, and call
$\Ec$ the \emph{critical} or \emph{threshold energy}.

Recently, an analogous threshold conjecture was established for wave maps
(see \cite{KrScTa08, RoSt06, StTa10-2, StTa10}
and, for hyperbolic space, \cite{KrSc09, T3, T4, T5, T6, T7}).
When $\mathcal{M}$ is a hyperbolic space, or, as in \cite{StTa10-2, StTa10},
a generic compact manifold, we may define the associated energy threshold 
$\Ec = \Ec(\mathcal{M})$ as follows.
Given a target manifold $\mathcal{M}$, consider the collection $\mathcal{S}$
of all non-constant finite-energy harmonic maps $\phi : \mathbf{R}^2 \to \mathcal{M}$.
If this set is empty, as is, for instance, the case when $\mathcal{M}$ is equal to a
hyperbolic space $\mathbf{H}^m$, then we formally set $\Ec = + \infty$.
If $\mathcal{S}$ is nonempty, then it turns out that the set $\{ E(\phi) : \phi \in \mathcal{S} \}$
has a least element and that, moreover, this energy value is positive.  In such case
we call this least energy $\Ec$.
The threshold $\Ec$ 
depends upon geometric and topological properties of the target manifold
$\mathcal{M}$;
see \cite[Chapter 6]{LiWa08} for further discussion.
This definition yields $\Ec = 4 \pi$ in the case of the sphere $\mathbf{S}^2$.
For further discussion of the critical energy level in the
wave maps setting, see \cite{StTa10, T3}.

We now summarize what is known for Schr\"odinger maps in $d = 2$.
Some of these developments postdate the submission of this article.
Asymptotic stability of harmonic maps of topological degree $\lvert m \rvert \geq 4$
under the Schr\"odinger flow
is established in \cite{GuKaTs08}.
The result is extended to maps of degree $\lvert m \rvert \geq 3$ in
\cite{GuNaTs10}.
A certain energy-class instability for degree-1 solitons of (\ref{SM}) is shown in \cite{BeTa10},
where it is also shown that global solutions always exist
for small localized equivariant perturbations of degree-1 solitons.
Finite-time blowup for (\ref{SM}) is demonstrated in \cite{MeRaRo11-3, MeRaRo11-2},
using less-localized equivariant perturbations of degree-1 solitons, thus
resolving the blowup assertion of the threshold conjecture.
Global wellposedness given data with small critical Sobolev norm (in all dimensions
$d \geq 2$) is shown in \cite{BeIoKeTa11}.
Recent work of the author \cite{Sm11} extends the result of \cite{BeIoKeTa11}
and the present conditional result to global regularity
(in $d = 2$) assuming small critical Besov norm $\dot{B}^{1}_{2, \infty}$.
In the radial setting (which excludes harmonic maps), \cite{GuKo11} establishes 
global wellposedness at any energy level.
Most recently, \cite{BeIoKeTa11b} establishes global existence and uniqueness
as well as scattering given 1-equivariant data with energy less than $4 \pi$.
Although stating the results only for data with energy less than $4 \pi$, \cite{BeIoKeTa11b} also notes
that their proofs remain valid for maps with energy slightly larger than $4 \pi$, suggesting
that the ``right" threshold conjecture for Schr\"odinger maps should be stated also in terms
of homotopy class, leading to a threshold of $8 \pi$ rather than $4 \pi$
in the case where the target is $\mathbf{S}^2$.
See the Introduction of \cite{BeIoKeTa11b} for further discussion of this point.

The main purpose of this paper is to show that 
(\ref{SM}) admits a unique smooth global solution $\varphi$
given smooth initial data $\varphi_0$ satisfying appropriate energy conditions
and assuming a priori boundedness of a certain $L^4$ spacetime norm of 
the spatial gradient of the solution $\varphi$.  In particular, we admit a restricted class of
initial data with energy ranging over the entire subthreshold range.  
As such, our main result is a small step toward
a large data regularity theory for (\ref{SM}) and the attendant
threshold conjecture.

In order to go beyond the small-energy results of \cite{BeIoKeTa11}, we
introduce physical-space proofs of local smoothing and bilinear Strichartz
estimates, in the spirit of \cite{PlVe09, PlVe, Tls}, that do not heavily depend upon perturbative methods.  
The local smoothing estimate that we establish is a nonlinear analogue of that shown in
\cite{IoKe06}.
The bilinear Strichartz estimate is a nonlinear analogue of the improved
bilinear Strichartz estimate of \cite{Bo98}.
These proofs more naturally account for magnetic nonlinearities, and
the technique developed here we believe to be of independent interest
and applicable to other settings.
For local smoothing in the context
of Schr\"odinger equations, see \cite{KePoVe93, KePoVe98, KePoVe04, IoKe05, IoKe06, IoKe07-2}.
For other Strichartz and smoothing results for magnetic Schr\"odinger equations,
see \cite{St07, DAnFa08, DAnFaVeVi10, ErGoSc08, ErGoSc09, FaVe09} and the references therein.
We also use in a fundamental way
the subthreshold \emph{caloric gauge} of \cite{Sm09}, which is an extension
of a construction introduced in \cite{Trenorm}.

To make these statements more precise, we now turn to some basic definitions
and elementary observations.

\subsection{Preliminaries}

First we establish some basic notation.  The boldfaced letters $\mathbf{Z}$ and $\mathbf{R}$
respectively denote the integers and real numbers. 
We use $\mathbf{Z}_+ = \{0, 1, 2 \ldots \}$ to denote the nonnegative integers.
Usual Lebesgue function spaces are denoted by $L^p$, and these sometimes include a subscript
to indicate the variable or variables of integration.  When function spaces are iterated, e.g., 
$L^\infty_t L^2_x$, the norms are applied starting with the rightmost one.
When we use $L^4$ without subscripts, we mean $L^4_{t,x}$.

We use $\mathbf{S}^2$ to denote the standard 2-sphere embedded in 3-dimensional Euclidean space:
$\{ x \in \mathbf{R}^3 : \lvert x \rvert = 1 \}$.  The ambient space $\mathbf{R}^3$ carries the usual
metric and $\mathbf{S}^2$ the inherited one.
Throughout, $\mathbf{S}^1$ denotes the unit circle.

We use
$\partial_x = ( \partial_{x_1}, \partial_{x_2} ) = ( \partial_1, \partial_2)$ 
to denote the gradient operator,
as throughout ``$\nabla$'' will stand for the Riemannian connection on $\mathbf{S}^2$.
As usual, ``$\Delta$" denotes the (flat) spatial Laplacian.

The symbol $\lvert \partial_x \rvert^\sigma$ denotes the Fourier multiplier with symbol
$\lvert \xi \rvert^\sigma$.  We also use standard Littlewood-Paley Fourier multipliers
$P_k$ and $P_{\leq k}$, respectively denoting restrictions
to frequencies $\sim 2^k$ and $\lesssim 2^k$; 
see section \S \ref{S:Function spaces}
for precise definitions.  We use $\hat{f}$ to denote the Fourier transform of a function $f$
in the spatial variables.

We also employ without further comment (finite-dimensional) vector-valued analogues of the above.

We use $f \lesssim g$ to denote the estimate $\lvert f \rvert \leq C \lvert g \rvert$ 
for an absolute constant $C > 0$.  As usual, the constant is allowed
to change line-to-line.
To indicate dependence of the implicit constant upon parameters (which,
for instance, can include functions), we use subscripts,
e.g. $f \lesssim_k g$.  As an equivalent alternative we write $f = O(g)$ (or,
with subscripts, $f = O_k(g)$, for instance) to denote $\lvert f \rvert \leq C \lvert g \rvert$.
If both $f \lesssim g$ and $g \lesssim f$, then we indicate this by writing $f \sim g$.

Now we introduce the notion of Sobolev spaces of functions mapping
from Euclidean space into $\mathbf{S}^2$.  The spaces are constructed
with respect to a choice of base point $Q \in \mathbf{S}^2$, the purpose of which
is to define a notion of decay:
instead of decaying to zero at infinity, our Sobolev class functions decay to $Q$.

For $\sigma \in [0, \infty)$,
let $H^\sigma = H^\sigma(\rr)$ denote the usual Sobolev space of complex-valued
functions on $\rr$.  For any $Q \in \mathbf{S}^2$, set
\begin{equation*}
H_Q^\sigma := \{ f : \rr \to \mathbf{R}^3 \text{ such that } |f(x)| \equiv 1 \text{ a.e. and }
f - Q \in H^\sigma \}.
\end{equation*}
This is a metric space with induced distance $d_Q^\sigma(f, g) = \lVert f - g \rVert_{H^\sigma}$.
For $f \in H_Q^\sigma$ we set $\lVert f \rVert_{H_Q^\sigma} = d_Q^\sigma(f, Q)$ for short.
We also define the spaces
\begin{equation*}
H^\infty := \bigcap_{\sigma \in \mathbf{Z}_+} H^\sigma
\quad \text{and} \quad
H_Q^\infty := \bigcap_{\sigma \in \mathbf{Z}_+} H_Q^\sigma.
\end{equation*}

For any time $T \in (0, \infty)$, the above definitions
may be extended to the spacetime slab $\rr \times (-T, T)$ (or $\rr \times [-T, T]$).
For any $\sigma, \rho \in \mathbf{Z}_+$,
let $H^{\sigma, \rho}(T)$ denote the Sobolev space of complex-valued functions on
$\rr \times (-T, T)$ with the norm
\begin{equation*}
\lVert f \rVert_{H^{\sigma, \rho}(T)} := \sup_{t \in (-T, T)}
\sum_{\rho^\prime = 0}^\rho \lVert \partial_t^{\rho^\prime} f(\cdot, t)\rVert_{H^\sigma},
\end{equation*}
and for $Q \in \mathbf{S}^2$ endow
\begin{equation*}
H_{Q}^{\sigma, \rho} := \{ f : \rr \times (-T, T) \to \mathbf{R}^3
\text{ such that } | f(x, t) | \equiv 1
\text{ a.e. and } f - Q \in H^{\sigma, \rho}(T) \}
\end{equation*}
with the metric induced by the $H^{\sigma, \rho}(T)$ norm.  Also, define the spaces
\begin{equation*}
H^{\infty, \infty}(T) = \bigcap_{\sigma, \rho \in \mathbf{Z}_+} H^{\sigma, \rho}(T)
\quad \text{and} \quad
H_Q^{\infty, \infty}(T) = \bigcap_{\sigma, \rho \in \mathbf{Z}_+} H_Q^{\sigma, \rho}(T).
\end{equation*}

For $f \in H^\infty$ and $\sigma \geq 0$ 
we define the homogeneous Sobolev norms as
\begin{equation*}
\lVert f \rVert_{\dot{H}^\sigma} =
\lVert \hat{f}(\xi) \cdot \lvert \xi \rvert^\sigma \rVert_{L^2}.
\end{equation*}

We mention two
important conservation laws obeyed by solutions of the Schr\"odinger
map system (\ref{SM}).  In particular, if $\varphi \in C((T_1, T_2) \to H_Q^\infty)$ solves
(\ref{SM}) on a time interval $(T_1, T_2)$, then both
\begin{equation*}
\int_\rr \lvert \varphi(t) - Q \rvert^2 \dx
\quad \text{and} \quad
\int_\rr \lvert \partial_x \varphi(t) \rvert^2 \dx
\end{equation*}
are conserved.  Hence the Sobolev norms $H_Q^0$ and $H_Q^1$ are conserved,
as well as the energy (\ref{Energy Def}).
Note also the time-reversibility
obeyed by (\ref{SM}), which in particular permits the smooth extension to $(-T, T)$
of a smooth solution on $[0, T)$.

According to our conventions,
\begin{equation*}
\lvert \partial_x \varphi(t) \rvert^2
:=
\sum_{m = 1, 2} \lvert \partial_m \varphi(t) \rvert^2.
\end{equation*}

We can now give a precise statement of a key known local result.
\begin{thm}[Local existence and uniqueness]
If the initial data $\varphi_0$ is such that $\varphi_0 \in H_Q^\infty$ for some $Q \in \mathbf{S}^2$,
then there exists a time $T = T(\lVert \varphi_0 \rVert_{H_Q^{25}}) > 0$ for which there exists
a unique solution $\varphi \in C([-T, T] \to H_Q^\infty)$ of the initial value problem (\ref{SM}).
\label{LWP}
\end{thm}
\begin{proof}
See \cite{SuSuBa86, ChShUh00, DiWa01, Mc07} and the references therein.
\end{proof}

\subsection{Global theory}

Theorem \ref{LWP} yields short-time existence and uniqueness as well as a blow-up criterion;
as such it is central to the continuity arguments used for global results.
In the small-energy setting, global regularity (and more)
was proved for (\ref{SM}) by Bejenaru, Ionescu, Kenig, and Tataru \cite{BeIoKeTa11}.
We now state a special case of their main result, omitting for the sake of brevity
the consideration of higher spatial dimensions
and continuity of the solution map.
\begin{thm}[Global regularity]
Let $Q \in \mathbf{S}^2$.  Then
there exists an $\varepsilon_0 > 0$ such that, for 
any $\varphi_0 \in H_Q^\infty$ with $\lVert \partial_x \varphi_0 \rVert_{L_x^2} \leq \varepsilon_0$, 
there is a unique solution $\varphi \in C(\mathbf{R} \to H_Q^\infty)$ of the initial value problem
(\ref{SM}).  Moreover, for any $T \in [0, \infty)$ and $\sigma \in \mathbf{Z}_+$,
\begin{equation*}
\sup_{t \in (-T, T)} \lVert \varphi(t) \rVert_{H_Q^\sigma} 
\lesssim_{\sigma, T, \lVert \varphi_0 \rVert_{H_Q^\sigma}} 1.
\end{equation*}
Also, given any $\sigma_1 \in \mathbf{Z}_+$, there exists
a positive
$\varepsilon_1 = \varepsilon_1(\sigma_1) \leq \varepsilon_0$
such that 
the uniform bounds
\begin{equation*}
\sup_{t \in \mathbf{R}} \lVert \varphi(t) \rVert_{H_Q^\sigma} 
\lesssim_\sigma
\lVert \varphi_0 \rVert_{H_Q^\sigma}
\end{equation*}
hold for all $1 \leq \sigma \leq \sigma_1$,
provided
$\lVert \partial_x \varphi_0 \rVert_{L_x^2} \leq \varepsilon_1$.
\label{BIKT1.1}
\end{thm}
A complete proof may be found in \cite{BeIoKeTa11}.  
Among the key contributions of their work are the construction of the main function spaces
and the completion of the linear estimate relating them, which includes an important maximal function
estimate.
A significant observation of \cite{BeIoKeTa11}, emphasized in their work,
is that it is important that these spaces take into account a local smoothing effect;
\cite{BeIoKeTa11} crucially uses this effect to help bring under control the worst term
of the nonlinearity.
Another novelty of \cite{BeIoKeTa11} is its implementation of 
the caloric gauge, which
was first introduced by Tao \cite{Trenorm} and subsequently
recommended by him for use
in studying Schr\"odinger maps \cite{TSchroedinger}. 
As the caloric gauge is defined using harmonic map heat flow, it can be thought of as
an intrinsic and
nonlinear analogue of classical Littlewood-Paley theory.  
In \cite{BeIoKeTa11}, both the intrinsic caloric gauge
and the extrinsic (and modern) Littlewood-Paley theory are used simultaneously.

Our main result extends Theorem \ref{BIKT1.1}.
\begin{thm}
Let $T > 0$ and $Q \in S^2$.
Let $\varepsilon_0 > 0$
and let $\varphi \in H_Q^{\infty, \infty}(T)$
be a solution of the Schr\"odinger map system (\ref{SM}) whose initial data $\varphi_0$
has energy $E_0 := E(\varphi_0) < E_{\mathrm{crit}}$ and satisfies the energy dispersion condition
\begin{equation}
\sup_{k \in \mathbf{Z}} 
\lVert P_k \partial_x \varphi_0 \rVert_{L^2_x}
\leq \varepsilon_0.
\label{EDcondition}
\end{equation}
Let $I \supset (-T, T)$ denote the maximal time interval for which
there exists a smooth (necessarily unique) extension of $\varphi$ satisfying (\ref{SM}).
Suppose a priori that
\begin{equation}
\sum_{k \in \mathbf{Z}} \lVert P_k \partial_x \varphi \rVert_{L^4_{t,x}(I \times \rr)}^2 
\leq \varepsilon_0^2.
\label{Main L4 Cal0}
\end{equation}
Then, for $\varepsilon_0$ sufficiently small,
\begin{equation}
\sup_{t \in (-T, T)} \lVert \varphi(t) \rVert_{H^\sigma_Q}
\lesssim_{\sigma, T, \lVert \varphi_0 \rVert_{H_Q^\sigma}} 1,
\label{SoboBound}
\end{equation}
for all $\sigma \in \mathbf{Z}_+$.
Additionally, $I = \mathbf{R}$, so that,
in particular, $\varphi$ admits a unique smooth global extension
$\varphi \in C( \mathbf{R} \to H^\infty_Q )$.
Moreover, for any $\sigma_1 \in \mathbf{Z}_+$, there exists a positive
$\varepsilon_1 = \varepsilon_1(\sigma_1) \leq \varepsilon_0$ such that
\begin{equation}
\lVert \varphi \rVert_{L^\infty_t {H^\sigma_Q}(\mathbf{R} \times \rr)}
\lesssim_\sigma
\lVert \varphi \rVert_{H^\sigma_Q(\rr)}
\label{SobolevBound}
\end{equation}
holds for all $0 \leq \sigma \leq \sigma_1$
provided
(\ref{EDcondition}) and (\ref{Main L4 Cal0}) hold with $\varepsilon_1$ in place of $\varepsilon_0$.
\label{MainTheorem}
\end{thm}
Note that the energy dispersion condition (\ref{EDcondition}) holds
automatically in the case of small energy.
In such case, our proofs may be modified
(essentially by collapsing
to or reverting to the arguments of \cite{BeIoKeTa11})
so that the a priori $L^4$ bound is not required.  
Such an $L^4$ bound, however, can then be seen to hold a posteriori.

Using time divisibility of the $L^4$ norm, we can replace
(\ref{Main L4 Cal0}) with
\begin{equation*}
\sum_{k \in \mathbf{Z}} \lVert P_k \partial_x \varphi \rVert_{L^4_{t,x}(I \times \rr)}^2 
\leq K
\end{equation*}
for any $K > 0$
provided we allow the threshold for $\varepsilon_0$ 
and the implicit constant in (\ref{SobolevBound}) to depend upon $K > 0$.
We work with (\ref{Main L4 Cal0}) as stated so as to avoid the additional technicalities
that would arise otherwise.

We now turn to a very rough sketch of the proof of Theorem \ref{MainTheorem};
for a detailed outline, see \S \ref{S:Main}.

\textbf{Basic setup and gauge selection.}

It suffices to prove homogeneous Sobolev variants of (\ref{SoboBound})
and (\ref{SobolevBound}) over a suitable range.  Thanks to mass and energy
conservation, we need only consider $\sigma > 1$.  For $\sigma \geq 1$, controlling
$\lVert \varphi(t) \rVert_{\dot{H}^\sigma}$ is equivalent to controlling
$\lVert \partial_x \varphi(t) \rVert_{\dot{H}^{\sigma - 1}}$.  We therefore consider the time
evolution of $\partial_x \varphi$, which may be written entirely in terms of
derivatives of the map $\varphi$.  A more intrinsic way of expressing these equations
is to select a \emph{gauge} rather than an extrinsic embedding and coordinate system.
We employ the caloric gauge, which is geometrically natural and is analytically well-suited
for studying
Schr\"odinger maps.  See \cite{Sm09} for the complete details of the construction.
It turns out that Sobolev bounds for the gauged derivative map imply corresponding
Sobolev bounds for the ungauged derivative map.  We schematically write
the gauged equation as
\begin{equation*}
(\partial_t - \Delta) \psi = \mathcal{N},
\label{Schem1}
\end{equation*}
where $\psi$ is $\partial_x \varphi$ placed in the caloric gauge and $\mathcal{N}$ is a nonlinearity
constructed in part from $\psi$ and $\partial_x \psi$.

\textbf{Function spaces and their interrelation.}

To prove global results in the energy-critical setting, 
we of course must look for bounds other than energy estimates
to control the solution.  
Local smoothing estimates and Strichartz estimates 
will be among the most important required.
Our goal is to prove control over $\psi$ within a suitable space through
the use of a bootstrap argument.
A standard setup requires a space, say $G$, for the functions $\psi$
and a space, say $N$, for the nonlinearity $\mathcal{N}$.  In fact, we work with stronger,
frequency-localized spaces, $G_k$ and $N_k$, to respectively hold $P_k \psi$
and $P_k \mathcal{N}$.  We want them to be related at least by the linear estimate
\begin{equation*}
\lVert P_k \psi \rVert_{G_k} \lesssim \lVert P_k \psi(t = 0) \rVert_{L^2_x} + \lVert P_k \mathcal{N} \rVert_{N_k}.
\end{equation*}
The hope, then, is to control $\lVert P_k \mathcal{N} \rVert_{N_k}$ in terms of
$\lVert P_k \psi(t = 0) \rVert_{L^2_x}$ and 
$\varepsilon \lVert P_k \psi \rVert_{G_k}$ (with $\varepsilon$ small),
so that, by proving (under a bootstrap hypothesis) a statement such as
\begin{equation*}
\lVert P_k \psi \rVert_{G_k} \lesssim \lVert P_k \psi(t = 0) \rVert_{L^2_x} + \varepsilon \lVert P_k \psi \rVert_{G_k},
\end{equation*}
we may conclude
\begin{equation}
\lVert P_k \psi \rVert_{G_k} \lesssim \lVert P_k \psi(t = 0) \rVert_{L^2_x}.
\label{SchemGoal}
\end{equation}
Once (\ref{SchemGoal}) is proved, showing
(\ref{SoboBound}) and (\ref{SobolevBound}) is reduced to the comparatively easy tasks
of unwinding the gauging and frequency localization steps so as to conclude
with a standard continuity argument.

\textbf{Controlling the nonlinearity.}

In this context, the main contribution of this paper lies in showing
that we may conclude (\ref{SchemGoal}) without assuming small energy.
The most difficult-to-control terms in the nonlinearity
$P_k \mathcal{N}$ are those involving a derivative landing on high-frequency pieces of
the derivative fields;
we represent them schematically as $A_{\mathrm{lo}} \partial \psi_{\mathrm{hi}}$.
Local smoothing estimates controlling the linear evolution (introduced in \cite{IoKe06, IoKe07-2})
were successfully used in \cite{BeIoKeTa11} to handle $A_{\mathrm{lo}} \partial_x \psi_{\mathrm{hi}}$.
These are not strong enough to control $A_{\mathrm{lo}} \partial_x \psi_{\mathrm{hi}}$
in the subthreshold energy setting.  We instead pursue a more covariant approach, working
directly with a certain covariant frequency-localized
Schr\"odinger equation (see \S \ref{S:Local smoothing}).
Our approach is also physical-space based,
in the vein of \cite{PlVe09, PlVe, Tls},
and modular.

%% file: GaugeFieldEquations.tex
\section{Gauge field equations} \label{S:Gauge Field Equations}

In \S \ref{SS:Derivative equations}
we pass to the derivative formulation of the Schr\"odinger map system (\ref{SM}).
All of the main arguments of our subsequent analysis take place at this level.
The derivative formulation is at once both overdetermined, reflecting geometric constraints,
and underdetermined, exhibiting \emph{gauge invariance}.
\S \ref{SS:Introduction to the caloric gauge} 
introduces the caloric gauge, 
which is the gauge
we select and work with throughout.  Both \cite{TSchroedinger} and \cite{BeIoKeTa11}
give good explanations justifying the use of the caloric gauge 
in our setting as opposed to alternative gauges.  
The reader is referred to \cite{Sm09} for the requisite construction of the caloric gauge
for maps with energy up to $E_{\mathrm{crit}}$.
\S \ref{SS:Frequency localization}
deals with frequency localizing components of the caloric gauge.
Proofs are postponed to \S \ref{S:The caloric gauge} so that we can more quickly turn our
attention to the gauged Schr\"odinger map system.

\subsection{Derivative equations}  \label{SS:Derivative equations}
We begin with some
constructions valid for any smooth function $\phi : \RR \times (-T, T) \to \mathbf{S}^2$.
For a more general and extensive introduction to the gauge formalism we now introduce,
see \cite{Trenorm}.
Space and time derivatives of $\phi$ are denoted by
 $\partial_\alpha \phi(x, t)$, where $\alpha = 1, 2, 3$ ranges
over the spatial variables $x_1, x_2$ and time $t$ with $\partial_3 = \partial_t$.

Select a (smooth) orthonormal frame
$(v(x, t), w(x, t))$ for $T_{\phi(x, t)} \mathbf{S}^2$, i.e.,
smooth functions $v, w: \RR \times (-T, T) \to T_{\phi(x, t)} \mathbf{S}^2$
such that at each point $(x, t)$ in the domain the vectors
$v(x, t), w(x, t)$ form an orthonormal basis for $T_{\phi(x, t)} \mathbf{S}^2$.
As a matter of convention we assume that $v$ and $w$ are chosen so that
$v \times w = \phi$.

With respect to this chosen frame we
then introduce the derivative fields $\psi_\alpha$, setting
\begin{equation}
\psi_\alpha := v \cdot \partial_\alpha \phi + i w \cdot \partial_\alpha \phi.
\label{Derivative Field}
\end{equation}
Then $\partial_\alpha \phi$ admits the representation
\begin{equation}
\partial_\alpha \phi = v \; \Re(\psi_\alpha) + w \; \Im(\psi_\alpha)
\label{Phi Frame Rep}
\end{equation}
with respect to the frame $(v, w)$.
The derivative fields can be thought of as arising from the following process:
First, rewrite the vector $\partial_\alpha \phi$ with respect to the orthonormal
basis $(v, w)$; then, identify $\RR$ with the complex numbers $\mathbf{C}$
according to $v \leftrightarrow 1$, $w \leftrightarrow i$.  
Note that this identification
respects the complex structure
of the target manifold.

Through this identification
the Riemannian connection on $\mathbf{S}^2$ pulls back to
a covariant derivative on $\mathbf{C}$, which we denote by
\begin{equation*}
D_\alpha := \partial_\alpha + i A_\alpha.
\end{equation*}
The real-valued connection coefficients $A_\alpha$
are defined via
\begin{equation}
A_\alpha := w \cdot \partial_\alpha v,
\label{Connection Coefficient}
\end{equation}
so that in particular
\begin{align*}
\partial_\alpha v &= - \phi \; \Re(\psi_\alpha) + w A_\alpha \\
\partial_\alpha w &= - \phi \; \Im(\psi_\alpha) - v A_\alpha.
\end{align*}

Due to the fact that the Riemannian connection on $\mathbf{S}^2$ is torsion-free,
the derivative fields satisfy the relations
\begin{equation}
D_\beta \psi_\alpha = D_\alpha \psi_\beta.
\label{Torsion Free}
\end{equation}
A straightforward calculation (which uses the fact that the sphere
has constant curvature $+1$) shows
\begin{equation*}
\partial_\beta A_\alpha - \partial_\alpha A_\beta =
\Im( \psi_\beta \overline{\psi_\alpha}) =: q_{\beta \alpha}.
\end{equation*}
The curvature of the connection is therefore given by
\begin{equation}
[D_\beta, D_\alpha] :=
D_\beta D_\alpha - D_\alpha D_\beta = i q_{\beta \alpha}.
\label{Curvature}
\end{equation}

Assuming now that we are given a smooth solution $\varphi$ of the Schr\"odinger map
system (\ref{SM}), we derive the equations satisfied by the derivative
fields $\psi_\alpha$.
The system (\ref{SM}) directly translates to
\begin{equation}
\psi_t = i D_\ell \psi_\ell
\label{Psi_t Equation}
\end{equation}
because
\begin{equation*}
\varphi \times \Delta \varphi = J(\varphi) (\varphi^* \nabla)_j \partial_j \varphi,
\end{equation*}
where $J(\varphi)$ denotes the complex structure $\varphi \times$
and $(\varphi^* \nabla)_j$ the pullback of the Levi-Civita connection
$\nabla$ on the sphere.

Let us pause to note the following conventions regarding indices.  
Roman typeface letters are used
to index spatial variables.  Greek typeface letters are used to index the spatial
variables along with time.  Repeated lettered indices within the same subscript or occurring
in juxtaposed terms
indicate an implicit summation over the
appropriate set of indices.

Using (\ref{Torsion Free}) and (\ref{Curvature}) in (\ref{Psi_t Equation}) yields
\begin{equation*}
D_t \psi_m = i D_\ell D_\ell \psi_m + q_{\ell m} \psi_\ell,
\end{equation*}
which is equivalent to the nonlinear Schr\"odinger equation
\begin{equation}
(i \partial_t + \Delta) \psi_m =
\cN_m,
\label{NLS}
\end{equation}
where
the nonlinearity $\cN_m$ is defined by the formula
\begin{equation*}
\cN_m 
:= 
-i A_\ell \partial_\ell \psi_m - i \partial_\ell (A_\ell \psi_m)
+ (A_t + A_x^2) \psi_m - i \psi_\ell \Im(\overline{\psi_\ell} \psi_m).
\end{equation*}
We split this nonlinearity as a sum $\cN_m = B_m + V_m$ with $B_m$ and $V_m$ defined by
\begin{equation}
B_m 
:= - i \partial_\ell (A_\ell \psi_m) - i A_\ell \partial_\ell \psi_m \label{B Def} 
\end{equation}
and
\begin{equation}
V_m := (A_t + A_x^2) \psi_m - i \psi_\ell \Im(\overline{\psi_\ell} \psi_m),
\label{V Def}
\end{equation}
thus separating 
the essentially semilinear magnetic potential terms
and the essentially semilinear electric potential terms 
from each other.

We now state the gauge formulation of the differentiated Schr\"odinger map system:
\begin{equation}
\left\{
\begin{array}{ll}
D_t \psi_m &= i D_\ell D_\ell \psi_m + \Im(\psi_\ell \overline{\psi_m}) \psi_\ell \\
D_\alpha \psi_\beta &= D_\beta \psi_\alpha \\
\Im(\psi_\alpha \overline{\psi_\beta}) &= \partial_\alpha A_\beta - \partial_\beta A_\alpha.
\end{array}
\right.
\label{DSM System}
\end{equation}
A solution $\psi_m$ to (\ref{DSM System}) cannot be determined uniquely without
first choosing an orthonormal frame $(v, w)$.  Changing a given choice
of orthonormal frame induces a gauge transformation and
may be represented as
\begin{equation*}
\psi_m \to e^{-i \theta} \psi_m,
\quad\quad
A_m \to A_m + \partial_m \theta
\end{equation*}
in terms of the gauge components.  The system (\ref{DSM System}) is invariant
with respect to such gauge transformations.

The advantage of working with this gauge formalism rather than the
Schr\"odinger map system or the derivative equations directly is that a 
carefully selected choice of gauge tames the nonlinearity.
In particular, when the caloric gauge is employed, the nonlinearity in (\ref{NLS})  
is nearly perturbative.

\subsection{Introduction to the caloric gauge} \label{SS:Introduction to the caloric gauge}

In this section we introduce the caloric gauge, which is the gauge we shall
employ throughout the remainder of the paper.  Gauges were first used to
study (\ref{SM}) in \cite{ChShUh00}.
We note here that the while the Coulomb gauge
would seem an attractive choice,
it turns out that this
gauge is not well-suited to the study of Schr\"odinger maps in low dimension,
as in low dimension parallel interactions of waves are more probable than in high dimension,
resulting in unfavorable $\text{high} \times \text{high} \to \text{low}$ cascades.
See \cite{TSchroedinger} and \cite{BeIoKeTa11} for further discussion and a comparison
of the Coulomb and caloric gauges.  Also see \cite[Chapter 6]{Tdis}
for a discussion of various gauges that have been used in the study of wave maps.

The caloric gauge
was introduced by Tao in \cite{Trenorm} in the setting of wave maps
into hyperbolic space.
In a series of unpublished papers \cite{T3, T4, T5, T6, T7},
Tao used this gauge in establishing
global regularity of wave maps
into hyperbolic space.  
In his unpublished note \cite{TSchroedinger},
Tao also suggested the caloric gauge
as a suitable gauge for the study of Schr\"odinger maps.
The caloric gauge was first used in the Schr\"odinger maps problem
by Bejenaru, Ionescu, Kenig, and Tataru in \cite{BeIoKeTa11} to establish global well-posedness
in the setting of initial data with sufficiently small critical norm.
We recommend \cite{Trenorm, T4, TSchroedinger, BeIoKeTa11} for background
on the caloric gauge and for helpful heuristics.

\begin{thm}[The caloric gauge]\label{T:CG}
Let $T \in (0, \infty)$, $Q \in S^2$, and let $\phi(x, t) \in H_{Q}^{\infty, \infty}(T)$
be such that $\sup_{t \in (-T, T)} E(\phi(t)) < E_{\mathrm{crit}}$.
Then
there exists a unique smooth extension 
$\phi(s, x, t) \in C([0, \infty) \to H_Q^{\infty, \infty}(T))$
solving the covariant heat equation
\begin{equation}
\partial_s \phi = \Delta \phi + \phi \cdot \lvert \partial_x \phi \rvert^2
\label{Covariant Heat}
\end{equation}
and with $\phi(0, x, t) = \phi(x, t)$.  Moreover, 
for any given choice of a (constant) orthonormal basis $(v_\infty, w_\infty)$
of $T_Q \mathbf{S}^2$,
there exist smooth functions
$v, w: [0, \infty) \times \RR \times (-T, T) \to S^2$ such that at each point $(s, x, t),$ the set
$\{v, w, \phi\}$ naturally forms an orthonormal basis for $\mathbf{R}^3$, the
gauge condition
\begin{equation}
w \cdot \partial_s v \equiv 0,
\label{Gauge Condition}
\end{equation}
is satisfied, and
\begin{equation}
\lvert \partial_x^\rho f(s) \rvert \lesssim_\rho \langle s \rangle^{-(\lvert \rho \rvert + 1)/2}
\label{Soft Decay Bounds}
\end{equation}
for each $f \in \{ \phi - Q, v - v_\infty, w - w_\infty \}$, multiindex $\rho$, and $s \geq 0$.
\end{thm}
\begin{proof}
This is a special case of the more general result \cite[Theorem 7.6]{Sm09}.
Whereas in \cite{Sm09} everything is stated in terms of the
category of Schwartz functions, in fact this requirement may be relaxed
to $H_Q^{\infty, \infty}(T)$ without difficulty (at least
in the case of compact target manifolds) since weighted decay
in $L^2$-based Sobolev spaces is not used in any proofs.
\end{proof}

In our application in this paper, $E(\varphi(t))$ is conserved. Therefore, here and elsewhere,
we set $E_0 := E(\varphi_0)$.

Having extended $v, w$ along the heat flow, we may likewise extend $A_x$ along the flow.
We record here for reference a technical bound from \cite{Sm09} that proves useful.
\begin{thm}
Assume the conditions of Theorem \ref{T:CG} are in force. Then
the following bound holds:
\begin{align}
\lVert A_x(s) \rVert_{L^2_x(\mathbf{R}^2)}
&\lesssim_{E_0} 1.
\label{A_x Bound}
\end{align}
\end{thm}
\begin{proof}
See \cite[\S \S7, 7.1]{Sm09}.
\end{proof}

\begin{cor}[Energy bounds for the frame]
Suppose that $\varphi$ is a Schr\"odinger map with energy $E_0 < \Ec$.
Then it holds that
\begin{equation}
\lVert \partial_x v \rVert_{L_t^\infty L_x^2} \lesssim_{E_0} 1.
\label{dv Bound}
\end{equation}
\end{cor}
\begin{proof}
Because $\lvert v \rvert \equiv 1$, it holds that $v \cdot \partial_m v \equiv 0$.
Therefore,
with respect to the orthonormal frame $(v, w, \varphi)$, 
the vector $\partial_m v$ admits the representation
\begin{equation}
\partial_m v = A_m \cdot w - \Re(\psi_m) \cdot \varphi.
\label{dv Representation}
\end{equation}
The bound (\ref{dv Bound}) then follows from
using
$\lvert w \rvert \equiv 1 \equiv \lvert \varphi \rvert$, $\lVert \psi_m \rVert_{L_x^2}
\equiv \lVert \partial_m \varphi \rVert_{L_x^2}$, energy conservation, and (\ref{A_x Bound})
all in (\ref{dv Representation}).
\end{proof}

Adopting the convention $\partial_0 = \partial_s$, 
and now and hereafter allowing all Greek indices to range
over heat time, spatial variables, and time, we define for all 
$(s, x, t) \in [0, \infty) \times \RR \times (-T, T)$ the various gauge components
\begin{align*}
\psi_\alpha &:= v \cdot \partial_\alpha \varphi + i w \cdot \partial_\alpha \varphi \\
A_\alpha &:= w \cdot \partial_\alpha v \\
D_\alpha &:= \partial_\alpha + A_\alpha \\
q_{\alpha \beta} &:= \partial_\alpha A_\beta - \partial_\beta A_\alpha.
\end{align*}
For $\alpha = 0, 1, 2, 3$ it holds that
\begin{equation*}
\partial_\alpha \varphi = v \; \Re(\psi_\alpha) + w \; \Im(\psi_\alpha).
\end{equation*}
The parallel transport condition $w \cdot \partial_s v \equiv 0$ is
equivalently expressed in terms of the connection coefficients as
\begin{equation}
A_s \equiv 0.
\label{A_s}
\end{equation}
Expressed in terms of the gauge, the heat flow (\ref{Covariant Heat})
lifts to
\begin{equation}
\psi_s = D_\ell \psi_\ell.
\label{Psi_s Equation}
\end{equation}
Using (\ref{Torsion Free}) and (\ref{Curvature}), we may rewrite 
the $D_m$ covariant derivative of (\ref{Psi_s Equation}) as
\begin{equation*}
\partial_s \psi_m = D_\ell D_\ell \psi_m + i \Im(\psi_m \overline{\psi_\ell}) \psi_\ell,
\end{equation*}
or equivalently
\begin{equation}
(\partial_s - \Delta) \psi_m =
i A_\ell \partial_\ell \psi_m +
i \partial_\ell (A_\ell \psi_m) -
A_x^2 \psi_m + i \psi_\ell \Im(\overline{\psi_\ell} \psi_m).
\label{HF}
\end{equation}
More generally, taking the $D_\alpha$ covariant derivative, we obtain
\begin{equation}
(\partial_s - \Delta) \psi_\alpha =
U_\alpha,
\label{genPsiEQ}
\end{equation}
where we set
\begin{equation}
U_\alpha 
:=
i A_\ell \partial_\ell \psi_\alpha + i \partial_\ell(A_\ell \psi_\alpha) 
- A_x^2 \psi_\alpha + i \psi_\ell \Im(\overline{\psi_\ell} \psi_\alpha),
\label{Heat Nonlinearity 0}
\end{equation}
which admits the alternative representation
\begin{equation}\label{Heat Nonlinearity}
U_\alpha 
=
2i A_\ell \partial_\ell \psi_\alpha + i (\partial_\ell A_\ell) \psi_\alpha
- A_x^2 \psi_\alpha + i \psi_\ell \Im(\overline{\psi_\ell} \psi_\alpha).
\end{equation}
From (\ref{Curvature}) and (\ref{A_s}) it follows that
\begin{equation*}
\partial_s A_\alpha = \Im(\psi_s \overline{\psi_\alpha}).
\end{equation*}
Integrating back from $s = \infty$ (justified using (\ref{Soft Decay Bounds})) yields
\begin{equation}
A_\alpha(s) = - \int_s^\infty \Im(\overline{\psi_\alpha} \psi_s)(s^\prime) \ds^\prime.
\label{CC Integral Rep}
\end{equation}

At $s = 0$, $\varphi$ satisfies both (\ref{SM}) and (\ref{Covariant Heat}),
or equivalently, $\psi_t(s = 0) = i \psi_s(s = 0)$. While for $s > 0$
it continues to be the case that $\psi_s = D_\ell \psi_\ell$ by construction,
we no longer necessarily have $\psi_t(s) = i D_\ell(s) \psi_\ell(s)$, i.e.,
$\varphi(s, x, t)$ is not necessarily a Schr\"odinger map at fixed $s > 0$. 
In the following lemma we derive an evolution equation for the commutator
$\Psi = \psi_t - i \psi_s$.
\begin{lem}[Flows do not commute]
Set $\Psi := \psi_t - i \psi_s$. Then
\begin{align}
\partial_s \Psi
&= D_\ell D_\ell \Psi + i \Im(\psi_t \bar{\psi}_\ell) \psi_\ell - \Im(\psi_s \bar{\psi}_\ell) \psi_\ell 
\label{commutator1}\\
&= D_\ell D_\ell \Psi + i \Im(\Psi \bar{\psi}_\ell) \psi_\ell + i \Im(i\psi_s \bar{\psi}_\ell) \psi_\ell
- \Im(\psi_s \bar{\psi}_\ell) \psi_\ell.
\label{commutator2}
\end{align}
\label{L:Commutator}
\end{lem}
\begin{proof}
We prove (\ref{commutator1}), since (\ref{commutator2}) is a trivial
consequence of it.

Applying (\ref{HF}) and (\ref{genPsiEQ}) to $\psi_s$
and $\psi_t$ and collapsing the covariant derivative terms
yields
\begin{align}
\partial_s \psi_t &= D_\ell D_\ell \psi_t + i \Im(\psi_t \overline{\psi_\ell}) \psi_\ell
\label{psit evolv} \\
\partial_s \psi_s &= D_\ell D_\ell \psi_s + i \Im(\psi_s \overline{\psi_\ell}) \psi_\ell.
\label{psis evolv}
\end{align}
Multiply (\ref{psis evolv}) by $i$ to obtain the $s$-evolution
of $i \psi_s$. Multiplication by $i$ commutes with $D_\ell$,
but fails to do so with $\Im(\cdot)$, and thus we obtain
\begin{equation}
\partial_s i \psi_m
= D_\ell D_\ell i \psi_s - \Im(\psi_s \overline{\psi_\ell}) \psi_\ell.
\label{psiis evolv}
\end{equation}
Together (\ref{psit evolv}) and (\ref{psiis evolv}) imply (\ref{commutator1}).
\end{proof}


\subsection{Frequency localization} \label{SS:Frequency localization}
Frequency localization plays an indispensable role in our analysis.  
In this subsection we establish some basic concepts and then state
some basic results for the caloric gauge.

Our notation for a standard
Littlewood-Paley frequency localization of a function $f$ to frequencies $\sim 2^k$ is $P_k f$
and to frequencies $\lesssim 2^k$ is $P_{\leq k} f$.
The particular localization chosen is of course immaterial to our analysis, 
but for definiteness is specified in the next section and chosen for convenience to
coincide with that in \cite{BeIoKeTa11}.

We shall frequently make use of the following standard \emph{Bernstein inequalities} for
$\rr$ with $\sigma \geq 0$ and $1 \leq p \leq q \leq \infty$:
\begin{align*}
\lVert P_{\leq k} \lvert \partial_x \rvert^\sigma f \rVert_{L_x^p(\rr)}
&\lesssim_{p, \sigma}
2^{\sigma k} \lVert P_{\leq k} f \rVert_{L_x^p(\rr)} \\
\lVert P_k \lvert \partial_x \rvert^{\pm \sigma} f \rVert_{L_x^p(\rr)}
&\lesssim_{p, \sigma}
2^{\pm \sigma k} \lVert P_k f \rVert_{L_x^p(\rr)} \\
\lVert P_{\leq k} f \rVert_{L_x^q(\rr)}
&\lesssim_{p, q}
2^{2k(1/p - 1/q)} \lVert P_{\leq k} f \rVert_{L_x^p(\rr)} \\
\lVert P_k f \rVert_{L_x^q(\rr)}
&\lesssim_{p, q}
2^{2k(1/p - 1/q)} \lVert P_k f \rVert_{L_x^p(\rr)}.
\end{align*}

A particularly important notion for us is that of a frequency envelope, as it provides
a way to rigorously manage the ``frequency leakage'' phenomenon and
the frequency cascades produced by nonlinear interactions.
We introduce a parameter $\delta$ in the definition; for the purposes of this paper
$\delta = \frac{1}{40}$ suffices.
\begin{defin}[Frequency envelopes]
A positive sequence $\{a_k\}_{k \in \mathbf{Z}}$ is a frequency envelope if it belongs to $\ell^2$
and is slowly varying:
\begin{equation}
a_k \leq a_j 2^{\delta \lvert k - j \rvert},
\quad \quad j, k \in \mathbf{Z}.
\label{Slowly Varying}
\end{equation}
A frequency envelope $\{a_k\}_{k \in \mathbf{Z}}$ is $\varepsilon$-energy dispersed if it satisfies
the additional condition
\begin{equation*}
\sup_{k \in \mathbf{Z}} a_k \leq \varepsilon.
\end{equation*}
\end{defin}
Note in particular that frequency envelopes satisfy the following summation rules:
\begin{align}
\sum_{k^\prime \leq k} 2^{p k^\prime} a_{k^\prime}
&\lesssim (p - \delta)^{-1} 2^{p k} a_k
&p > \delta \label{Sum 1}\\
\sum_{k^\prime \geq k} 2^{-p k^\prime} a_{k^\prime}
&\lesssim (p - \delta)^{-1} 2^{-p k} a_k
&p > \delta \label{Sum 2}.
\end{align}
In practice we work with $p$ bounded away from $\delta$, e.g., $p > 2 \delta$ suffices,
and iterate these inequalities only $O(1)$ times.
Therefore, in applications we drop the factors $(p - \delta)^{-1}$ appearing in
(\ref{Sum 1}) and (\ref{Sum 2}).

Finally, pick a positive integer $\sigma_1$ and hold it fixed throughout the remainder of this section.
Results in this section hold for any such $\sigma_1$, though implicit constants
are allowed to depend upon this choice.

Given initial data $\varphi_0 \in H_Q^\infty$, define
for all $\sigma \geq 0$ and $k \in \mathbf{Z}$
\begin{equation}
c_k(\sigma) 
:= 
\sup_{k^\prime \in \mathbf{Z}} 
2^{-\delta \lvert k - k^\prime \rvert} 2^{\sigma k^\prime}
\lVert P_{k^\prime} \partial_x \varphi_0 \rVert_{L_x^2}.
\label{c Envelope}
\end{equation}
Set $c_k := c_k(0)$ for short.
For $\sigma \in [0, \sigma_1]$ it then holds that
\begin{equation}
\lVert \partial_x \varphi_0 \rVert_{\dot{H}_x^\sigma}^2 \sim
\sum_{k \in \mathbf{Z}} c_k^2(\sigma)
\quad\quad \text{and} \quad\quad
\lVert P_k \partial_x \varphi_0 \rVert_{L_x^2} \leq c_k(\sigma) 2^{-\sigma k}.
\label{c cons}
\end{equation}
Similarly, for $\varphi \in H_Q^{\infty, \infty}(T)$, define for all
$\sigma \geq 0$ and $k \in \mathbf{Z}$
\begin{equation}
\gamma_k(\sigma) := \sup_{k^\prime \in \mathbf{Z}} 2^{-\delta \lvert k - k^\prime \rvert}
2^{\sigma k^\prime} \lVert P_{k^\prime} \varphi \rVert_{L_t^\infty L_x^2}.
\label{gamma Envelope}
\end{equation}
Set $\gamma_k := \gamma_k(1)$.

\begin{thm}[Frequency-localized energy bounds for heat flow]
Let $f \in \{\varphi, v, w\}$.  Then for $\sigma \in [1, \sigma_1]$ the bound
\begin{equation}
\lVert P_k f(s) \rVert_{L_t^\infty L_x^2}
\lesssim
2^{-\sigma k} \gamma_k(\sigma) (1 + s 2^{2k})^{-20}
\label{Hard Envelope Bounds}
\end{equation}
holds
and for any $\sigma, \rho \in \mathbf{Z}_+$ it holds that
\begin{equation}
\sup_{k \in \mathbf{Z}} \sup_{s \in [0, \infty)}
(1 + s)^{\sigma / 2} 2^{\sigma k} \lVert P_k \partial_t^\rho f(s) \rVert_{L_t^\infty L_x^2} < \infty.
\label{Soft Envelope Bounds}
\end{equation}
\label{Energy Envelope Decay}
\end{thm}

\begin{cor}[Frequency-localized energy bounds for the caloric gauge]
For $\sigma \in [0, \sigma_1 - 1]$, it holds that
\begin{equation}
\lVert P_k \psi_x(s) \rVert_{L_t^\infty L_x^2} +
\lVert P_k A_m(s) \rVert_{L_t^\infty L_x^2}
\lesssim
2^k 2^{-\sigma k} \gamma_k(\sigma) (1 + s 2^{2k})^{-20}.
\label{Hard Field Bounds}
\end{equation}
Moreover, for any $\sigma \in \mathbf{Z}_+$,
\begin{equation}
\sup_{k \in \mathbf{Z}} \sup_{s \in [0, \infty)} (1 + s)^{\sigma / 2}
2^{\sigma k} 2^{-k} \left(
\lVert P_k (\partial_t^\rho \psi_x(s)) \rVert_{L_t^\infty L_x^2} +
\lVert P_k (\partial_t^\rho A_x(s)) \rVert_{L_t^\infty L_x^2} \right)
< \infty
\label{Soft Field Space Bounds}
\end{equation}
and
\begin{equation}
\sup_{k \in \mathbf{Z}} \sup_{s \in [0, \infty)} (1 + s)^{\sigma / 2}
2^{\sigma k} \left(
\lVert P_k (\partial_t^\rho \psi_t(s)) \rVert_{L_t^\infty L_x^2} +
\lVert P_k (\partial_t^\rho A_t(s)) \rVert_{L_t^\infty L_x^2} \right)
< \infty.
\label{Soft Field Time Bounds}
\end{equation}
\label{EED Cor}
\end{cor}
We prove Theorem \ref{Energy Envelope Decay} and its corollary in \S \ref{S:The caloric gauge}.

Note that Corollary \ref{EED Cor} has as an elementary consequence the following:
\begin{cor}
For $\sigma \in [0, \sigma_1 - 1]$, it holds that
\begin{equation}
\lVert P_k \psi_x(0, \cdot, 0) \rVert_{L_x^2}
\lesssim
2^{-\sigma k} c_k(\sigma).
\end{equation}
\label{DataCor}
\end{cor}

%% file: FunctionSpaces.tex
\section{Function spaces and basic estimates} \label{S:Function spaces}

\subsection{Definitions}

\begin{defin}[Littlewood-Paley multipliers]
Let $\eta_0 : \mathbf{R} \to [0, 1]$ be a smooth even function vanishing outside the interval
$[-8/5, 8/5]$ and equal to $1$ on $[-5/4, 5/4]$.  For $j \in \mathbf{Z}$, set
\begin{equation*}
\chi_j(\cdot) = \eta_0( \cdot / 2^j ) - \eta_0( \cdot / 2^{j - 1}),
\quad
\chi_{\leq j}(\cdot) = \eta_0(\cdot / 2^j).
\end{equation*}
Let $P_k$ denote the operator on $L^\infty(\RR)$ defined by the Fourier multiplier
$\xi \to \chi_k(\lvert \xi \rvert)$.  For any interval $I \subset \mathbf{R}$, let
$\chi_I$  be the Fourier multiplier defined by $\chi_I = \sum_{j \in I \cap \mathbf{Z}} \chi_j$ and 
let $P_I$ denote its corresponding operator on $L^\infty(\RR)$.  We shall denote
$P_{(-\infty, k]}$ by $P_{\leq k}$ for short.  For $\theta \in \mathbf{S}^1$ and $k \in \mathbf{Z}$, we define
the operators $P_{k, \theta}$ by the Fourier multipliers $\xi \to \chi_k(\xi \cdot \theta)$.
\label{D:LP}
\end{defin}

Some frequency interactions in the nonlinearity of (\ref{NLS}) can be controlled
using the following Strichartz estimate:
\begin{lem}[Strichartz estimate]
Let $f \in L_x^2(\RR)$ and $k \in \mathbf{Z}$.  Then the Strichartz estimate
\begin{equation*}
\lVert e^{i t \Delta} f \rVert_{L_{t,x}^4} \lesssim \lVert f \rVert_{L_x^2}
\end{equation*}
holds, as does
the maximal function bound
\begin{equation*}
\lVert e^{i t \Delta} P_k f \rVert_{L_x^4 L_t^\infty} \lesssim
2^{k/2} \lVert f \rVert_{L_x^2}.
\end{equation*}
\end{lem}
The first bound is the original Strichartz estimate (see \cite{St77}) 
and the second follows from scaling.
These will be augmented with certain lateral Strichartz estimates to be introduced shortly.
Strichartz estimates alone are not sufficient for controlling the nonlinearity
in (\ref{NLS}).  The additional control required comes from local smoothing
and maximal function estimates.
Certain local smoothing
spaces localized to cubes were introduced in \cite{KePoVe93} to study the local wellposedness
of Schr\"odinger equations with general derivative nonlinearities.
Stronger spaces were introduced in \cite{IoKe07} to prove a low-regularity global result.
In the Schr\"odinger map setting, local smoothing spaces were first used in \cite{IoKe06}
and subsequently in \cite{IoKe07-2, BeIoKe07, Be08}.  The particular
local smoothing/maximal function spaces we shall use were introduced in \cite{BeIoKeTa11}.

For a unit length $\theta \in \mathbf{S}^1$, we denote by $H_\theta$ its orthogonal
complement in $\RR$ with the induced measure.
Define the lateral spaces $L_\theta^{p,q}$ as those consisting of all 
measurable $f$ for which the norm
\begin{equation*}
\lVert h \rVert_{L_\theta^{p,q}} =
\left( \int_{\mathbf{R}} \left(
\int_{H_\theta \times \mathbf{R}} \lvert h(x_1 \theta + x_2, t )
\rvert^q dx_2 dt \right)^{p/q} dx_1 \right)^{1/p},
\end{equation*}
is finite.
We make the usual modifications when $p = \infty$ or $q = \infty$.
The most important spaces for our analysis are the local smoothing space
$L_\theta^{\infty, 2}$ and the inhomogeneous local smoothing space $L_\theta^{1,2}$.
To move between these spaces we use the maximal function space $L_\theta^{2, \infty}$.

The following two estimates were shown in \cite{IoKe06} and \cite{IoKe07-2}:
\begin{lem}[Local smoothing]
Let $f \in L_x^2(\RR)$, $k \in \mathbf{Z}$, and $\theta \in \mathbf{S}^1$.  Then
\begin{equation*}
\lVert e^{i t \Delta} P_{k, \theta} f \rVert_{L_{\theta}^{\infty, 2}}
\lesssim
2^{-k/2} \lVert f \rVert_{L_x^2}.
\end{equation*}
For $f \in L_x^2(\mathbf{R}^d)$, 
the maximal function space bound
\begin{equation*}
\lVert e^{i t \Delta} P_k f \rVert_{L_\theta^{2, \infty}} \lesssim
2^{k(d-1)/2} \lVert f \rVert_{L_x^2}
\end{equation*}
holds for dimension $d \geq 3$.
\label{L:LS}
\end{lem}

In $d = 2$, the maximal function bound fails due to a logarithmic divergence.
In order to overcome this, we exploit Galilean invariance as in \cite{BeIoKeTa11}
(the idea goes back to \cite{Tat01} in the setting of wave maps).

For $p, q \in [1, \infty]$, $\theta \in \mathbf{S}^1$, $\lambda \in \mathbf{R}$,
define $L_{\theta, \lambda}^{p,q}$ using the norm
\begin{equation*}
\lVert h \rVert_{L_{\theta, \lambda}^{p,q}} =
\lVert T_{\lambda \theta}(h) \rVert_{L_\theta^{p,q}} =
\left( \int_{\mathbf{R}} \left(
\int_{H_\theta \times \mathbf{R}}
\lvert h((x_1 + t \lambda)\theta + x_2, t)\rvert^q dx_2 dt
\right)^{p/q} dx_1 \right)^{1/p},
\end{equation*}
where $T_w$ denotes the Galilean transformation
\begin{equation*}
T_w(f)(x,t) = e^{-i x \cdot w / 2} e^{-i t \lvert w \rvert^2 / 4} f(x + t w, t).
\end{equation*}
With $W \subset \mathbf{R}$ finite we define the spaces $L_{\theta, W}^{p,q}$ by
\begin{equation*}
L_{\theta, W}^{p,q} = \sum_{\lambda \in W} L_{\theta, \lambda}^{p,q}, \quad \quad
\lVert f \rVert_{L_{\theta, W}^{p,q}} =
\inf_{f = \sum_{\lambda \in W} f_\lambda} \sum_{\lambda \in W}
\lVert f_\lambda \rVert_{L_{\theta, \lambda}^{p,q}}.
\end{equation*}
For $k \in \mathbf{Z}$, $\mathcal{K} \in \mathbf{Z}_+$, set
\begin{equation*}
W_k := \{ \lambda \in [-2^k, 2^k] : 2^{k + 2 \mathcal{K}} \lambda \in \mathbf{Z} \}.
\end{equation*}
In our application we shall work on a finite time interval
$[- 2^{2 \mathcal{K}}, 2^{2 \mathcal{K}}]$ in order to ensure that the $W_k$
are finite.  This still suffices for proving global results so long as our effective bounds
are proved with constants independent of $T, \mathcal{K}$.
As discussed in \cite[\S 3]{BeIoKeTa11}, restricting $T$ to a finite time interval
avoids introducing additional technicalities.

\begin{lem}[Local smoothing/maximal function estimates]
Let $f \in L_x^2(\RR)$, $k \in \mathbf{Z}$, and $\theta \in \mathbf{S}^1$.  Then
\begin{equation*}
\lVert e^{i t \Delta} P_{k, \theta} f \rVert_{L_{\theta, \lambda}^{\infty, 2}}
\lesssim 2^{-k / 2} \lVert f \rVert_{L_{x}^2}, \quad \quad
\lvert \lambda \rvert \leq 2^{k - 40},
\end{equation*}
and moreover, if $T \in (0, 2^{2 \mathcal{K}}]$, then
\begin{equation*}
\lVert 1_{[-T, T]}(t) e^{i t \Delta} P_k f
\rVert_{L_{\theta, W_{k + 40}}^{2, \infty}} \lesssim
2^{k/2} \lVert f \rVert_{L_{x}^2}.
\end{equation*}
\label{LSMF}
\end{lem}
\begin{proof}
The first bound follows from Lemma \ref{L:LS} via a Galilean boost.
The second is more involved and proven in \cite[\S 7]{BeIoKeTa11}.
\end{proof}

\begin{lem}[Lateral Strichartz estimates]
Let $f \in L_x^2(\RR)$, $k \in \mathbf{Z}$, and $\theta \in \mathbf{S}^1$.
Let $ 2 < p \leq \infty, 2 \leq q \leq \infty$ and $1/p + 1/q = 1/2$. Then
\begin{align*}
\lVert e^{i t \Delta} P_{k, \theta} f \rVert_{L_\theta^{p, q}}
&\lesssim
2^{k(2/p - 1/2)} \lVert f \rVert_{L_x^2}, &p \geq q,
\\
\lVert e ^{i t \Delta} P_k f \rVert_{L_\theta^{p,q}}
&\lesssim_p
2^{k(2/p - 1/2)} \lVert f \rVert_{L_x^2}, &p\leq q.
\end{align*}
\label{Lateral}
\end{lem}
\begin{proof}
Informally speaking, these bounds follow from interpolating between the $L^4$ Strichartz
estimate and the local smoothing/maximal function estimates of Lemma \ref{LSMF}.
See \cite[Lemma 7.1]{BeIoKeTa11} for the rigorous argument.
\end{proof}

We now introduce the main function spaces.
Let $T > 0$.  For $k \in \mathbf{Z}$, let $I_k = \{ \xi \in \mathbf{R}^2 : \lvert \xi \rvert
\in [2^{k-1}, 2^{k+1}] \}$.  
Let
\begin{equation*}
L_k^2(T) := \{ f \in L^2(\mathbf{R}^2 \times [-T, T] ) : \mathrm{supp}\; \hat{f}(\xi, t) \subset I_k \times
[-T, T] \}.
\end{equation*}

For $f \in L^2 (\mathbf{R}^2 \times [-T, T] )$, let
\begin{equation*}
\begin{split}
\lVert f \rVert_{F_k^0(T)} :=& \;
\lVert f \rVert_{L_t^\infty L_x^2} +
\lVert f \rVert_{L_{t,x}^4}
+ 2^{-k/2} \lVert f \rVert_{L_x^4 L_t^\infty} \\
&+ 2^{-k/6} \sup_{\theta \in \mathbf{S}^1}
\lVert f \rVert_{L^{3,6}_\theta} + 
+ 2^{-k/2} \sup_{\theta \in \mathbf{S}^1}
\lVert f \rVert_{L_{\theta, W_{k + 40}}^{2, \infty}}.
\end{split}
\end{equation*}
We then define, similarly to as in \cite{BeIoKeTa11}, $F_k(T)$, $G_k(T)$, $N_k(T)$ 
as the normed spaces of functions in $L_k^2(T)$ for which the corresponding norms are finite:
\begin{align*}
\lVert f \rVert_{F_k(T)} :=& \inf_{J, m_1, \ldots, m_J \in \mathbf{Z}_+ }
\inf_{f = f_{m_1} + \cdots + f_{m_J}}
\sum_{j = 1}^J 2^{m_j} \lVert f_{m_j} \rVert_{F_{k + m_j}^0} \\
\lVert f \rVert_{G_k(T)} :=& \;\lVert f \rVert_{F_k^0(T)}
+ 2^{k / 6} \sup_{\lvert j - k \rvert \leq 20} \sup_{\theta \in \mathbf{S}^1}
\lVert P_{j, \theta} f \rVert_{L_\theta^{6,3}} \\
&+ 2^{k/2} \sup_{\lvert j - k \rvert \leq 20} \sup_{\theta \in \mathbf{S}^1}
\sup_{\lvert \lambda \rvert < 2^{k - 40}}
\lVert P_{j, \theta} f \rVert_{L_{\theta, \lambda}^{\infty, 2}} \\
\lVert f \rVert_{N_k(T)} :=& \inf_{f = f_1 + f_2 + f_3 + f_4 + f_5 + f_6}
\lVert f_1 \rVert_{L_{t,x}^{4/3}}
+ 2^{k/6} \lVert f_2 \rVert_{L_{\hat{\theta}_1}^{3/2, 6/5}}
+ 2^{k/6} \lVert f_3 \rVert_{L_{\hat{\theta}_2}^{3/2, 6/5}} \\
& + 2^{-k/6}
\lVert f_4 \rVert_{L^{6/5,3/2}_{\hat{\theta}_1}}
+ 2^{-k/6}
\lVert f_5 \rVert_{L^{6/5,3/2}_{\hat{\theta}_2}}
+ 2^{-k/2} \sup_{\theta \in \mathbf{S}^1} 
\lVert f_6 \rVert_{L_{\theta, W_{k - 40}}^{1,2}},
\end{align*}
where $(\hat{\theta}_1, \hat{\theta}_2)$ denotes the canonical basis in $\mathbf{R}^2$.

There are a few minor differences between these spaces and those appearing
in \cite{BeIoKeTa11}. The space $F^0_k$ now includes the 
lateral Strichartz space $L^{3,6}_\theta$,
whereas in \cite{BeIoKeTa11}, only $G_k$ was endowed with this norm. 
The net effect on the space $G_k$ is that it is left unchanged.
The space $F_k$, however, now explicitly incorporates this particular lateral Strichartz structure.
Note though, that for fixed $\theta \in \mathbf{S}^1$, we have by
enough applications of Young's and H\"older's inequalities that
\begin{align*}
2^{-k/6} \lVert f \rVert_{L^{3,6}_\theta}
&=
2^{-k/6} \left( \int_{\R} \left( \int_{H_\theta \times \R}
\lvert f(x_1 \theta + x_2, t) \rvert^6 dx_2 dt \right)^{1/2} dx_1 \right)^{1/3} \\
&\lesssim
2^{-k/6} \left( \int_{\R}
\lVert f \rVert_{L^4_{\theta, t}}^2 \lVert f \rVert_{L^\infty_{\theta, t}} dx_1 \right)^{1/3} \\
&\lesssim
2^{-k/6} \left( \int_{\R} \lVert f \rVert_{L^4_{\theta, t}}^4 dx_1 \right)^{1/6}
\left( \int_{\R} \lVert f \rVert_{L^\infty_{\theta, t}}^2 dx_1 \right)^{1/6} \\
&\lesssim
\lVert f \rVert_{L^4}^{2/3} \cdot 2^{-k/6} \lVert f \rVert_{L^{2, \infty}_\theta}^{1/3} \\
&\lesssim
\lVert f \rVert_{L^4} + 2^{-k/2} \lVert f \rVert_{L^{2,\infty}_\theta}.
\end{align*}
We also make one change to the $N_k$ space: We explicitly incorporate
$L^{6/5,3/2}_\theta$.

Incorporating these extra lateral Strichartz spaces affords us greater flexibility in
certain estimates: We can avoid having to use local smoothing/maximal function
spaces if we are willing to give up some decay. This tradeoff pays off in 
\S \ref{S:Local smoothing}, where as a consequence we can
prove a stronger
local smoothing estimate for a certain magnetic nonlinear Schr\"odinger 
equation in the one regime where this improvement is absolutely essential.

\begin{prop}[Main linear estimate]
Assume $\mathcal{K} \in \mathbf{Z}_+$, $T \in (0, 2^{2 \mathcal{K}}]$ and $k \in \mathbf{Z}$.  Then
for each $u_0 \in L^2$ that is frequency-localized to $I_k$ and for any $h \in N_k(T)$, the solution $u$ of
\begin{equation*}
(i \partial_t + \Delta_x) u = h, \quad\quad u(0) = u_0
\end{equation*}
satisfies
\begin{equation*}
\lVert u \rVert_{G_k(T)} \lesssim \lVert u(0) \rVert_{L_x^2} + \lVert h \rVert_{N_k(T)}.
\end{equation*}
\label{MainLinearEstimate}
\end{prop}
\begin{proof}
See \cite[Proposition 7.2]{BeIoKeTa11} for details.
Our changes to the spaces necessitate only minor changes in their proof,
as we must incorporate
$L^{6/5,3/2}_{\hat{\theta}_1}$ and $L^{6/5,3/2}_{\hat{\theta}_2}$
into the space $N_k^0(T)$.
\end{proof}

The spaces $G_k(T)$ are used to hold projections $P_k \psi_m$ of the derivative fields
$\psi_m$ satisfying (\ref{NLS}).  The main components of $G_k(T)$ are the
local smoothing/maximal function spaces
$L_{\theta, \lambda}^{\infty, 2}$,
$L_{\theta, W_{k + 40}}^{2, \infty}$,
and the lateral Strichartz spaces.  The local smoothing and
maximal function space components play 
an essential role in recovering the derivative loss that is due to the magnetic
nonlinearity.

The spaces $N_k(T)$ hold frequency projections of the nonlinearities in (\ref{NLS}).  
Here the main
spaces are the inhomogeneous local smoothing spaces $L_{\theta, W_{k - 40}}^{1,2}$
and the Strichartz spaces, both chosen to match those of $G_k(T)$.

The spaces $G_k(T)$ clearly embed in $F_k(T)$.  Two key properties enjoyed only by
the larger spaces $F_k(T)$ are
\begin{equation*}
\lVert f \rVert_{F_k(T)} \approx \lVert f \rVert_{F_{k+1}(T)},
\end{equation*}
for $k \in \mathbf{Z}$ and $f \in F_k(T) \cap F_{k+1}(T)$,
and
\begin{equation*}
\lVert P_k(uv) \rVert_{F_k(T)} \lesssim
\lVert u \rVert_{F_{k^\prime}(T)}
\lVert v \rVert_{L_{t,x}^\infty}
\end{equation*}
for $k, k^\prime \in \mathbf{Z}$, $\lvert k - k^\prime \rvert \leq 20$, $u \in F_{k^\prime}(T)$,
$v \in L^\infty(\RR \times [-T, T])$.
Both of these properties follow readily from the definitions.

In order to bound the nonlinearity of (\ref{NLS}) in $N_k(T)$, it is important to gain
regularity from the parabolic heat-time smoothing effect.  The desired frequency-localized
bounds do not (or at least not so readily) propagate in heat-time in the spaces $G_k(T)$, 
whereas these bounds do propagate with decay
in the larger spaces $F_k(T)$.
Note that since the $F_k(T)$ norm is translation
invariant, it holds that
\begin{equation*}
\lVert e^{s \Delta} h \rVert_{F_k(T)} \lesssim (1 + s 2^{2k})^{-20} \lVert h \rVert_{F_k(T)}
\quad \quad s \geq 0,
\end{equation*}
for $h \in F_k(T)$.
In certain bilinear estimates we do not need the full strength of the spaces
$F_k(T)$ and instead can use the bound
\begin{equation}
\lVert f \rVert_{F_k(T)} \lesssim \lVert f \rVert_{L_x^2 L_t^\infty} + \lVert f \rVert_{L_{t,x}^4},
\label{FsoftBound}
\end{equation}
which follows from
\begin{equation*}
\lVert f \rVert_{L_{\theta, W_{k + m_j}}^{2, \infty}} \leq
\lVert f \rVert_{L_\theta^{2, \infty}}
\lesssim
2^{k/2} \lVert f \rVert_{L_x^2 L_t^\infty}.
\end{equation*}

We introduce one more class of function spaces.  These can be viewed as a refinement
of the Strichartz part of $F_k(T)$.  For $k \in \mathbf{Z}$ and $\omega \in [0, 1/2]$
we define $S_k^\omega(T)$ to be the normed space of functions belonging to
$L_k^2(T)$ whose norm
\begin{equation}
\lVert f \rVert_{S_k^\omega(T)}
=
2^{\omega k} \left(
\lVert f \rVert_{L_t^\infty L_x^{2_\omega}} +
\lVert f \rVert_{L_t^4 L_x^{p_\omega}} +
2^{-k/2} \lVert f \rVert_{L_x^{p_\omega} L_t^\infty} \right)
\label{SomegaDef}
\end{equation}
is finite, 
where the exponents $2_\omega$ and $p_\omega$ are determined by
\begin{equation*}
\frac{1}{2_\omega} - \frac{1}{2}
=
\frac{1}{p_\omega} - \frac{1}{4}
=
\frac{\omega}{2}.
\end{equation*}
Note that $F_k(T) \hookrightarrow S_k^0(T)$
and that by Bernstein we have
\begin{equation*}
\lVert f \rVert_{S_k^{\omega^\prime}(T)}
\lesssim
\lVert f \rVert_{S_k^{\omega}(T)},
\quad\quad
\omega^\prime \leq \omega.
\end{equation*}

\subsection{Bilinear estimates}

\begin{lem}[Bilinear estimates on $N_k(T)$]
For $k, k_1, k_3 \in \mathbf{Z}$, $h \in L_{t,x}^2$, $f \in F_{k_1}(T)$, and $g \in G_{k_3}(T)$,
we have the following inequalities under the given restrictions on $k_1, k_3$.
\begin{align}
\lvert k_1 - k \rvert \leq 80: \quad
\lVert P_k (h f) \rVert_{N_k(T)} 
&\lesssim 
\lVert h \rVert_{L_{t,x}^2} \lVert f \rVert_{F_{k_1}(T)} 
\label{bilin1} \\
k_1 \leq k - 80: \quad
\lVert P_k (h f) \rVert_{N_k(T)} 
&\lesssim 2^{- \lvert k - k_1 \rvert /6} 
\lVert h \rVert_{L_{t,x}^2} \lVert f \rVert_{F_{k_1}(T)} 
\label{bilin2} \\
k \leq k_3 - 80: \quad
\lVert P_k (h g) \rVert_{N_k(T)} 
&\lesssim 2^{- \lvert k - k_3 \rvert /6} 
\lVert h \rVert_{L_{t,x}^2} \lVert g \rVert_{G_{k_3}(T)} 
\label{bilin3}.
\end{align}
\label{L:NkBilinear}
\end{lem}
\begin{proof}
Estimate (\ref{bilin1}) follows from H\"older's inequality
and the definition of $F_k(T), N_k(T)$:
\begin{equation*}
\lVert F f \rVert_{L^{4/3}}
\leq
\lVert F \rVert_{L^2} \lVert f \rVert_{L^4}.
\end{equation*}
For  (\ref{bilin2}) and
(\ref{bilin3}), we use an angular partition of unity in frequency to
write
\begin{equation*}
f = f_1 + f_2, \quad \quad
\lVert f_1 \rVert_{L^{3,6}_{\hat{\theta}_1}}
+
\lVert g_1 \rVert_{L^{3,6}_{\hat{\theta}_2}}
\lesssim
2^{k_1 / 6} \lVert f \rVert_{F_k(T)}
\end{equation*}
and
\begin{equation*}
g = g_1 + g_2, \quad \quad
\lVert g_1 \rVert_{L^{6,3}_{\hat{\theta}_1}}
+
\lVert g_1 \rVert_{L^{6,3}_{\hat{\theta}_2}}
\lesssim
2^{-k_1 / 6} \lVert g \rVert_{G_k(T)}.
\end{equation*}
Then
\begin{align*}
\lVert P_k(F f) \rVert_{N_k(T)}
&\lesssim 2^{-k/6} 
\left( \lVert F f_1 \rVert_{L^{6/5,3/2}_{\hat{\theta}_1}} +
\lVert F f_2 \rVert_{L^{6/5, 3/2}_{\hat{\theta}_2}} \right) \\
&\lesssim 2^{-k/6}
\lVert F \rVert_{L^2}
\left(
\lVert f_1 \rVert_{L^{3,6}_{\hat{\theta}_1}}
+
\lVert f_1 \rVert_{L^{3,6}_{\hat{\theta}_2}}
\right) \\
&\lesssim
2^{(k_1 - k)/6} \lVert F \rVert_{L^2} \lVert f \rVert_{F_{k_1}(T)}.
\end{align*}
and
\begin{align*}
\lVert P_k(F g) \rVert_{N_k(T)}
&\lesssim 2^{k/6} 
\left( \lVert F g_1 \rVert_{L^{3/2,6/5}_{\hat{\theta}_1}} +
\lVert F g_2 \rVert_{L^{3/2, 6/5}_{\hat{\theta}_2}} \right) \\
&\lesssim 2^{k/6}
\lVert F \rVert_{L^2}
\left(
\lVert g_1 \rVert_{L^{6,3}_{\hat{\theta}_1}}
+
\lVert g_1 \rVert_{L^{6,3}_{\hat{\theta}_2}}
\right) \\
&\lesssim
2^{(k - k_1)/6} \lVert F \rVert_{L^2} \lVert g \rVert_{G_{k_3}(T)}.
\end{align*}
\end{proof}

\begin{lem}[Bilinear estimates on $L_{t,x}^2$]
For $k_1, k_2, k_3 \in \mathbf{Z}$, $f_1 \in F_{k_1}(T)$, $f_2 \in F_{k_2}(T)$, and $g \in G_{k_3}(T)$,
we have
\begin{align}
\lVert f_1 \cdot f_2 \rVert_{L_{t,x}^2} &\lesssim \lVert f_1 \rVert_{F_{k_1}(T)} \lVert f_2 \rVert_{F_{k_2}(T)}
\label{bilin4} \\
k_1 \leq k_3: \quad
\lVert f \cdot g \rVert_{L_{t,x}^2} &\lesssim
2^{- \lvert k_1 - k_3 \rvert / 6} \lVert f \rVert_{F_{k_1}(T)} \lVert g \rVert_{G_{k_3}(T)}. 
\label{bilin5}
\end{align}
\label{L:L2Bilinear}
\end{lem}
\begin{proof}
It suffices to show
\begin{equation}
\lVert f g \rVert_{L^2} \lesssim \lVert f \rVert_{F^0_{k_1}(T)}
\lVert g \rVert_{G_{k_2}(T)},
\quad \quad k_1 \geq k_2 - 100
\label{bilin4 suff}
\end{equation}
and
\begin{equation}
\lVert f g \rVert_{L^2}
\lesssim
2^{(k_1 - k_2)/6} \lVert f \rVert_{F^0_{k_1}(T)} \lVert g \rVert_{G_{k_2}(T)},
\quad \quad k_1 < k_2 - 100.
\label{bilin5 suff}
\end{equation}
Estimate (\ref{bilin4 suff}) follows from estimating each factor in $L^4$.
For (\ref{bilin5 suff}), we first observe that, using a smooth partition of unity
in frequency space, we may assume that $\hat{g}$ is supported in
the set
\begin{equation*}
\left\{ \xi : \lvert \xi \rvert \in [2^{k_2 - 1}, 2^{k_2 + 1}] \; \text{and} \;
\xi \cdot \theta_0 \geq 2^{k_2 - 5} \right\}
\end{equation*}
for some direction $\theta_0 \in \mathbf{S}^1$.
Then
\begin{equation*}
\lVert f g \rVert_{L^2}
\lesssim
\lVert f \rVert_{L^{3,6}_{\theta_0}}
\lVert g \rVert_{L^{6,3}_{\theta_0}}
\lesssim 2^{(k_1 - k_2)/6} \lVert f \rVert_{F^0_{k_1}(T)}
\lVert g \rVert_{G_{k_2}(T)}
\end{equation*}
\end{proof}

We also have the following stronger estimates, which rely upon the local smoothing and
maximal function spaces.
\begin{lem}[Bilinear estimates using local smoothing/maximal function bounds]
For $k, k_1, k_2 \in \mathbf{Z}$, $h \in L^2_{t,x}$, $f \in F_{k_1}(T)$,
$g \in G_{k_2}(T)$, we have the following inequalities under
the given restrictions on $k_1, k_2$.
\begin{align}
k_1 \leq k - 80: \quad
&\lVert P_k (h f) \rVert_{N_k(T)} 
\lesssim 2^{- \lvert k - k_1 \rvert / 2} 
\lVert h \rVert_{L_{t,x}^2} \lVert f \rVert_{F_{k_1}(T)} 
\label{bilin2 improv}
\\
k_1 \leq k_2: \quad
&\lVert f \cdot g \rVert_{L_{t,x}^2} 
\lesssim
2^{- \lvert k_1 - k_2 \rvert / 2} \lVert f \rVert_{F_{k_1}(T)} 
\lVert g \rVert_{G_{k_2}(T)}.
\label{bilin5 improv}
\end{align}
\end{lem}
\begin{proof}
Estimate (\ref{bilin2 improv}) follows from the definitions since
\begin{equation*}
\lVert P_k(h f) \rVert_{N_k(T)}
\lesssim
2^{-k/2}
\sup_{\theta \in \mathbf{S}^1}
\lVert h f \rVert_{L^{1,2}_{\theta, W_{k-40}}}
\lesssim
2^{-k/2}
\sup_{\theta \in \mathbf{S}^1}
\lVert f \rVert_{L^{2,\infty}_{\theta, W_{k_1 + 40}}}
\lVert h \rVert_{L^2_{t,x}}.
\end{equation*}
The proof of (\ref{bilin5 improv}) parallels that of (\ref{bilin5})
and is omitted (see \cite[Lemma 6.5]{BeIoKeTa11} for details.)
\end{proof}

\subsection{Trilinear estimates and summation}

We combine the bilinear estimates to
establish some trilinear estimates.
As we do not control local smoothing norms along the heat flow,
we will oftentimes be able to put only one term in a $G_k$
space. Nonetheless, such estimates still exhibit good off-diagonal
decay.

Define the sets $Z_1(k), Z_2(k), Z_3(k) \subset \mathbf{Z}^3$ as follows:
\begin{equation}
\begin{split}
Z_1(k) := 
&\{ (k_1, k_2, k_3) \in \mathbf{Z}^3 :
k_1, k_2 \leq k - 40 \; \mathrm{ and } \;  \lvert k_3 - k \rvert \leq 4 \}, 
\\
Z_2(k) :=
&\{ (k_1, k_2, k_3) \in \mathbf{Z}^3 :
k, k_3 \leq k_1 - 40 \; \mathrm{ and } \;  \lvert k_2 - k_1 \rvert \leq 45\},
\\
Z_3(k) :=
&\{ (k_1, k_2, k_3) \in \mathbf{Z}^3 :
k_3 \leq k \; \mathrm{ and } \; \lvert k - \max\{k_1, k_2\} \rvert \leq 40 \\
& \; \mathbf{or} \;
k_3 > k \; \mathrm{and} \; \lvert k_3 - \max\{k_1, k_2\} \rvert \leq 40 \}.
\end{split}
\label{Z def}
\end{equation}

In our main trilinear estimate, we avoid using local smoothing / maximal function
spaces.
\begin{lem}[Main trilinear estimate]
Let $C_{k,k_1,k_2,k_3}$ denote the best constant $C$ in the estimate
\begin{equation}
\lVert P_k \left( P_{k_1} f_1 P_{k_2} f_2 P_{k_3} g \right) \rVert_{N_k(T)} \lesssim
C
\lVert P_{k_1} f_1 \rVert_{F_{k_1}(T)}
\lVert P_{k_2} f_2 \rVert_{F_{k_2}(T)}
\lVert P_{k_3} g \rVert_{G_{k_3}(T)}.
\label{trilinFFG}
\end{equation}
The best constant $C_{k,k_1,k_2,k_3}$ satisfies the bounds
\begin{displaymath}
C_{k,k_1,k_2,k_3} \lesssim
\left\{ \begin{array}{ll}
2^{ - \lvert (k_1 + k_2)/6 - k/3 \rvert} & (k_1, k_2, k_3) \in Z_1(k) \\
2^{- \lvert k - k_3 \rvert / 6} & (k_1, k_2, k_3) \in Z_2(k) \\
2^{- \lvert \Delta k \rvert / 6} & (k_1, k_2, k_3) \in Z_3(k) \\
0 & (k_1, k_2, k_3) \in \mathbf{Z}^3 \setminus \{ Z_1(k) \cup Z_2(k) \cup Z_3(k) \},
\end{array} \right.
\end{displaymath}
where $\Delta k = \max \{ k, k_1, k_2, k_3 \} - \min \{ k, k_1, k_2, k_3 \} \geq 0$.
\label{L:Trilin1}
\end{lem}
\begin{proof}
After placing the term $P_k \left( P_{k_1} f_1 P_{k_2} f_2 P_{k_3} g \right)$ 
in $L_{t,x}^{4/3}$ and then using H\"older's inequality to bound each factor in $L_{t,x}^4$,
it follows from Bernstein that
\begin{equation}
C_{k, k_1, k_2, k_3} \lesssim 1,
\label{Ck prelim}
\end{equation}
and so, in particular, for any choice of integers $k, k_1, k_2, k_3$, 
such a constant $C_{k,k_1,k_2,k_3}$ exists.

Frequencies not represented in one of
$Z_1(k), Z_2(k), Z_3(k)$ cannot interact so as to yield a frequency in $I_k$.
Over $Z_1(k)$, we apply (\ref{bilin2}) and (\ref{bilin5}).

On $Z_2(k)$ we apply (\ref{bilin2}) if $k > k_3$ and (\ref{bilin3}) if $k \leq k_3$.
We conclude with (\ref{bilin4}).

On $Z_3(k)$ we may assume without loss of generality that $k_1 \leq k_2$.
First suppose that $k_3 \leq k$ and $\lvert k - k_2 \rvert \leq 40$.
If $k_1 \leq k_3$, then use (\ref{bilin2}), applying (\ref{bilin4})
to $P_{k_2}f_2 P_{k_3}g$. If $k_3 < k_1$, then use (\ref{bilin4}) on
$P_{k_1} f_1 P_{k_2} f_2$ instead.

Now suppose that $k_3 > k$ and $\lvert k_3 - k_2 \rvert \leq 40$.
If $k_1 \leq k$, then use (\ref{bilin1}), applying (\ref{bilin5})
to $P_{k_1}f_1 P_{k_3} g$. If $k_{\min} = k$, then use (\ref{bilin3})
and (\ref{bilin4}).
\end{proof}

\begin{cor}
Let $\{ a_k \}$, $\{ b_k \}$, $\{ c_k \}$ be $\delta$-frequency envelopes. 
Let $C_{k, k_1, k_2, k_3}$ be as in Lemma \ref{L:Trilin1}.
Then
\begin{equation*}
\sum_{(k_1, k_2, k_3) \in \mathbf{Z}^3 \setminus Z_2(k)} 
C_{k, k_1, k_2, k_3} a_{k_1} b_{k_2} c_{k_3}
\lesssim
a_k b_k c_k.
\end{equation*}
\label{C:MTE1}
\end{cor}
\begin{proof}
By Lemma \ref{L:Trilin1}, it suffices to restrict the sum to $(k_1, k_2, k_3)$ lying in
$Z_1(k) \cup Z_3(k)$.
On $Z_1(k)$, the sum is bounded by
\begin{align*}
&\sum_{(k_1, k_2, k_3) \in Z_1(k)} 2^{- \lvert (k_1 + k_2)/6 - k/3 \rvert} a_{k_1} b_{k_2} c_{k_3} \\
&\lesssim
\sum_{k_1, k_2 \leq k - 40} 2^{- \lvert (k_1 + k_2)/6 - k/3 \rvert} 2^{\delta \lvert 2k - k_1 - k_2 \rvert} a_k
b_k c_k \\
&\lesssim
a_k b_k c_k.
\end{align*}

On $Z_3$, we may assume without loss of generality that $k_2 \leq k_1$.
The sum is then controlled by
\begin{align*}
&\sum_{(k_1, k_2, k_3) \in Z_3(k)} 2^{- \lvert \Delta k \rvert / 6} a_{k_1} b_{k_2} c_{k_3} 
\\
&\lesssim
\sum_{ \substack{ k_2 \leq k \\ k_3 \leq k \\ \lvert k_1 - k \rvert \leq 40}}  
2^{- \lvert k - \min\{k_2, k_3\} \rvert / 6} 
a_{k_1} b_{k_2} c_{k_3}
+
\sum_{ \substack{ k_2 \leq k_1 \\ k_1 >  k \\ \lvert k_3 - k_1 \rvert \leq 40 }} 
2^{- \lvert k_1 - \min\{k_2, k\} \rvert / 6} 
a_{k_1} b_{k_2} c_{k_3}
\\
&\lesssim
\sum_{ \substack{ k_2 \leq k \\ k_3 \leq k}} 2^{- \lvert k - \min\{k_2, k_3\} \rvert / 6} 
a_k b_{k_2} c_{k_3}+
\sum_{ \substack{ k_2 \leq k_1 \\ k_1 >  k}} 2^{- \lvert k_1 - \min\{k_2, k\} \rvert / 6} 
a_{k_1} b_{k_2} c_{k_1}.
\end{align*}
The first of these summands is controlled by
\begin{align*}
&\sum_{k_3 \leq k_2 \leq k} 2^{-\lvert k - k_3 \rvert / 6} a_k b_{k_2} c_{k_3}
+
\sum_{k_2 < k_3 \leq k} 2^{-\lvert k - k_2 \rvert / 6} a_k b_{k_2} c_{k_3}
\\
&\lesssim
\sum_{k_3 \leq k_2 \leq k} 2^{-\lvert k - k_3 \rvert / 6} 2^{\delta \lvert k - k_2 \rvert} a_k b_k
c_{k_3}
+
\sum_{k_2 < k_3 \leq k} 2^{- \lvert k - k_2 \rvert / 6} 2^{\delta \lvert k - k_3 \rvert} a_k b_{k_2} c_{k} \\
&\lesssim
\sum_{k_3 \leq k} 2^{(\delta - 1/6)\lvert k - k_3 \rvert} a_k b_k c_{k_3}
+
\sum_{k_2 < k} 2^{(\delta - 1/6)\lvert k - k_2 \rvert} a_k b_{k_2} c_{k}
\\
&\lesssim
\sum_{k_3 \leq k} 2^{(2\delta - 1/6)\lvert k - k_3 \rvert} a_k b_k c_{k}
+
\sum_{k_2 < k} 2^{(2\delta - 1/6)\lvert k - k_2 \rvert} a_k b_k c_{k}
\\
&\lesssim a_k b_k c_k.
\end{align*}
The second is controlled by
\begin{align*}
&\sum_{k \leq k_2 \leq k_1} 2^{- \lvert k_1 - k \rvert / 6} a_{k_1} b_{k_2} c_{k_1}
+
\sum_{k_2 < k \leq k_1} 2^{- \lvert k_1 - k_2 \rvert / 6} a_{k_1} b_{k_2} c_{k_1}
\\
&\lesssim
\sum_{k \leq k_2 \leq k_1} 2^{-\lvert k_1 - k \rvert / 6} 2^{\delta \lvert k_2 - k \rvert} 
a_{k_1} b_k c_{k_1}
+
\sum_{k_2 < k \leq k_1} 2^{\lvert k_1 - k_2 \rvert / 6} 2^{\delta \lvert k_2 - k \vert}
a_{k_1} b_k c_{k_1}
\\
&\lesssim
\sum_{k \leq k_1} 2^{(\delta - 1/6)\lvert k_1 - k \rvert} a_{k_1} b_k c_{k_1} +
\sum_{k_2 < k \leq k_1} 2^{(\delta - 1/6) \lvert k_1 - k_2 \rvert} a_{k_1} b_k c_{k_1}
\\
&\lesssim
\sum_{k \leq k_1} 2^{(3\delta - 1/6)\lvert k_1 - k \rvert} a_k b_k c_{k}+
\sum_{k_2 < k \leq k_1} 2^{(3 \delta - 1/6) \lvert k_1 - k_2 \rvert} a_k b_k c_{k}
\\
&\lesssim
a_k b_k c_k.
\end{align*}

\end{proof}

\begin{cor}
Let $\{a_k \}, \{ b_k \}$ be $\delta$-frequency envelopes. 
Let $C_{k, k_1, k_2, k_3}$ be as in Lemma \ref{L:Trilin1}.
Then
\begin{equation*}
\sum_{(k_1, k_2, k_3) \in Z_2(k) \cup Z_3(k)}
2^{\max\{k, k_3\} - \max\{k_1, k_2\}}
C_{k, k_1, k_2, k_3} a_{k_1} b_{k_2} c_{k_3}
\lesssim a_k b_k c_k
\end{equation*}
\label{C:MTE2}
\end{cor}
\begin{proof}
On $Z_3(k)$, $\max\{k_1, k_2\} \sim \max\{k, k_3\}$,
and so the bound on $Z_3(k)$ follows from
Corollary \ref{C:MTE1}.

Note that $\max\{k_1, k_2\} > \max\{k, k_3\}$
on $Z_2$, where the sum is controlled by
\begin{align*}
&\sum_{(k_1, k_2, k_3) \in Z_2(k)} 
2^{\max\{k, k_3\} - \max\{k_1, k_2\}}
2^{- \lvert k - k_3 \rvert / 6} a_{k_1} b_{k_2} c_{k_3} \\
&\lesssim
\sum_{k, k_3 \leq k_1 - 40} 
2^{\max\{k, k_3\} - k_1}
2^{- \lvert k - k_3 \rvert / 6} a_{k_1} b_{k_1} c_{k_3},
\end{align*}
Restricting the sum to $k_3 \leq k$, we get
\begin{equation*}
\sum_{k_3 \leq k \leq k_1 - 40} 
2^{- \lvert k - k_1 \rvert}
2^{- \lvert k - k_3 \rvert / 6} a_{k_1} b_{k_1} c_{k_3}
\lesssim
a_k b_k c_k
\end{equation*}
Over the complementary range $k \leq k_3 \leq k_1 - 40$, we have
\begin{align*}
&\sum_{k \leq k_3 \leq k_1 - 40} 
2^{-\lvert k_3 - k_1 \rvert }
2^{- \lvert k - k_3 \rvert / 6} a_{k_1} b_{k_1} c_{k_3} \\
&\lesssim
a_k b_k c_k
\sum_{k \leq k_3 \leq k_1 - 40}
2^{-\lvert k_3 - k_1 \rvert }
2^{- \lvert k - k_3 \rvert / 6}
2^{2 \delta \lvert k_1 - k \rvert}
2^{\delta \lvert k - k-3 \rvert}.
\end{align*}
Performing the change of variables $j := k_1 - k_3, \ell := k_3 - k$, we control
the sum by
\begin{equation*}
\sum_{j, \ell \geq 0} 2^{-j} 2^{-\ell / 6} 2^{2 \delta (j + \ell)} 2^{\delta \ell}
\lesssim
\sum_{j, \ell \geq 0} 2^{(2 \delta - 1) j} 2^{(3 \delta - 1/6) \ell}
\lesssim
1.
\end{equation*}
\end{proof}

Taking advantage of the local smoothing/maximal function spaces,
we can obtain the following improvement.
\begin{lem}[Main trilinear estimate improvement over $Z_1$]
The best constant $C_{k,k_1,k_2,k_3}$ in $(\ref{trilinFFG})$
satisfies the improved estimate
\begin{equation}
C_{k,k_1,k_2,k_3}
\lesssim
2^{ - \lvert (k_1 + k_2)/2 - k \rvert}
\end{equation}
when $\{ k_1, k_2, k_3 \} \in Z_1(k)$.
\end{lem}

%% file: MainTheorem.tex
 \section{Proof of Theorem \ref{MainTheorem}} \label{S:Main}

In this section we outline the proof of Theorem \ref{MainTheorem},
taking as our starting point the local result stated in Theorem \ref{LWP}.

For technical reasons related to the function space definitions of the last section, 
it will be convenient to construct a solution $\varphi$ on a time
interval $(-2^{2 \mathcal{K}}, 2^{2 \mathcal{K}})$ for some given $\mathcal{K} \in \mathbf{Z}_+$
and proceed to prove bounds that are uniform in $\mathcal{K}$.  We assume
$1 \ll \mathcal{K} \in \mathbf{Z}_+$ is chosen and hereafter fixed.
Invoking Theorem \ref{LWP}, we assume that we have a solution 
$\varphi \in C([-T, T] \to H_Q^\infty)$ of (\ref{SM}) on the time interval $[-T, T]$ for some
$T \in (0, 2^{2 \mathcal{K}})$. 
In order to extend $\varphi$ to a solution on all of
$(-2^{2 \mathcal{K}}, 2^{2 \mathcal{K}})$ with uniform bounds 
(uniform in $T, \mathcal{K}$),
it suffices
to prove
uniform a priori estimates on
\begin{equation*}
\sup_{t \in (-T, T)} \lVert \varphi(t) \rVert_{H_Q^\sigma}
\end{equation*}
for, say, $\sigma$ in the interval $[1, \sigma_1]$, with $\sigma_1 \gg 1$ chosen sufficiently large
(e.g., $\sigma_1 = 25$ will do).  

The first step in our approach, carried out in \S \ref{S:Gauge Field Equations},
is to lift the Schr\"odinger map system (\ref{SM})
to the tangent bundle and view it with respect to the caloric gauge.
Recall that the lift of (\ref{SM}) expressed in terms of the caloric gauge takes the form
(\ref{NLS}),
or, equivalently,
\begin{equation}
(i \partial_t + \Delta) \psi_m = B_m + V_m,
\label{dSM}
\end{equation}
with initial data $\psi_m(0)$.  Here $B_m$ and $V_m$ respectively denote
the magnetic and electric potentials (see (\ref{B Def}) and (\ref{V Def}) for
definitions).

The goal then becomes proving
a priori bounds on $\lVert \psi_m \rVert_{L_t^\infty H_x^\sigma}$.
Herein lies the heart of the argument, and
the purpose of this section is not only to give a high level description
of the proof of Theorem \ref{MainTheorem}, 
but also to outline the proof of the key a priori bounds.
To establish these bounds, we in fact prove stronger frequency-localized estimates.
The argument naturally splits into several components, and we consider
each individually below.

Finally, to complete the proof of Theorem \ref{MainTheorem}, we must transfer the a priori
bounds on the derivative fields $\psi_m$ back to bounds on the map $\varphi$,
thereby allowing us to close a bootstrap argument.
Once the derivative field bounds are established,
this is, comparatively speaking, an easy task, and we take it up in the last subsection.

We return now to (\ref{dSM}), projecting it to frequencies $\sim 2^k$
using the Littlewood-Paley multiplier $P_k$.
Applying the linear estimate of
Lemma \ref{MainLinearEstimate}
then yields
\begin{equation}
\lVert P_k \psi_m \rVert_{G_k(T)} \lesssim \lVert P_k \psi_m (0) \rVert_{L_x^2} +
\lVert P_k V_m \rVert_{N_k(T)} + \lVert P_k B_m \rVert_{N_k(T)}.
\label{mainestimate}
\end{equation}
In order to express control of the $G_k(T)$ norm of $P_k \psi_m$ in terms of the initial data,
we introduce the following frequency envelopes.
Let $\sigma_1 \in \mathbf{Z}_+$ be positive.
For $\sigma \in [0, \sigma_1 - 1]$, set
\begin{equation}
b_k(\sigma) = \sup_{k^\prime \in \mathbf{Z}} 2^{\sigma k^\prime} 2^{-\delta \lvert k - k^\prime \rvert}
\lVert P_{k^\prime} \psi_x \rVert_{G_k(T)}.
\label{b Envelope}
\end{equation}
By (\ref{Soft Field Space Bounds}),
these envelopes are finite and in $\ell^2$.
We abbreviate $b_k(0)$ by setting $b_k := b_k(0)$.

We now state the key result for solutions of the gauge field equation (\ref{dSM}).
\begin{thm}
Assume $T \in (0, 2^{2 \mathcal{K}})$ and $Q \in \mathbf{S}^2$.
Choose $\sigma_1 \in \mathbf{Z}_+$ positive.
Let $\varepsilon_1 > 0$
and let $\varphi \in H_Q^{\infty, \infty}(T)$
be a solution of the Schr\"odinger map system (\ref{SM}) whose initial data $\varphi_0$
has energy $E_0 := E(\varphi_0) < E_{\mathrm{crit}}$ and satisfies the energy dispersion condition
\begin{equation}
\sup_{k \in \mathbf{Z}} c_k \leq \varepsilon_1.
\label{iED}
\end{equation}
Assume moreover that
\begin{equation}
\sum_{k \in \mathbf{Z}} \lVert P_k \psi_x \rVert_{L^4_{t,x}(I \times \rr)}^2 \leq \varepsilon_1^2
\label{Main L4 Cal}
\end{equation}
for any smooth extension $\varphi$ on $I$, $[-T, T] \subset I \subset (-2^{2 \mathcal{K}}, 2^{2 \mathcal{K}})$.
Suppose that the bootstrap hypothesis
\begin{equation}
b_k \leq \varepsilon_1^{-1/10} c_k
\label{Main Bootstrap}
\end{equation}
is satisfied.
Then, for $\varepsilon_1$ sufficiently small,
\begin{equation}
b_k(\sigma) \lesssim c_k(\sigma)
\label{Field Conclusion}
\end{equation}
holds for all $\sigma \in [0, \sigma_1 - 1]$ and $k \in \mathbf{Z}$.
\label{EnvelopeTheorem}
\end{thm}
\begin{proof}
We use a continuity argument to prove Theorem \ref{EnvelopeTheorem}.
For $T^\prime \in (0, T]$, let
\begin{equation*}
\Psi(T^\prime) = \sup_{k \in \mathbf{Z}} c_k^{-1}
\lVert P_k \psi_m(s = 0) \rVert_{G_k(T^\prime)}.
\end{equation*}
Then $\psi : (0, T] \to [0, \infty)$ is well-defined, increasing, continuous, and satisfies
\begin{equation*}
\lim_{T^\prime \to 0} \psi(T^\prime) \lesssim 1.
\end{equation*}
The critical implication to establish is
\begin{equation*}
\Psi(T^\prime) \leq \varepsilon_1^{-1/10} \implies
\Psi(T^\prime) \lesssim 1,
\end{equation*}
which in particular follows from
\begin{equation}
b_k \lesssim c_k.
\label{Goal}
\end{equation}
We also must similarly establish
\begin{equation}
b_k(\sigma) \lesssim c_k(\sigma)
\label{Goal2}
\end{equation}
for $\sigma \in (0, \sigma_1 - 1]$.
The next several subsections describe the main steps of the proof of (\ref{Goal})
and (\ref{Goal2}),
to which the bulk of the remainder of this paper is dedicated.
In
\S \ref{SS:OutlineConclusion}
we complete the high level argument used to prove
(\ref{Goal})
and (\ref{Goal2}).
\end{proof}

\begin{cor}
Given the conditions of Theorem \ref{EnvelopeTheorem},
\begin{equation}
\lVert P_k \lvert \partial_x \rvert^\sigma \partial_m \varphi 
\rVert_{L_t^\infty L_x^2( (-T, T) \times \rr)}
\lesssim c_k(\sigma)
\label{Main Conclusion}
\end{equation}
holds for all $\sigma \in [0, \sigma_1 - 1]$.
\label{ETcor}
\end{cor}
The proof we defer to \S \ref{SS:Degauging}.

Together Theorem \ref{LWP}, Theorem \ref{EnvelopeTheorem}, and Corollary
\ref{ETcor} are almost enough to establish Theorem \ref{MainTheorem}.
The next lemma provides the final piece.  We also defer its proof to \S \ref{SS:Degauging}.

\begin{lem}
It holds that
\begin{equation*}
\sum_{k \in \mathbf{Z}} \lVert P_k \psi_x \rVert_{L_{t,x}^4}^2
\sim
\sum_{k \in \mathbf{Z}} \lVert P_k \partial_x \varphi \rVert_{L_{t,x}^4}^2.
\end{equation*}
\label{equiv mt con}
\end{lem}
Note that this lemma affords us a condition equivalent to (\ref{Main L4 Cal})
whose advantage lies in the fact that it is not expressed in terms of gauges.

\begin{proof}[Proof of Theorem \ref{MainTheorem}]
Fix $\sigma_1 \in \mathbf{Z}_+$ positive and let $\varepsilon_1 = \varepsilon_1(\sigma_1) \geq 0$.
It suffices to prove (\ref{SobolevBound}) on the time interval $[-T, T]$ provided the
estimate is uniform in $T$.  In view of Theorem \ref{LWP} and mass-conservation, proving
\begin{equation}
\lVert \partial_x \varphi \rVert_{L^\infty_t \dot{H}^\sigma_Q((-T, T) \times \rr)}
\lesssim_\sigma
\lVert \partial_x \varphi \rVert_{\dot{H}^\sigma_Q(\rr)}
\label{NonGlobSoboBd}
\end{equation}
for $\sigma \in [0, \sigma_1 - 1]$ with $\sigma_1 = 25$ is
enough to establish (\ref{SoboBound}).

By virtue of Lemma \ref{equiv mt con}, the assumptions of Theorem \ref{MainTheorem}
are equivalent to those of \ref{EnvelopeTheorem}.  Therefore 
we have access
to Corollary \ref{ETcor}, which states that (\ref{Main Conclusion}) holds
$\sigma \in [0, \sigma_1 - 1]$.
Using (\ref{c cons}) 
and the Littlewood-Paley square function
completes the proof of (\ref{NonGlobSoboBd}).

Global existence and
(\ref{SobolevBound}) then follow 
via a standard bootstrap argument 
from 
Theorem \ref{LWP}
and from the fact that the constants
in (\ref{NonGlobSoboBd}) are uniform in $T$.
\end{proof}

The remainder of this section is organized as follows.
In \S \ref{SS:ParabolicEstimates} we state the key lemmas of parabolic type that
are used to control the electric and magnetic nonlinearities.
In \S \ref{SS: SmoothingStrichartz} we state bounds that rely principally upon
local smoothing, including a bilinear Strichartz estimate; they find application in controlling the worst
magnetic nonlinearity terms.

In \S \ref{SS:V} we piece together the parabolic estimates to control the electric potential.
In \S \ref{SS:B} we decompose the magnetic potential into two main pieces and demonstrate
how to control one of these pieces.

In \S \ref{SS:OutlineConclusion} we close the bootstrap argument
proving Theorem \ref{EnvelopeTheorem}. Here the remaining piece of the magnetic potential
is addressed using a certain nonlinear version of a bilinear Strichartz estimate.

Finally, in \S \ref{SS:Degauging}, we prove Corollary \ref{ETcor} and
Lemma \ref{equiv mt con}.

\subsection{Parabolic estimates} \label{SS:ParabolicEstimates}

By ``parabolic estimates" we mean those that principally rely upon the smoothing effect of
the harmonic map heat flow.  We include here only those that play a direct role in controlling
the nonlinearity $\cN$.  These are proved in \S \ref{S:Parabolic Estimates}, where
a host of auxiliary parabolic estimates are included as well.  
As the proofs rely upon a bootstrap argument
that takes advantage of energy dispersion (\ref{iED}), these bounds rely upon
this smallness constraint implicitly.  On the other hand, $L^4$ smallness (\ref{Main L4 Cal}) is not used in the
proofs of these bounds, but rather only in their application in this paper.

\begin{lem}
For $\sigma \in [0, \sigma_1 - 1]$,
the derivative fields $\psi_m$ satisfy
\begin{equation}
\lVert P_k \psi_m(s) \rVert_{F_k(T)} \lesssim 
(1 + s 2^{2k})^{-4} 2^{-\sigma k} b_k(\sigma)
\label{PsiF}
\end{equation}
for $s \geq 0$.
\end{lem}
This estimate is used in \S \ref{SS:B} in controlling the magnetic nonlinearity,
which schematically looks like $A \partial_x \psi$.  To recover the loss of
derivative, it is important to take advantage of parabolic smoothing by
invoking representation (\ref{CC Integral Rep}) of $A$.  
Within the integral we schematically have $\psi(s) D_x \psi(s)$,
and hence
(\ref{PsiF}) allows us to take advantage of (\ref{bilin1})--(\ref{bilin5})
in bounding this term.
We prove (\ref{PsiF}) in \S \ref{SS:DFC}.

\begin{lem}
For $\sigma \in [0, \sigma_1 - 1]$,
the derivative fields $\psi_\ell$ and connection coefficients $A_m$ satisfy
\begin{equation}
\lVert P_k (A_m(s) \psi_\ell(s)) \rVert_{F_k(T)} \lesssim
(s 2^{2k})^{-3/8} (1 + s2^{2k})^{-2} 2^{-(\sigma - 1)k} b_k(\sigma).
\label{APsiF}
\end{equation}
\end{lem}
Like the previous estimate, this estimate is also used 
in \S \ref{SS:B} in controlling the magnetic nonlinearity.
Its proof is given in \S \ref{SS:CCC}.
The need for this estimate arises from the need
to control $D_x \psi$ appearing in
representation (\ref{CC Integral Rep}) of $A$.

The next several estimates are used in \S \ref{SS:V} to control the electric
potential.  In particular, they provide a source of smallness crucial here for closing
the bootstrap argument.  
They are proved in \S \ref{SS:CCC}.

\begin{lem}
For $\sigma \in [2 \delta, \sigma_1 - 1]$, the connection coefficient $A_x$ satisfies
\begin{equation}
\lVert A_x^2 \rVert_{L_{t,x}^2} 
\lesssim
\sup_{j \in \mathbf{Z}} b_j^2
\cdot
\sum_{k \in \mathbf{Z}} b_k^2
\label{AxL2}
\end{equation}
and
\begin{equation}
\lVert P_k A_x^2(0) \rVert_{L^2_{t,x}}
\lesssim
2^{-\sigma k} b_k(\sigma) \cdot \sup_{j} b_j \cdot \sum_{\ell \in \mathbf{Z}} b_\ell^2.
\label{PkAxL2}
\end{equation}
\end{lem}

\begin{lem}
For $\sigma \in [2 \delta, \sigma_1 - 1]$, the connection coefficient $A_t$ satisfies
\begin{equation}
\lVert A_t \rVert_{L_{t,x}^2}
\lesssim
(1 + \sum_{j \in \mathbf{Z}} b_j^2)^2 \sum_{k \in \mathbf{Z}} \lVert P_k \psi_x(0) \rVert_{L_{t,x}^4}^2
\label{AtL2}
\end{equation}
and
\begin{equation}
\lVert P_k A_t \rVert_{L^2_{t,x}} 
\lesssim
(1 + \sum_p b_p^2) \tilde{b}_k 2^{-\sigma k} b_k(\sigma).
\label{PkAtL2}
\end{equation}
\end{lem}

In subsequent estimates the following shorthand will be useful:
\begin{equation}
\cV := 
(1 + \sum_{j \in \mathbf{Z}} b_j^2)^2
\sum_{\ell \in \mathbf{Z}} \lVert P_\ell \psi_x(0) \rVert_{L_{t,x}^4}^2
+ (1 + \sum_{\ell} b_\ell^2)  \sup_{k \in \mathbf{Z}} b_k^2.
\label{calV def}
\end{equation}
Under the assumptions of Theorem \ref{EnvelopeTheorem},
$\cV$ is a very small quantity, being at least as good as $O(\varepsilon_1^{1/2})$.

\subsection{Smoothing and Strichartz} \label{SS: SmoothingStrichartz}

The key result of \S \ref{S:Local smoothing} is the following frequency-localized bilinear Strichartz estimate.
\begin{thm}
Suppose that $\psi_m$ satisfies (\ref{NLS}) on $[-T, T]$.  
Assume $\sigma \in [0, \sigma_1 - 1]$.
Let the frequency envelopes
$b_j$ and $c_j$ be defined as in (\ref{b Envelope}) and 
(\ref{c Envelope}).
Let $\cV$ be given by (\ref{calV def}). 
Suppose also that  $2^{j - k} \ll1$.  Then
\begin{align}
2^{k - j} (1 + s 2^{2j})^{8}
\lVert P_j \psi_\ell(s) \cdot P_k \psi_m(0) \rVert_{L^2_{t,x}}^2
\lesssim 2^{-2 \sigma k}
c_j^2 c_k^2(\sigma) + 
\cV^2 b_j^2 b_k^2(\sigma).
\label{BSs}
\end{align}
\label{T:BilinearStrichartz}
\end{thm}
In \S \ref{SS:App}  
we split the proof into two cases: $s = 0$ and $s > 0$,
the more involved being the $s = 0$ case.  In either case,
if instead we only were to appeal to the local smoothing-based estimate (\ref{bilin5 improv})
and the frequency envelope definition (\ref{b Envelope}), then we would
get the bound
\begin{equation*}
2^{k - j} (1 + s 2^{2j})^8
\lVert P_j \psi_\ell(s) \cdot P_k \psi_m(0) \rVert_{L^2_{t,x}}^2
\lesssim
b_j^2 b_k^2.
\end{equation*}
In practice this sort of bound must needs be summed over $j \ll k$.
When initial energy is assumed to be small, as is done in \cite{BeIoKeTa11},
the sum $\sum_j b_j^2 \ll 1$ is small, and consequently the resulting term perturbative.
In our subthreshold energy setting this is no longer the case, as in fact the sum
may be large.  What (\ref{BSs}) reveals, though, is that any $b_j$ contributions
come with a power of $\varepsilon$.  
In view of additional work which we present in due course,
this turns out to be sufficient for establishing $b_k \lesssim c_k$.

An interesting related bound is the following local smoothing estimate, also
proved in \S \ref{SS:App}. 
It arises as an easy corollary of our proof of Theorem \ref{T:BilinearStrichartz}.
\begin{thm}
Suppose that $\psi_m$ satisfies (\ref{NLS}) on $[-T, T]$.  
Assume $\sigma \in [0, \sigma_1 - 1]$.
Let the frequency envelopes
$b_j(\sigma)$ and $c_j(\sigma)$ be defined as in (\ref{b Envelope}) and 
(\ref{c Envelope}).
Also, let $\cV$ be given by (\ref{calV def}).
Then
\begin{equation}
2^{k} \sup_{\lvert j - k \rvert \leq 20} \sup_{\theta \in \mathbf{S}^1}
\lVert P_{j, \theta} P_k \psi_m
\rVert_{L_{\theta}^{\infty, 2}}^2
 \lesssim
2^{-2 \sigma k} c_k^2(\sigma) +
\cV 2^{-2 \sigma k} b_k^2(\sigma)
\label{LS}
\end{equation}
holds for each $k \in \mathbf{Z}$.
\label{ls thm}
\end{thm}
We note that (\ref{LS}) likely extends to $L^{\infty, 2}_{\theta, \lambda}$ for
$\lambda$ satisfying $|\lambda| < 2^{k -40}$, though we do not prove this.
For comparison, note that
from the definition of (\ref{b Envelope}) we have
\begin{equation}
2^{k} \sup_{\lvert j - k \rvert \leq 20} \sup_{\theta \in \mathbf{S}^1}
\sup_{\lvert \lambda \rvert < 2^{k - 40}} \lVert P_{j, \theta} P_k \psi_m
\rVert_{L_{\theta, \lambda}^{\infty, 2}}^2
 \lesssim
2^{-2 \sigma k} b_k^2(\sigma).
\label{LSold}
\end{equation}
On the other hand,
while the right hand side of (\ref{LS}) may indeed be large,
it so happens thanks to our hypotheses of energy dispersion
and $L^4$ smallness that the $b_k(\sigma)$ term is perturbative.
For our purposes, this is a substantial improvement over (\ref{LSold}).
However, it can be seen from the argument in \S \ref{SS:OutlineConclusion} that
even an extension of (\ref{LS}) to $L^{\infty, 2}_{\theta, \lambda}$ spaces
is not sufficient for proving 
$b_k(\sigma) \lesssim c_k(\sigma)$: it is important that we can replace two ``$b_j$" terms
with corresponding ``$c_j$" terms as in (\ref{BSs}).

\subsection{Controlling the electric potential $V$} \label{SS:V}

\begin{lem}
Suppose that $\sigma < \frac{1}{6} - 2\delta$.
Then the electric potential term $V_m$ satisfies the estimate
\begin{equation}
\lVert P_k V_m \rVert_{N_k(T)} \lesssim
\left( \lVert A_x^2 \rVert_{L_{t,x}^2} + \lVert A_t \rVert_{L_{t,x}^2}
+ \lVert \psi_x^2 \rVert_{L_{t,x}^2} \right) 2^{-\sigma k} b_k(\sigma).
\label{VmBound}
\end{equation}
\label{L:VmBound}
\end{lem}
\begin{proof}
Letting $f \in \{ A_t, A_x^2, \psi_x^2\}$, we bound $P_k (f \psi_x) $ in $N_k(T)$.
Begin with the following Littlewood-Paley decomposition of $P_k(f \psi_x)$:
\begin{align*}
P_k(f \psi_x) =& \; P_k (P_{<k - 80}f P_{k - 5 < \cdot < k + 5} \psi_x) + \\
&\sum_{ \substack{\lvert k_1 - k \rvert \leq 4 \\ k_2 \leq k - 80}} P_k(P_{k_1}f P_{k_2} \psi_x) + \\
&\sum_{ \substack{\lvert k_1 - k_2 \rvert \leq 90 \\ k_1, k_2 > k - 80}}
P_k( P_{k_1}f P_{k_2} \psi_x ).
\end{align*}
The first term is controlled using H\"older's inequality:
\begin{align*}
\lVert P_k (P_{<k - 80}f P_{k - 5 < \cdot < k + 5} \psi_x) \rVert_{N_k(T)}
&\leq \lVert P_k (P_{<k - 80}f P_{k - 5 < \cdot < k + 5} \psi_x) \rVert_{L_{t,x}^{4/3}} \\
&\leq \lVert P_{< k - 80} f \rVert_{L_{t,x}^2} \lVert P_{k - 5 < \cdot < k + 5} \psi_x \rVert_{L_{t,x}^4}.
\end{align*}
To control the second term we apply (\ref{bilin2}):
\begin{equation*}
\lVert P_k (P_{k_1} f P_{k_2} \psi_x) \rVert_{N_k(T)} 
\lesssim 
2^{(k_2 - k)/6}
\lVert P_{k_1} f \rVert_{L_{t,x}^2} \lVert P_{k_2} \psi_x \rVert_{G_{k_2}(T)}.
\end{equation*}
Using (\ref{b Envelope}), (\ref{Sum 1}), and $\sigma < 1/6 - 2 \sigma$,
we conclude
\begin{equation*}
\lVert  \sum_{ \substack{\lvert k_1 - k \rvert \leq 4 \\ k_2 < k - 80}} P_k(P_{k_1}f P_{k_2} \psi_x)
\rVert_{N_k(T)}
\lesssim
2^{-\sigma k} b_k(\sigma) \sum_{\lvert k_1 - k \rvert \leq 4} \lVert P_{k_1} f \rVert_{L_{t,x}^2}.
\end{equation*}
To control the high-high interaction, apply (\ref{bilin3}):
\begin{equation*}
\lVert P_k( P_{k_1}f P_{k_2} \psi_x ) \rVert_{N_k(T)}
\lesssim
2^{(k - k_2)/6} \lVert P_{k_1} f \rVert_{L_{t,x}^2} \lVert P_{k_2} \psi_x \rVert_{G_{k_2}(T)}.
\end{equation*}
Therefore, by (\ref{b Envelope}),
\begin{equation*}
\sum_{ \substack{\lvert k_1 - k_2 \rvert \leq 90 \\ k_1, k_2 > k - 80}}
\lVert 
P_k( P_{k_1}f P_{k_2} \psi_x )
\rVert_{N_k(T)}
\lesssim
\sum_{ \substack{\lvert k_1 - k_2 \rvert \leq 90 \\ k_1, k_2 > k - 80}}
2^{(k - k_2)/6} \lVert P_{k_1} f \rVert_{L_{t,x}^2} 2^{-\sigma k_2} b_{k_2}(\sigma).
\end{equation*}
Using Cauchy-Schwarz and (\ref{Sum 2}) yields
\begin{equation*}
\sum_{ \substack{\lvert k_1 - k_2 \rvert \leq 90 \\ k_1, k_2 > k - 80}}
\lVert 
P_k( P_{k_1}f P_{k_2} \psi_x )
\rVert_{N_k(T)}
\lesssim
2^{-\sigma k} b_k(\sigma) 
\left( \sum_{k_1 \geq k - 80} \lVert P_{k_1} f \rVert_{L_{t,x}^2}^2 \right)^{1/2},
\end{equation*}
and so, by switching the $L_{t,x}^2$ and $\ell^2$ norms, 
we get from the standard square function estimate
that
\begin{equation*}
\sum_{ \substack{\lvert k_1 - k_2 \rvert \leq 90 \\ k_1, k_2 > k - 80}}
\lVert 
P_k( P_{k_1}f P_{k_2} \psi_x )
\rVert_{N_k(T)}
\lesssim
\lVert f \rVert_{L_{t,x}^2} 2^{-\sigma k} b_k(\sigma).
\end{equation*}
\end{proof}

\begin{cor}
For $\sigma \in [0, \sigma_1 - 1]$ it holds that
\begin{equation*}
\lVert P_k V_m \rVert_{N_k(T)} 
\lesssim
\cV 2^{-\sigma k} b_k(\sigma).
\end{equation*}
\label{V bound corollary}
\end{cor}
\begin{proof}
Given (\ref{VmBound}), this is a direct consequence of
(\ref{AxL2}), (\ref{AtL2}), and the fact that
\begin{equation*}
\lVert f \rVert_{L^4_{t,x}}^2
\lesssim
\sum_{k \in \mathbf{Z}} \lVert P_k f \rVert_{L^4_{t,x}}^2.
\end{equation*}
Therefore the result holds for $\sigma < 1/6 - 2 \delta$.

To extend the proof to larger $\sigma$, we may mimic the proof of
Lemma \ref{L:VmBound} by performing the same Littlewood-Paley decomposition
and then, with regard to the first and third terms of the decomposition, proceeding as
before in the proof of that lemma. The argument, however, must be modified in handling the
term
\begin{equation}
\sum_{ \substack{|k_1 - k| \leq 4 \\ k_2 \leq k - 80}} P_k (P_{k_1}f P_{k_2} \psi_x),
\label{V-term2}
\end{equation}
where $f \in \{A_t, A_x^2, \psi_x^2\}$.
We take different approaches according to the choice of $f$.

When $f = A_x^2$, we apply (\ref{bilin2}) and invoke (\ref{PkAxL2}) to obtain
\begin{align*}
\lVert \sum_{\substack{ | k_1 - k | \leq 4 \\ k_2 < k - 80} } P_k( P_{k_1} A_x^2 P_{k_2} \psi_x) \rVert_{N_k(T)}
&\lesssim \;
\sum_{\substack{ | k_1 - k | \leq 4 \\ k_2 < k - 80} }
2^{(k_2 - k)/6} \lVert P_{k_1} A_x^2 \rVert_{L^2_{t,x}}
\lVert P_{k_2} \psi_x \rVert_{G_{k_2}(T)} \\
&\lesssim \;
\sum_{\substack{ | k_1 - k | \leq 4 \\ k_2 < k - 80} }
2^{(k_2 - k)/6}
2^{-\sigma k_1} b_{k_1}(\sigma) b_{k_2} \cdot \sup_{j} b_j \cdot \sum_{\ell} b_\ell^2 \\
&\lesssim \;
2^{-\sigma k} b_k(\sigma) \cdot b_k \cdot \sup_{j} b_j \cdot \sum_j b_j^2.
\end{align*}

In the case where $f = A_t$, we apply (\ref{bilin2}) and use (\ref{PkAtL2}) to conclude
\begin{equation*}
\lVert \sum_{\substack{ | k_1 - k | \leq 4 \\ k_2 < k - 80} } P_k( P_{k_1} A_t P_{k_2} \psi_x) \rVert_{N_k(T)}
\lesssim
2^{-\sigma k} b_k(\sigma) \tilde{b}_k b_k (1 + \sum_p b_p^2),
\end{equation*}
which suffices by Cauchy-Schwarz.

Finally we turn to $f = \psi_x^2$, which we further decompose as
\[
f = 
2 \sum_{\substack{|j_1 - k | \leq 4 \\ j_2 < k - 80}} P_{j_1} \psi_x P_{j_2} \psi_x
+
\sum_{\substack{|j_1 - j_2| \leq 8 \\ j_1, j_2 \geq k - 80}} P_{j_1} \psi_x P_{j_2} \psi_x.
\]
To control the high-low term, we apply estimate (\ref{bilin5}):
\begin{align*}
\sum_{\substack{|j_1 - k | \leq 4 \\ j_2 < k - 80}}
\lVert P_{j_1} \psi_x P_{j_2} \psi_x \rVert_{L^2}
&\lesssim \;
 \sum_{\substack{|j_1 - k | \leq 4 \\ j_2 < k - 80}}
2^{(j_2-j_1)/6} b_{j_2} 2^{-\sigma j_1} b_{j_1}(\sigma) \\
&\lesssim \;
2^{-\sigma k} b_k b_k(\sigma).
\end{align*}

We turn to the high-high case. 
The full trilinear expression is given by
\begin{equation*}
\sum_{\substack{ | k_1 - k | \leq 4 \\ k_2 < k - 80} } 
P_k( P_{k_1} (
\sum_{\substack{|j_1 - j_2| \leq 8 \\ j_1, j_2 \geq k_1 - 80}}
P_{j_1} \psi_x P_{j_2} \psi_x)
\cdot P_{k_2} \psi_x)
\end{equation*}

We can drop the $P_{k_1}$ factor because of the summation ranges:
\begin{equation*}
\sum_{\substack{ | k_1 - k | \leq 4 \\ k_2 < k - 80} } 
\sum_{\substack{|j_1 - j_2| \leq 8 \\ j_1, j_2 \geq k_1 - 80}}
P_k(  P_{j_1} \psi_x P_{j_2} \psi_x \cdot P_{k_2} \psi_x).
\end{equation*}
We apply estimate (\ref{bilin2}) with $h = P_{j_2} \psi_x P_{k_2} \psi_x$ to get
\begin{align*}
\sum_{\substack{ | k_1 - k | \leq 4 \\ k_2 < k - 80} } 
&\sum_{\substack{|j_1 - j_2| \leq 8 \\ j_1, j_2 \geq k_1 - 80}}
\lVert P_k(  P_{j_1} \psi_x P_{j_2} \psi_x \cdot P_{k_2} \psi_x) \rVert_{N_k(T)} \\
&\lesssim \;
\sum_{\substack{ | k_1 - k | \leq 4 \\ k_2 < k - 80} } 
\sum_{\substack{|j_1 - j_2| \leq 8 \\ j_1, j_2 \geq k_1 - 80}}
2^{-|j_1 - k|/6} \lVert P_{j_1} \psi_x \rVert_{G_{j_1}(T)}
\lVert P_{j_2} \psi_x P_{k_2} \psi_x \rVert_{L^2}.
\end{align*}
Next we use (\ref{bilin5}) to control the $L^2$ norm:
\begin{align*}
\sum_{\substack{ | k_1 - k | \leq 4 \\ k_2 < k - 80} } 
&\sum_{\substack{|j_1 - j_2| \leq 8 \\ j_1, j_2 \geq k_1 - 80}}
2^{-|j_1 - k|/6} \lVert P_{j_1} \psi_x \rVert_{G_{j_1}(T)}
\lVert P_{j_2} \psi_x P_{k_2} \psi_x \rVert_{L^2} \\
&\lesssim \;
\sum_{\substack{ | k_1 - k | \leq 4 \\ k_2 < k - 80} } 
\sum_{\substack{|j_1 - j_2| \leq 8 \\ j_1, j_2 \geq k_1 - 80}}
2^{-|j_1 - k|/6} 2^{-|j_2 - k_2|/6}
2^{- \sigma j_1} b_{j_1}(\sigma) b_{j_2} b_{k_2}
\end{align*}
In this sum we can replace the factor $2^{-|j_2 - k_2|/6}$ by the larger factor
$2^{-|k - k_2|/6}$, from which it is seen that the whole sum is controlled by
\begin{equation*}
2^{-\sigma k} b_k(\sigma) b_k
\sum_{k_2 < k - 80} 2^{-|k - k_2|/6} b_{k_2}
\lesssim 
2^{-\sigma k} b_k^2 b_k(\sigma)
\end{equation*}
\end{proof}


\subsection{Decomposing the magnetic potential}
\label{SS:DMP}
\label{SS:B}

We begin by introducing a paradifferential decomposition of the magnetic nonlinearity,
splitting it into two pieces. This decomposition depends upon a frequency
parameter $k \in \mathbf{Z}$, which we suppress in the notation;
this same $k$ will also be the output frequency whose behavior we are interested
in controlling.
The decomposition
also depends upon the frequency gap parameter $\varpi \in \mathbf{Z}_+$. 
How $\varpi$ is chosen and the exact role it plays are discussed in \S \ref{SS:App}.
There it is shown that $\varpi$ may be set equal to a sufficiently large
universal constant (independent of $\varepsilon, \varepsilon_1, k$, etc.).

Define $A_{\mathrm{lo \wedge lo}}$ as
\begin{equation*}
A_{m, \mathrm{lo \wedge lo}}(s)
:=
-\sum_{k_1, k_2 \leq k - \varpi}
\int_s^\infty \Im( \overline{P_{k_1} \psi_m} P_{k_2} \psi_s )(s^\prime) ds^\prime
\end{equation*}
and $A_{\mathrm{hi \lor hi}}$ as
\begin{equation*}
A_{m, \mathrm{hi \lor hi}}(s)
:=
-\sum_{\max\{k_1, k_2\} > k - \varpi}
\int_s^\infty \Im( \overline{P_{k_1} \psi_m} P_{k_2} \psi_s )(s^\prime) ds^\prime
\end{equation*}
so that
$A_m = A_{m, \mathrm{lo \wedge lo}} + A_{m, \mathrm{hi \lor hi}}$.
Similarly define $B_{\mathrm{lo \wedge lo}}$ as
\begin{equation*}
B_{m, \mathrm{lo \wedge lo}} 
:= - i \sum_{k_3} \left( \partial_\ell( A_{\ell, \mathrm{lo \wedge lo}} P_{k_3} \psi_m)
+
A_{\ell, \mathrm{lo \wedge lo}} \partial_\ell P_{k_3} \psi_m \right)
\end{equation*}
and
$B_{\mathrm{hi \lor hi}}$ as
\begin{equation*}
B_{m, \mathrm{hi \lor hi}} 
:= - i \sum_{k_3} \left( \partial_\ell( A_{\ell, \mathrm{hi \lor hi}} P_{k_3} \psi_m)
+
A_{\ell, \mathrm{hi \lor hi}} \partial_\ell P_{k_3} \psi_m \right)
\end{equation*}
so that
$B_m = B_{m, \mathrm{lo \wedge lo}} + B_{m, \mathrm{hi \lor hi}}$.

Our goal is to control $P_k B_m$ in $N_k(T)$.
We consider first $P_k B_{m, \mathrm{hi \lor hi}}$,
performing a trilinear Littlewood-Paley decomposition.
In order for frequencies $k_1, k_2, k_3$ to have an output in
this expression at a frequency $k$, we must have
$(k_1, k_2, k_3) \in Z_2(k) \cup Z_3(k) \cup Z_0(k)$,
where 
\begin{equation}\label{Zzero}
Z_0(k) := Z_1(k) \cap \{ (k_1, k_2, k_3) \in \mathbf{Z}^3 :
k_1, k_2 > k - \varpi \}
\end{equation}
and the other $Z_j(k)$'s are defined in (\ref{Z def}).
We apply Lemma \ref{L:Trilin1} to bound
$P_k B_{m, \mathrm{hi \lor hi}}$ in $N_k(T)$ by
\begin{equation*}
\begin{split}
\sum_{\substack{ (k_1,k_2,k_3) \in \\ Z_2(k) \cup Z_3(k)
\cup Z_0(k)}}
& \int_0^\infty
2^{\max\{k, k_3\}} C_{k,k_1,k_2,k_3} \lVert P_{k_1} \psi_x(s) \rVert_{F_{k_1}} \times \\
& \times \lVert P_{k_2} (D_\ell \psi_\ell(s)) \rVert_{F_{k_2}}
 \lVert P_{k_3} \psi_m(0) \rVert_{G_{k_3}} ds,
\end{split}
\end{equation*}
which, thanks to (\ref{PsiF}) and (\ref{APsiF}), is
controlled by
\begin{equation*}
\begin{split}
\sum_{\substack{ (k_1,k_2,k_3) \in \\ Z_2(k) \cup Z_3(k)
\cup Z_0(k)}}
&2^{\max\{k, k_3\}} C_{k,k_1,k_2,k_3} b_{k_1} b_{k_2} b_{k_3} \times \\
& \times \int_0^\infty
(1 + s2^{2k_1})^{-4} 2^{k_2} (s 2^{2k_2})^{-3/8}
(1 + s2^{2k_2})^{-2}
ds.
\end{split}
\end{equation*}
As
\begin{equation}
\int_0^\infty
(1 + s2^{2k_1})^{-4} 2^{k_2} (s 2^{2k_2})^{-3/8}
(1 + s2^{2k_2})^{-2}
ds
\lesssim
2^{-\max \{k_1, k_2\}},
\label{Ik1k2}
\end{equation}
we reduce to
\begin{equation}\label{MagBound}
\sum_{\substack{ (k_1,k_2,k_3) \in \\ Z_2(k) \cup Z_3(k)
\cup Z_0(k)}}
2^{\max\{k, k_3\} - \max\{k_1,k_2\}} 
C_{k,k_1,k_2,k_3}
b_{k_1} b_{k_2} b_{k_3}.
\end{equation}

To estimate $P_k B_{m, \mathrm{hi \lor hi}}$ on $Z_2 \cup Z_3$,
we apply Corollary \ref{C:MTE2} and use the energy dispersion
hypothesis.
As for $Z_0(k)$, we note that its cardinality $\lvert Z_0(k) \rvert$
satisfies $\lvert Z_0(k) \rvert \lesssim \varpi$ independently
of $k$. Hence for fixed $\varpi$ summing over this set
is harmless given sufficient energy dispersion. We obtain a bound of
\begin{equation}\label{Bsig0}
\lVert P_k \Bhi \rVert_{N_k(T)} \lesssim b_k^2 b_k \lesssim \cV b_k.
\end{equation}


Consider now the leading term $P_k B_{m, \mathrm{lo \wedge lo}}$.
Bounding this in $N_k$ with any hope of summing
requires the full strength of the decay that comes from the
local smoothing/maximal function estimates. 
Such bounds as are immediately at our disposal
(i.e., (\ref{bilin2 improv}) and (\ref{bilin5 improv}), however,
do not bring $B_{m, \mathrm{lo \wedge lo}}$ within the perturbative
framework, instead yielding a bound of the form
\begin{equation*}
\sum_{\substack{ k_1, k_2 \leq k - \varpi \\ \lvert k_3 - k \rvert \leq 4}} 
b_{k_1} b_{k_2} b_{k_3},
\end{equation*}
which is problematic since even $\sum_{j \ll k} c_j^2 \sim E_0^2 = O(1)$
for $k$ large enough.
This stands in sharp contrast with the small energy setting.

In the next section, however, we are able to capture enough
improvement in such estimates so as to barely bring
$B_{m, \mathrm{lo \wedge lo}}$ back within reach
of our bootstrap approach.

Finally, we need for $\sigma > 0$ an estimate analogous to (\ref{Bsig0}).
Returning to the proof of (\ref{MagBound}), we remark that any $b_{k_j}$
may be replaced by $2^{-\sigma k_j} b_{k_j}$; in order to obtain an analogue of
(\ref{Bsig0}), we must make replacements judiciously so as to retain summability.
In particular, for any $(k_1, k_2, k_3) \in Z_2(k) \cup Z_3(k) \cup Z_0(k)$,
we replace $b_{\kx}$ with $2^{-\sigma \kx} b_{\kx}(\sigma)$ so that (\ref{MagBound})
becomes
\[
\sum_{\substack{ (k_1,k_2,k_3) \in \\ Z_2(k) \cup Z_3(k)
\cup Z_0(k)}}
2^{\max\{k, k_3\} - \max\{k_1,k_2\}} 
C_{k,k_1,k_2,k_3}
b_{\kn} b_{\kd} 2^{-\sigma \kx} b_{\kx}(\sigma),
\]
where $\kn, \kd, \kx$ denote, respectively, the min, mid, and max of $\{k_1, k_2, k_3\}$.
Over the set $Z_2(k) \cup Z_3(k) \cup Z_0(k)$ (see (\ref{Z def}) and (\ref{Zzero}) for definitions),
we have $\kx \gtrsim k$, which guarantees summability due to straightforward 
modifications of Corollaries \ref{C:MTE1} and \ref{C:MTE2}.
Therefore
\begin{equation*}
\lVert P_k \Bhi \rVert_{N_k(T)} \lesssim b_k^2 2^{-\sigma k} b_k(\sigma),
\end{equation*}
which, combined with (\ref{Bsig0}) and the definition (\ref{calV def}) of $\cV$, implies
\begin{cor}\label{B bound corollary}
Assume $\sigma \in [0, \sigma_1 - 1]$.
The term $\Bhi$ satisfies the estimate
\begin{equation}
\lVert P_k \Bhi \rVert_{N_k(T)}
\lesssim
\cV 2^{-\sigma k} b_k(\sigma).
\end{equation}
\end{cor}


\subsection{Closing the gauge field bootstrap} \label{SS:OutlineConclusion}

We turn first to the completion of the proof of Theorem \ref{EnvelopeTheorem},
as we now have in place all of the estimates that we need to prove (\ref{Goal}).

Using the main linear estimate of Proposition \ref{MainLinearEstimate}
and the decomposition introduced in \S \ref{SS:DMP},
we obtain
\begin{equation}
\begin{split}
\lVert P_k \psi_m \rVert_{G_k(T)}
\lesssim& \;
\lVert P_k \psi_m(0) \rVert_{L^2_x} +
\lVert P_k V_m \rVert_{N_k(T)} \\
&\; + \lVert P_k \Bhi \rVert_{N_k(T)} +
\lVert P_k \Blo \rVert_{N_k(T)}.
\end{split}
\end{equation}
In \S\S \ref{SS:V}, \ref{SS:DMP} it is shown that $P_k V_m$
and $P_k \Bhi$ are perturbative in the sense that
\begin{equation*}
\lVert P_k V_m \rVert_{N_k(T)}
+
\lVert P_k \Bhi \rVert_{N_k(T)}
\lesssim \cV 2^{-\sigma k} b_k(\sigma),
\end{equation*}
To handle $P_k \Blo$, we first write
\begin{equation*}
P_k \Blo
=
- i \partial_\ell( \Alol P_k \psi_m)
+ R,
\end{equation*}
where $R$ is a perturbative remainder
(thanks to a slight modification of technical Lemma \ref{L:ls app RHS}).
Therefore
\begin{equation}
\lVert P_k \psi_m \rVert_{G_k(T)}
\lesssim
2^{-\sigma k} c_k(\sigma) + \cV 2^{-\sigma k} b_k(\sigma)
+ \lVert \partial_\ell( \Alol P_k \psi_m) \rVert_{N_k(T)}.
\label{psi G preliminary}
\end{equation}
Thus it remains to control
$- i \partial_\ell( \Alol P_k \psi_m)$,
which we expand as
\begin{equation}
\begin{split}
- i P_k \partial_\ell \sum_{ \substack{k_1, k_2 \leq k - \varpi \\ \lvert k_3 - k \rvert \leq 4}}
\int_0^\infty \Im(\overline{P_{k_1} \psi_\ell} P_{k_2} \psi_s)(s^\prime)
P_{k_3} \psi_m(0) ds^\prime,
\end{split}
\label{lo expanded}
\end{equation}
and whose $N_k(T)$ norm we denote by $\Nlo$.
In the $\sigma = 0$ case
the key is to apply Theorem \ref{T:BilinearStrichartz}
to $\overline{P_{k_1} \psi_\ell}(s^\prime)$ and $P_{k_3} \psi_m(0)$,
after first placing all of (\ref{lo expanded}) in $N_k(T)$ using
(\ref{bilin2 improv}).
We obtain
\begin{equation*}
\begin{split}
\Nlo
&\lesssim \; 
2^k \sum_{ \substack{k_1, k_2 \leq k - \varpi \\ \lvert k_3 - k \rvert \leq 4}}
2^{- \lvert k - k_2 \rvert / 2} 2^{- \lvert k_1 - k_3 \rvert / 2} 2^{-\max\{k_1, k_2\}} b_{k_2}
\left(c_{k_1} c_{k_3} + \cV^{1/2} b_{k_1} b_{k_3} \right) 
\\
&\lesssim \;
2^k \sum_{k_1, k_2 \leq k - \varpi} 2^{(k_1 + k_2)/2 - k} 2^{-\max\{k_1, k_2\}}
b_{k_2}(c_{k_1} c_k + \cV^{1/2} b_{k_1} b_k)
\end{split}
\end{equation*}
Without loss of generality we restrict the sum to $k_1 \leq k_2$:
\begin{equation*}
\sum_{k_1 \leq k_2 \leq k - \varpi} 2^{(k_1 - k_2)/2} b_{k_2}(c_{k_1} c_k + \cV^{1/2} b_{k_1} b_k)
\end{equation*}
Using the frequency envelope property to sum off the diagonal, we reduce to
\begin{equation*}
\Nlo
\lesssim
\sum_{j \leq k - \varpi} (b_j c_j c_k + \cV^{1/2} b_j^2 b_k).
\end{equation*}
Combining this with (\ref{psi G preliminary}) and the fact that
$R$ is perturbative,
we obtain
\begin{equation}\label{Sig0Boot}
b_k \lesssim c_k + \cV b_k +
\sum_{j \leq k - \varpi} (b_j c_j c_k + \cV^{1/2} b_j^2 b_k),
\end{equation}
which, in view of our choice of $\cV$, reduces to
\begin{equation*}
b_k \lesssim c_k + c_k \sum_{j \leq k - \varpi} b_j c_j .
\end{equation*}
Squaring and applying Cauchy-Schwarz yields
\begin{equation}
b_k^2 \lesssim  (1 +  \sum_{j \leq k - \varpi} b_j^2 ) c_k^2.
\label{SCS}
\end{equation}
Setting
\begin{equation*}
B_k := 1 + \sum_{j < k} b_j^2
\end{equation*}
in (\ref{SCS}) leads to
\begin{equation*}
B_{k+1} \leq B_k (1 + C c_k^2)
\end{equation*}
with $C > 0$ independent of $k$.
Therefore
\begin{equation*}
B_{k + m} 
\leq 
B_k \prod_{\ell = 1}^m (1 + C c_{k + \ell}^2)
\leq
B_k \exp(C \sum_{\ell = 1}^m c_{k + \ell}^2)
\lesssim_{E_0} B_k.
\end{equation*}
Since $B_k \to 1$ as $k \to -\infty$, we conclude
\begin{equation*}
B_k \lesssim_{E_0} 1
\end{equation*}
uniformly in $k$, so that, in particular,
\begin{equation}
\sum_{j \in \mathbf{Z}} b_j^2
\lesssim 1,
\label{b l2 bounded}
\end{equation}
which, joined with (\ref{SCS}), implies (\ref{Goal}).

The proof of (\ref{Goal2}) is almost an immediate consequence.
Instead of (\ref{Sig0Boot}), we obtain
\begin{equation*}
b_k(\sigma) \lesssim c_k(\sigma) + \cV b_k(\sigma)
+ \sum_{j \leq k - \varpi} (b_j c_j c_k(\sigma) + \cV^{1/2} b_j^2 b_k(\sigma)),
\end{equation*}
which suffices to prove (\ref{Goal2}) in view of (\ref{b l2 bounded}).


\subsection{De-gauging} \label{SS:Degauging}

The previous subsections overcome the most significant obstacles
encountered in
proving conditional global regularity.
All of the key estimates therein apply to the Schr\"odinger map system placed in the caloric
gauge, and a bootstrap argument is in fact run and closed at that level.  
This final subsection justifies the whole approach, 
showing how to transfer these results obtained at the gauge level back
to the underlying Schr\"odinger map itself.  

\begin{proof}[Proof of (\ref{Main Conclusion})]
To gain control over the derivatives $\partial_m \varphi$ in $L_t^\infty L_x^2$,
we utilize representation (\ref{Phi Frame Rep}) and perform a Littlewood-Paley
decomposition.  We only indicate how to handle the term $v \cdot \Re( \psi_m )$,
as the term $w \cdot \Im( \psi_m )$ may be handled similarly.  
Starting with
\begin{align}
P_k (v \Re(\psi_m) )
=&
\sum_{\lvert k_2 - k \rvert \leq 4} 
P_k( P_{\leq k - 5} v \cdot P_{k_2} \Re(\psi_m) ) + \nonumber \\
&\sum_{\substack{ \lvert k_1 - k \rvert \leq 4 \\ k_2 \leq k - 4}}
P_k (P_{k_1} v \cdot P_{k_2} \Re(\psi_m) ) + \nonumber \\
&\sum_{\substack{ \lvert k_1 - k_2 \rvert \leq 8 \\ k_1, k_2 \geq k - 4}}
P_k( P_{k_1} v \cdot P_{k_2} \Re(\psi_m) ),
\label{vp lp}
\end{align}
we proceed to bound each term in $L_t^\infty L_x^2$.

In view of the fact that $\lvert v \rvert \equiv 1$, 
the low-high frequency interaction is controlled by
\begin{align}
\sum_{\lvert k_2 - k \rvert \leq 4} \lVert P_k( P_{\leq k - 5} v \cdot P_{k_2} \Re(\psi_m) )
\rVert_{L_t^\infty L_x^2}
&\lesssim
\lVert P_{\leq k - 5} v \rVert_{L_{t,x}^\infty}
\lVert P_k \psi_m \rVert_{L_t^\infty L_x^2} \nonumber \\
&\lesssim
\lVert P_k \psi_m \rVert_{L_t^\infty L_x^2} \nonumber \\
&\lesssim
c_k.
\label{vp LH}
\end{align}
To control the high-low frequency interaction, we use 
H\"older's inequality, Bernstein's inequality, 
(\ref{c cons}) and Bernstein's inequality again, and finally
the bound (\ref{dv Bound}) along with summation rule (\ref{Sum 1}):
\begin{align}
\sum_{\substack{ \lvert k_1 - k \rvert \leq 4 \\ k_2 \leq k - 4}}
\lVert P_k (P_{k_1} v \cdot P_{k_2} \Re(\psi_m) ) \rVert_{L_t^\infty L_x^2}
&\lesssim
\sum_{\substack{ \lvert k_1 - k \rvert \leq 4 \\ k_2 \leq k - 4}}
\lVert P_{k_1} v \rVert_{L_t^\infty L_x^2} 
\lVert P_{k_2} \psi_m \rVert_{L_{t,x}^\infty} \nonumber \\
&\lesssim
\sum_{\substack{ \lvert k_1 - k \rvert \leq 4 \\ k_2 \leq k - 4}}
\lVert P_{k_1} v \rVert_{L_t^\infty L_x^2} \cdot
2^{k_2} \lVert P_{k_2} \psi_m \rVert_{L_t^\infty L_x^2} \nonumber \\
&\lesssim
\sum_{\substack{ \lvert k_1 - k \rvert \leq 4 \\ k_2 \leq k - 4}}
\lVert P_{k_1} \partial_x v \rVert_{L_t^\infty L_x^2} \cdot
2^{k_2 - k} c_{k_2} \nonumber \\
&\lesssim
c_k.
\label{vp HL}
\end{align}
To control the high-high frequency interaction, we use Bernstein's inequality,
Cauchy-Schwarz, Bernstein again, (\ref{dv Bound}), and finally (\ref{Sum 2}):
\begin{align}
\sum_{\substack{ \lvert k_1 - k_2 \rvert \leq 8 \\ k_1, k_2 \geq k - 4}}
\lVert P_k( P_{k_1} v \cdot P_{k_2} \Re(\psi_m) ) \rVert_{L_t^\infty L_x^2}
&\lesssim
\sum_{\substack{ \lvert k_1 - k_2 \rvert \leq 8 \\ k_1, k_2 \geq k - 4}}
2^k \lVert P_{k_1} v \cdot P_{k_2} \Re( \psi_m ) \rVert_{L_t^\infty L_x^1} \nonumber \\
&\lesssim
\sum_{\substack{ \lvert k_1 - k_2 \rvert \leq 8 \\ k_1, k_2 \geq k - 4}}
2^k \lVert P_{k_1} v \rVert_{L_t^\infty L_x^2} \lVert P_{k_2} \psi_m \rVert_{L_t^\infty L_x^2} \nonumber \\
&\lesssim
\sum_{\substack{ \lvert k_1 - k_2 \rvert \leq 8 \\ k_1, k_2 \geq k - 4}}
2^{k - k_1} \lVert P_{k_1} \partial_x v \rVert_{L_t^\infty L_x^2}
\lVert P_{k_2} \psi_m \rVert_{L_t^\infty L_x^2} \nonumber \\
&\lesssim
\sum_{k_2 \geq k - 4} 2^{k - k_2} c_{k_2} \nonumber \\
&\lesssim
c_k.
\label{vp HH}
\end{align}
Combining (\ref{vp LH}), (\ref{vp HL}), and (\ref{vp HH})
and applying them in (\ref{vp lp}), we obtain
\begin{equation*}
\lVert P_k (v \; \Re (\psi_m) ) \rVert_{L_t^\infty L_x^2} \lesssim c_k.
\end{equation*}
As the above calculation holds with $w$ in place of $v$, we conclude 
(recalling (\ref{Phi Frame Rep})) that
\begin{equation*}
\lVert P_k \partial_x \varphi \rVert_{L_t^\infty L_x^2} \lesssim c_k.
\end{equation*}
Hence (\ref{Main Conclusion}) holds for $\sigma = 0$.

Now we turn to the case $\sigma \in [0, \sigma_1 - 1]$.
By using Bernstein's inequality in (\ref{vp LH}) and (\ref{vp HH}), 
we may obtain
\begin{align}
\sum_{\lvert k_2 - k \rvert \leq 4} \lVert P_k( P_{\leq k - 5} v \cdot P_{k_2} \Re(\psi_m) )
\rVert_{L_t^\infty L_x^2}
&\lesssim
2^{-\sigma k} c_k(\sigma) \label{vp LHs} \\
\sum_{\substack{ \lvert k_1 - k_2 \rvert \leq 8 \\ k_1, k_2 \geq k - 4}}
\lVert P_k( P_{k_1} v \cdot P_{k_2} \Re(\psi_m) ) \rVert_{L_t^\infty L_x^2}
&\lesssim
2^{-\sigma k} c_k(\sigma),
\label{vp HHs}
\end{align}
as well as analogous estimates with $w$ in place of $v$.  
Such a direct argument, however, does not yield the analogue of (\ref{vp HL}).
We circumvent this obstruction as follows.
Let $\mathcal{C} \in (0, \infty)$ be the best constant for which
\begin{equation}
\lVert P_k \partial_x \varphi \rVert_{L_t^\infty L_x^2}
\leq
\mathcal{C} 2^{-\sigma k} c_k(\sigma)
\label{v bestC}
\end{equation}
holds for $\sigma \in [0, \sigma_1 - 1]$.
Such a constant exists by smoothness and the fact that the $c_k(\sigma)$
are frequency envelopes.
In view of definition (\ref{gamma Envelope}) and 
estimate (\ref{Hard Envelope Bounds}),
we similarly have
\begin{equation}
\lVert P_k \partial_x v(0) \rVert_{L_t^\infty L_x^2}
\lesssim 
\mathcal{C} 2^{-\sigma k} c_k(\sigma).
\label{vw strapbound}
\end{equation}
Using (\ref{vw strapbound}) in (\ref{vp HL}), we obtain
\begin{equation}
\sum_{\substack{ \lvert k_1 - k \rvert \leq 4 \\ k_2 \leq k - 4}}
\lVert P_k (P_{k_1} v \cdot P_{k_2} \Re(\psi_m) ) \rVert_{L_t^\infty L_x^2}
\lesssim
\mathcal{C} 2^{-\sigma k} c_k c_k(\sigma).
\label{vp HLs}
\end{equation}
From the representations (\ref{Phi Frame Rep}) and (\ref{vp lp}), 
and from the estimates (\ref{vp LHs}), (\ref{vp HHs}), and (\ref{vp HLs}), 
along with the analogous estimates for $w$, it follows that
\begin{equation*}
\lVert P_k \partial_x \varphi \rVert_{L_t^\infty L_x^2}
\lesssim
(1 + c_k \mathcal{C}) 2^{-\sigma k} c_k(\sigma).
\end{equation*}
In view of energy dispersion ($c_k \leq \varepsilon$)
and the optimality of $\mathcal{C}$ in (\ref{v bestC}), we conclude
\begin{equation*}
\mathcal{C} \lesssim 1 + \varepsilon \mathcal{C}
\end{equation*}
so that $\mathcal{C} \lesssim 1$.
Therefore
\begin{equation*}
\lVert P_k \partial_x^\sigma \partial_m \varphi \rVert_{L_t^\infty L_x^2}
\sim 2^{\sigma k} \lVert P_k \partial_m \varphi \rVert_{L_t^\infty L_x^2}
\lesssim c_k(\sigma),
\end{equation*}
which completes the proof of (\ref{Main Conclusion}).
\end{proof}

It will be convenient in certain arguments to use the 
weaker frequency envelope defined by
\begin{equation}
\tilde{b}_k = \sup_{k^\prime \in \mathbf{Z}} 2^{-\delta \lvert k - k^\prime \rvert}
\lVert P_{k^\prime} \psi_x \rVert_{L_{t,x}^4}.
\label{tb Envelope}
\end{equation}

\begin{proof}[Proof of Lemma \ref{equiv mt con}]
Let us first establish
\begin{equation*}
\sum_{k \in \mathbf{Z}} \lVert P_k \psi_x \rVert_{L_{t,x}^4}^2
\lesssim
\sum_{k \in \mathbf{Z}} \lVert P_k \partial_x \varphi \rVert_{L_{t,x}^4}^2.
\end{equation*}
We use (\ref{Derivative Field}), i.e.,
$\psi_m = v \cdot \partial_m \varphi + i w \cdot \partial_m \varphi$,
but for the sake of exposition only treat $v \cdot \partial_m \varphi$.
We start with the Littlewood-Paley decomposition
\begin{align*}
P_k \psi_m(0)
=&
\sum_{\lvert k_2 - k \rvert \leq 4}
P_k (P_{\leq k - 5}v \cdot P_{k_2} \partial_m \varphi) + \\
&
\sum_{ \substack{ \lvert k_1 - k \rvert \leq 4 \\ k_2 \leq k - 4} }
P_k (P_{k_1} v \cdot P_{k_2} \partial_m \varphi) + \\
&
\sum_{ \substack{ \lvert k_1 - k_2 \rvert \leq 8 \\ k_1, k_2 \geq k-4 } }
P_k (P_{k_1} v \cdot P_{k_2} \partial_m \varphi).
\end{align*}
In view of $\lvert v \rvert \equiv 1$, the $L_{t,x}^4$ norm of the low-high interaction is controlled
by $\tilde{b}_k$ (see (\ref{tb Envelope})).
To control the high-low interaction, we use H\"older's and Bernstein's inequalities
along with (\ref{dv Bound}):
\begin{align*}
\sum_{ \substack{ \lvert k_1 - k \rvert \leq 4 \\ k_2 \leq k - 4} }
\lVert P_k (P_{k_1} v \cdot P_{k_2} \partial_m \varphi \rVert_{L_{t,x}^4}
&\lesssim
\sum_{ \substack{ \lvert k_1 - k \rvert \leq 4 \\ k_2 \leq k - 4} }
\lVert P_{k_1} v \rVert_{L_t^\infty L_x^4} \cdot
\lVert P_{k_2} \partial_m \varphi \rVert_{L_t^4 L_x^\infty} \\
&\lesssim
\sum_{ \substack{ \lvert k_1 - k \rvert \leq 4 \\ k_2 \leq k - 4} }
2^{k_1 / 2} \lVert P_{k_1} v \rVert_{L_t^\infty L_x^2}
2^{k_2 / 2} \lVert P_{k_2} \partial_m \varphi \rVert_{L_{t,x}^4} \\
&\lesssim
\sum_{ \substack{ \lvert k_1 - k \rvert \leq 4 \\ k_2 \leq k - 4} }
2^{k_1} \lVert P_{k_1} v \rVert_{L_t^\infty L_x^2} \tilde{b}_k \\
&\lesssim \tilde{b}_k.
\end{align*}
To control the high-high interaction, we use Bernstein, H\"older, Bernstein again,
and (\ref{dv Bound}):
\begin{align*}
\sum_{ \substack{ \lvert k_1 - k_2 \rvert \leq 8 \\ k_1, k_2 \geq k-4 } }
\lVert P_k (P_{k_1} v \cdot P_{k_2} \partial_m \varphi) \rVert_{L_{t,x}^4}
&\lesssim
\sum_{ \substack{ \lvert k_1 - k_2 \rvert \leq 8 \\ k_1, k_2 \geq k-4 } }
2^{k/2} \lVert P_{k_1} v \cdot P_{k_2} \partial_m \varphi \rVert_{L_t^4 L_x^2} \\
&\lesssim
\sum_{ \substack{ \lvert k_1 - k_2 \rvert \leq 8 \\ k_1, k_2 \geq k-4 } }
2^{k/2} \lVert P_{k_1} v \rVert_{L_t^\infty L_x^4} \lVert P_{k_2} \partial_m \varphi \rVert_{L_{t,x}^4} \\
&\lesssim
\sum_{ \substack{ \lvert k_1 - k_2 \rvert \leq 8 \\ k_1, k_2 \geq k-4 } }
2^{(k + k_1)/2} \lVert P_{k_1} v \rVert_{L_t^\infty L_x^2} \lVert P_k \partial_m \varphi \rVert_{L_{t,x}^4} \\
&\lesssim
\sum_{ \substack{ \lvert k_1 - k_2 \rvert \leq 8 \\ k_1, k_2 \geq k-4 } }
2^{(k - k_1)/2} \lVert P_{k_1} \partial_x v \rVert_{L_t^\infty L_x^2}
\lVert P_{k_2} \partial_m \varphi \rVert_{L_{t,x}^4} \\
&\lesssim
\sum_{k_2 \geq k - 4} 2^{(k - k_2)/4} \tilde{b}_{k_2} \\
&\lesssim
\tilde{b}_k.
\end{align*}
Therefore
\begin{equation*}
\lVert P_k \psi_m(0) \rVert_{L_{t,x}^4}
\lesssim
\tilde{b}_k
\end{equation*}
and
\begin{equation*}
\sum_{k \in \mathbf{Z}} \lVert P_k \psi_m(0) \rVert_{L_{t,x}^4}^2
\lesssim
\sum_{k \in \mathbf{Z}} \tilde{b}_k^2
\sim
\sum_{k \in \mathbf{Z}} \lVert P_k \partial_m \varphi(0) \rVert_{L_{t,x}^4}^2.
\end{equation*}
By using (\ref{Phi Frame Rep}), creating an $L^4$ frequency envelope for
$P_k \partial_m \varphi(0)$,
and reversing the roles of $\psi_\alpha$ and
$\partial_\alpha \varphi$ in the preceding argument,
we conclude the reverse inequality
\begin{equation*}
\sum_{k \in \mathbf{Z}} \lVert P_k \partial_m \varphi(0) \rVert_{L_{t,x}^4}^2
\lesssim
\sum_{k \in \mathbf{Z}} \lVert P_k \psi_m(0) \rVert_{L_{t,x}^4}^2.
\end{equation*}
\end{proof}

%% file: LocalSmoothingBilinearStrichartz.tex
\section{Local smoothing and bilinear Strichartz} \label{S:Local smoothing}

The main goal of this section is to establish the improved bilinear Strichartz estimate
of Theorem \ref{T:BilinearStrichartz}.
As a by-product we also obtain the frequency-localized 
local smoothing estimate of Theorem \ref{ls thm}.

Our approach is to first establish abstract local smoothing and bilinear Strichartz estimates
for solutions to certain magnetic nonlinear Schr\"odinger equations.
These are in the spirit of \cite{PlVe09, PlVe, Tls}.
We shall then apply these to Schr\"odinger maps, in particular to the paralinearized 
derivative field equations written with respect to the caloric gauge.

We introduce some notation.
Let $I_k(\mathbf{R}^d)$ denote the set $\{ \xi \in \Rd : \lvert \xi \rvert \in [-2^{k-1}, 2^{k+1}] \}$ and
$I_{(-\infty, k]} := \bigcup_{j \leq k} I_j$.
For a $d$-vector-valued function $B = ( B_\ell )$ on $\Rd$ with real entries,
define the magnetic Laplacian $\Delta_B$, acting on complex-valued
functions $f$, via
\begin{equation}
\Delta_B f := (\partial_x + i B)((\partial_x + i B) f) = 
\Delta f + i (\partial_\ell B_\ell) f + 2i B_\ell \partial_\ell f - B_\ell^2 f.
\label{Magnetic}
\end{equation}
For a unit vector $\e \in \mathbf{S}^{d-1}$, denote by $\{ x \cdot \e = 0 \}$ the
orthogonal complement in $\Rd$ of the span of $\e$, equipped
with the induced measure. Given $\e$, we can construct a
positively oriented orthonormal basis $\e, \e_1, \ldots, \e_{d-1}$ of $\Rd$
so that $\e_1, \ldots, \e_{d-1}$ form an orthonormal basis for $\{ x \cdot \e = 0 \}$.
For complex-valued functions $f$ on $\Rd$,
define $E_{\e}(f) : \mathbf{R} \to \mathbf{R}$ as
\begin{equation}
E_{\e}(f)(x_0) :=
\int_{x \cdot \e = 0} \lvert f \rvert^2 dx^\prime
=
\int_{\mathbf{R}^{d-1}} 
\lvert f(x_0 \e + x_j \e_j) \rvert^2
dx^\prime,
\label{Ee def}
\end{equation}
where the implicit sum runs over $1, 2, \ldots, d-1$, and
$dx^\prime$ is standard $d-1$-dimensional Lebesgue measure.
We also adopt the following notation for this section:
for $z, \zeta$ complex, 
\begin{equation*}
z \wedge \zeta := z \bar{\zeta} - \bar{z} \zeta
= 2i \Im(z \bar{\zeta}).
\end{equation*}

\subsection{The key lemmas}

\begin{lem}[Abstract almost-conservation of energy]
Let $d \geq 1$ and
$\e \in \mathbf{S}^{d-1}$. Let $v$ be a $C^\infty_t(H^\infty_x)$ function 
on $\mathbf{R}^d \times [0, T]$ solving
\begin{equation}
(i \partial_t + \Delta_\cA) v = \Lambda_v
\label{Abstract NLS}
\end{equation}
with initial data $v_0$.
Take $\cA_\ell$ to be real-valued, smooth, and bounded,
with
$\Delta_{\cA}$ defined via (\ref{Magnetic}).
Then
\begin{equation}
\lVert v \rVert^2_{L^\infty_t L^2_x} 
\leq
\lVert v_0 \rVert_{L^2_x}^2 + 
\left\lvert
\int_0^T \int_{\mathbf{R}^d}
v \wedge \Lambda_v dx dt
\right\rvert.
\label{Abstract energy}
\end{equation}
\label{L:Abstract energy}
\end{lem}
\begin{proof}
We begin with
\begin{equation*}
\frac{1}{2} \partial_t \int \lvert v \rvert^2 dx 
=
\int \Im( \bar{v} \partial_t v ) dx,
\end{equation*}
which may equivalently be written as
\begin{equation*}
i \partial_t \int \lvert v \rvert^2 dx 
=
- \int v \wedge i \partial_t v dx.
\end{equation*}
Substituting from (\ref{Abstract NLS}) yields
\begin{equation*}
i \partial_t \int \lvert v \rvert^2 dx
= \int v \wedge \left( \Delta_\cA v
- \Lambda_v \right) dx.
\end{equation*}
Expanding $\Delta_\cA$ using (\ref{Magnetic})
and using the straightforward relations
\begin{equation*}
\partial_{\ell}(v \wedge i \cA_\ell v) =
v \wedge i(\partial_\ell \cA_\ell) v
+ v \wedge 2i \cA_\ell \partial_\ell v
\end{equation*}
and
\begin{equation*}
\partial_\ell (v \wedge \partial_\ell v) = v \wedge \Delta v,
\end{equation*}
we get
\begin{align*}
i \partial_t \int \lvert v \rvert^2 dx
=& \; \int \partial_\ell (v \wedge \partial_\ell v) dx
+\int \partial_\ell (v \wedge i \cA_\ell v) dx 
\\
& \; - \int v \wedge \cA_\ell^2 v dx
- \int v \wedge \Lambda_v dx.
\end{align*}
The first two terms on the right hand side vanish upon integration in $x$;
the third is equal to zero because $\cA_\ell^2$ is real.
Integrating in time and taking absolute values
therefore yields
\begin{equation*}
\left\vert \int_{\Rd} \lvert v(T^\prime) \rvert^2 - \lvert v_0 \rvert^2 dx \right\rvert
=
\left\lvert \int_0^{T^\prime} \int_{\Rd} v \wedge \Lambda_v dx dt \right\rvert
\end{equation*}
 for any time $T^\prime \in (0, T]$.
\end{proof}

\begin{lem}[Local smoothing preparation]
Let $d \geq 1$ and
$\e \in \mathbf{S}^{d-1}$. 
Let $j, k \in \mathbf{Z}$ and $j = k + O(1)$.
Let $\varepsilon_m > 0$ be a small positive number
such that $\varepsilon_m 2^{O(1)} \ll 1$.
Let $v$ be a $C^\infty_t(H^\infty_x)$ function 
on $\mathbf{R}^d \times [0, T]$ solving
\begin{equation}
(i \partial_t + \Delta_\cA) v = \Lambda_v,
\end{equation}
where
$\cA_\ell$
is real-valued, smooth, and
satisfies the estimate
\begin{equation}
\lVert \cA \rVert_{L^\infty_{t,x}} \leq \varepsilon_m 2^k.
\label{A L-infinity}
\end{equation}
The solution $v$ is assumed to have (spatial) frequency support in $I_k$, 
with the additional constraint that 
$\e \cdot \xi \in [2^{j-1}, 2^{j+1}]$
for all $\xi$ in the support of $\hat{v}$.
Then
\begin{equation}
2^j \int_0^T E_{\e}(v) dt
\lesssim 
\lVert v \rVert_{L^\infty_t L^2_x}^2 + 
\left\lvert \int_0^T \int_{x \cdot \e \geq 0} v \wedge \Lambda_v dx dt \right\rvert 
+ 2^j \int_0^T E_{\e}(v + i 2^{-j} \partial_\e v) dt.
\label{LS 1}
\end{equation}
\label{L:LS prep}
\end{lem}
\begin{proof}
We begin by introducing
\begin{equation*}
M_\e(t) :=  \int_{x \cdot \e \geq 0} \lvert v(x, t) \rvert^2 dx.
\end{equation*}
Then
\begin{equation}
0 \leq M_\e(t) \leq \lVert v(t) \rVert_{L_x^2(\Rd)}^2 \leq
\lVert v \rVert_{L_t^\infty L_x^2([-T, T] \times \Rd)}^2.
\label{M bound}
\end{equation}
Differentiating in time yields
\begin{align*}
i \dot{M}_\e(t) 
&= 
\int_{x \cdot \e \geq 0}
v \wedge (i \partial_t v) dx \\
&=
\int_{x \cdot \e \geq 0}
v \wedge 
(\Delta_\cA v
- \Lambda_v)
dx,
\end{align*}
which may be rewritten as
\begin{equation}
i \dot{M}_\e(t) = \int_{x \cdot \e \geq 0}
\partial_\ell (v \wedge (\partial_\ell + i \cA_\ell)v) dx
- \int_{x \cdot \e \geq 0} v \wedge \Lambda_v dx.
\label{Mdot}
\end{equation}
By integrating by parts,
\begin{equation*}
\int_{x \cdot \e \geq 0}
\partial_\ell (v \wedge (\partial_\ell + i \cA_\ell)v) dx
=
-\int_{x \cdot \e = 0} v \wedge 
(\partial_{\e} v + i \e \cdot \cA v) dx^\prime,
\end{equation*}
and therefore (\ref{Mdot}) may be rewritten as
\begin{equation}
-\int_{x \cdot \e = 0} v \wedge 
(\partial_{\e} v + i \e \cdot \cA v) dx^\prime
= 
i \dot{M}_\e(t) + 
\int_{x \cdot \e \geq 0} v \wedge \Lambda_v dx.
\label{Mdot2}
\end{equation}
On the one hand, we have the heuristic that
$\partial_\e v \approx i 2^j v$ since $v$ has localized frequency support.
On the other hand,
since $\cA$ is real-valued, we have
\begin{equation}
\int_0^T \int_{x \cdot \e = 0} v \wedge
 i \e \cdot \cA v dx^\prime dt=
2 \int_0^T \int_{x \cdot \e = 0} \e 
\cdot \cA \lvert v \rvert^2 dx^\prime dt
\label{A identity}
\end{equation}
and hence by assumption (\ref{A L-infinity}) also
\begin{equation}
\int_0^T \int_{x \cdot \e = 0} \lvert \cA \rvert \lvert v \rvert^2 dx^\prime dt
\leq 
\varepsilon_m 
2^k \int_0^T \int_{x \cdot \e = 0} \lvert v \rvert^2 dx^\prime dt.
\label{Gap Bound}
\end{equation}
Together these facts motivate rewriting $v \wedge \partial_\e v$ as
\begin{equation}
v \wedge \partial_\e v = 2 \cdot i 2^j \lvert v \rvert^2
+ v \wedge (\partial_\e v - i 2^j v).
\label{v identity}
\end{equation}
Using (\ref{A identity}), (\ref{v identity}), and
the bounds (\ref{Gap Bound}) and (\ref{M bound}) in (\ref{Mdot2}),
we obtain
by time-integration that
\begin{equation*}
\begin{split}
(1 - \varepsilon_m 2^{k - j})
2^{j}
 \int_0^T E_\e(v) dt
\leq& \;
\lVert v \rVert_{L^\infty_t L^2_x}^2 + 
\left\lvert \int_0^T \int_{x \cdot \e \geq 0} v \wedge \Lambda_v dx dt \right\rvert \\
&+ 2 \cdot 2^j \int_0^T \int_{x \cdot \e = 0}
\lvert v + i 2^{-j} \partial_\e v \rvert \lvert v \rvert dx^\prime dt.
\end{split}
\end{equation*}
Applying Cauchy-Schwarz to the last term yields
\begin{equation*}
2^j \int_0^T \int_{x \cdot \e = 0} \lvert v + i 2^{-j} \partial_\e v \rvert \lvert v \rvert dx^\prime dt
\leq 
8 \cdot 2^j \int_0^T E_{\e}(v + i 2^{-j} \partial_\e v) dt 
+ \frac{1}{8}  \cdot 2^j \int_0^T E_{\e}(v) dt.
\end{equation*}
Therefore (\ref{LS 1}).
\end{proof}

We now describe the constraints on the nonlinearity that we shall require in the
abstract setting
\begin{defin}
Let $\cP$ be a fixed finite subset of $\{ 1 < p < \infty \}$.
A bilinear form $B(\cdot, \cdot)$
is said to be \emph{adapted} 
to $\cP$ provided it measures its arguments
in Strichartz-type spaces, the estimate
\begin{equation*}
\left\lvert 
\int_0^T \int_{\Rd} f \wedge g dx dt \right\rvert
\lesssim
B(f, g)
\end{equation*}
holds for all complex-valued functions $f, g$ on $\Rd \times [0, T]$,
Bernstein's inequalities hold in both arguments of $B$, and
these arguments are measured in $L^p_x$
only for $p \in \cP$. 
Given $B(\cdot, \cdot)$ and $\e \in \mathbf{S}^{d-1}$,
we define $B_\e(\cdot, \cdot)$ via
\begin{equation*}
B_\e(f, g) := B(f, \chi_{\{ x \cdot \e \geq 0 \}} g).
\end{equation*}
\label{D:good}
\end{defin}

\begin{defin}
Let $\e \in \mathbf{S}^{d-1}$ and
let $\cA_\ell$ be real-valued and smooth.
Let $v$ be a $C_t^\infty(H_x^\infty)$ function on $\mathbf{R}^d \times [0, T]$
solving
\begin{equation*}
(i \partial_t + \Delta_\cA) v = \Lambda_{v}.
\end{equation*}
Assume $v$ is (spatially) frequency-localized to $I_k$
with the additional constraint that 
$\e \cdot \xi \in [2^{j-1}, 2^{j+1}]$
for all $\xi$ in the support of $\hat{v}$.
Define a sequence of functions $\{ v^{(m)} \}_{m = 1}^\infty$
by setting $v^{(1)} = v$ and
\begin{equation*}
v^{(m+1)} := v^{(m)} + i 2^{-j} \partial_\e v^{(m)}.
\end{equation*}
By (\ref{Magnetic}) and the Leibniz rule,
\begin{equation*}
(i \partial_t + \Delta_\cA) v^{(m)} = \Lambda_{v^{(m)}},
\end{equation*}
where
\begin{equation*}
\Lambda_{v^{(m)}} 
:=
(1 + i 2^{-j} \partial_\e) \Lambda_{v^{(m-1)}}
+ i2^{-j}( i \partial_\e \partial_\ell A_\ell - \partial_\e A_\ell^2) v^{(m-1)}
- 2^{-j + 1} (\partial_\e A_\ell) \partial_\ell v^{(m-1)}.
\end{equation*}
The sequence $\{ v^{(m)} \}_{m=1}^\infty$ is called the
\emph{derived sequence} corresponding to $v$.

Suppose we are given a form $B$ adapted to $\cP$.
The derived sequence is
said to be \emph{controlled} with respect to $B_\e$
provided that 
$B_\e(v^{(m)}, \Lambda_{v^{(m)}}) < \infty$
for each $m \geq 1$.
\label{D:improving}
\end{defin}
We remark that if the derived sequence $\{ v^{(m)} \}_{m=1}^\infty$
of $v$ is controlled, then for all $\ell \geq 1$, the
derived sequences $\{ v^{(m)} \}_{m = \ell}^\infty$ are also controlled.

\begin{thm}[Abstract local smoothing]
Let $d \geq 1$ and
$\e \in \mathbf{S}^{d-1}$.  
Let $j, k \in \mathbf{Z}$ and $j = k + O(1)$.
Let $\varepsilon_m > 0$ be a small positive number
such that $\varepsilon_m 2^{O(1)} \ll 1$.
Let $\eta > 0$. Let $\cP$ be a fixed finite subset of $(1, \infty)$
with $2 \in \cP$, and let $B_\e$ be a form adapted to $\cP$.
Let $v$ be a $C^\infty_t(H^\infty_x)$ function 
on $\mathbf{R}^d \times [0, T]$ solving
\begin{equation}
(i \partial_t + \Delta_\cA) v = \Lambda_v,
\label{Magnetic NLS}
\end{equation}
where
$\cA_\ell$
is real-valued, smooth, has spatial Fourier support in $I_{(-\infty, k]}$, and
satisfies the estimate
\begin{equation}
\lVert \cA \rVert_{L^\infty_{t,x}} \leq \varepsilon_m 2^k.
\end{equation}
The solution $v$ is assumed to have (spatial) frequency support in $I_k$.
We take $\Lambda_v$ to be frequency-localized to $I_{(-\infty, k]}$.

Assume moreover that 
\begin{equation}
\e \cdot \xi \in [(1 - \eta)2^j, (1 + \eta)2^j]
\label{small support}
\end{equation}
for all $\xi$ in the support of $\hat{v}$.

If the derived sequence of $v$ is controlled with
respect to $B_\e$, then
there exist $\eta^* > 0$ such that, 
for all $0 \leq \eta < \eta^*$,
the local smoothing estimate
\begin{equation}
2^j \int_0^T E_{\e}(v) dt 
\lesssim
\lVert v \rVert_{L^\infty_t L^2_x}^2
+ B_\e(v, \Lambda_v)
\label{Abstract LS}
\end{equation}
holds uniformly in $T$ and $j = k + O(1)$.
\label{T:Abstract LS}
\end{thm}
\begin{proof}

The foundation for proving (\ref{Abstract LS}) is 
(\ref{LS 1}),
which for an adapted form $B_\e$ implies
\begin{equation}
2^j \int_0^T E_{\e}(v) dt
\lesssim 
\lVert v \rVert_{L^\infty_t L^2_x}^2 + 
B_\e(v, \Lambda_v)
+ 2^j \int_0^T E_{\e}(v + i 2^{-j} \partial_\e v) dt.
\label{LS 2}
\end{equation}
Therefore our goal is control the last term in (\ref{LS 2}).
This we do using a bootstrap argument that hinges
upon the fact that $\tilde{v} := v + i 2^{-j} \partial_\e  v$
is the second term in the derived sequence of $v$, and
that being ``controlled" is an inherited property (in the sense
of the comments following Definition \ref{D:improving}).

By Bernstein's and H\"older's inequalities, we have
\begin{equation*}
2^j \int_0^T E_\e(v) dt
\lesssim
2^{2j} T \lVert v \rVert_{L_t^\infty L_x^2}^2.
\end{equation*}
for any $v$.
For fixed $T > 0$ and $k \in \mathbf{Z}$, let $K_{T, k} \geq 1$ 
be the best constant for which the inequality
\begin{equation}
2^j \int_0^T E_{\e}(v) dt 
\leq K_{T, k}
\left(
\lVert v \rVert_{L^2_x}^2 + B_\e(v, \Lambda_{v})
\right)
\label{BestK}
\end{equation}
holds
for all controlled sequences.
Applying (\ref{BestK}) to $\tilde{v}$ results in
\begin{equation}
2^j \int_0^T E_\e(\tilde{v}) dt
\leq
K_{T, k} 
\left(
\lVert \tilde{v} \rVert_{L^2_x}^2 + B_\e(\tilde{v}, \Lambda_{\tilde{v}})
\right),
\label{K v tilde}
\end{equation}
and thus we seek to control norms of $\tilde{v}$ in terms of those
of $v$.

Let $\widetilde{P}_k, \widetilde{P}_{j, \e}$ denote slight fattenings
of the Fourier multipliers $P_k, P_{j, \e}$.
On the one hand, Plancherel implies
\begin{equation}
\lVert 
(1 + i 2^{-j} \partial_\e) \widetilde{P}_{j, \e} \widetilde{P}_k
\rVert_{L^2_x \to L^2_x} \lesssim \eta.
\label{Plancherel}
\end{equation}
On the other hand, Bernstein's inequalities imply
\begin{equation*}
\lVert 
(1 + i 2^{-j} \partial_\e) \widetilde{P}_{j, \e} \widetilde{P}_k
\rVert_{L^p_x \to L^p_x} \lesssim 1,
\quad \quad 1 \leq p \leq \infty.
\end{equation*}
Therefore it follows from Riesz-Thorin interpolation that
\begin{displaymath}
\lVert 
(1 + i 2^{-j} \partial_\e) \widetilde{P}_{j, \e} \widetilde{P}_k
\rVert_{L^p_x \to L^p_x}
\lesssim
\left\{
\begin{array}{ll}
\eta^{2/p}  & 2 \leq p < \infty \\
\eta^{2 - 2/p} & 1 < p \leq 2.
\end{array}
\right.
\end{displaymath}
Restricting to $p \in \cP$, we conclude that there exists
a $q > 0$ such that
\begin{equation}
\lVert 
(1 + i 2^{-j} \partial_\e) \widetilde{P}_{j, \e} \widetilde{P}_k
\rVert_{L^p_x \to L^p_x}
\lesssim \eta^q
\label{RTI}
\end{equation}
for all $p \in \cP$ and all $\eta$ small enough.

Applying (\ref{RTI}) and Bernstein to $\tilde{v}$ yields
\begin{equation*}
\lVert \tilde{v} \rVert_{L^2_x}
\lesssim
\eta^q
\lVert v \rVert_{L^2_x},
\quad\quad
B_\e(\tilde{v}, \Lambda_{\tilde{v}})
\lesssim
\eta^q
B_\e(v, \Lambda_v),
\end{equation*}
which, combined with (\ref{K v tilde}) and (\ref{LS 2}),
leads to
\begin{equation*}
2^j \int_0^T E_\e(v) dt
\lesssim
(1 + \eta^q K_{T, k})
\left(
\lVert v \rVert_{L^\infty_t L^2_x}^2 +
B_\e(v, \Lambda_v)
\right).
\end{equation*}
As $K_{T, k}$ is the best constant for which (\ref{BestK}) holds,
it follows that
\begin{equation*}
K_{T, k} \lesssim 1 + \eta^q K_{T, k}
\end{equation*}
and hence that $K_{T, k} \lesssim 1$ for $\eta$ small enough.
\end{proof}

\begin{cor}
Given the assumptions of Theorem \ref{T:Abstract LS}, it holds that
\begin{equation*}
2^j \int_0^T E_{\e}(v) dt 
\lesssim
\lVert v_0 \rVert_{L^2_x}^2
+ B(v, \Lambda_v)
+ B_\e(v, \Lambda_v)
\end{equation*}
\label{C:Abstract LS}
\end{cor}
\begin{proof}
This is an immediate consequence of Theorem \ref{T:Abstract LS} and
Lemma \ref{L:Abstract energy}.
\end{proof}

\begin{cor}[Abstract bilinear Strichartz]
Let $d \geq 1$ and $\e \in \mathbf{S}^{d-1}$. 
Set $\tilde{\e} = (-\e, \e) / \sqrt{2}$.
Let $j, k \in \mathbf{Z}$ and $j = k + O(1)$.
Let $\varepsilon_m > 0$ be a small positive number
such that $\varepsilon_m 2^{O(1)} \ll 1$.
Let $\eta > 0$. Let $\cP$ be a fixed
finite subset of $(1, \infty)$ with $2 \in \cP$, and let $B_{\tilde{\e}}$
be a form that is adapted to $\cP$.

Let $w(x, y)$ be a $C^\infty_t(H^\infty_{x,y})$ function on 
$\mathbf{R}^{2d} \times [0, T]$,equal to $w_0$ at $t=0$
and
solving
\begin{equation*}
(i \partial_t + \Delta_\cA) w = \Lambda_w,
\end{equation*}
where $\cA_{k^\prime}$ is real-valued, smooth, has spatial Fourier
support in $I_{(-\infty, k]}$, and satisfies the estimate
\begin{equation*}
\lVert \cA \rVert_{L^\infty_{t,x,y}} \leq \varepsilon_m 2^k.
\end{equation*}

Assume $w$ has (spatial) frequency support in $I_k$ and that
\begin{equation*}
\tilde{\e} \cdot \xi \in [(1 - \eta) 2^j, (1 + \eta) 2^j]
\end{equation*}
for all $\xi$ in the support of $\hat{w}$.
Take $\Lambda_w$ to be frequency-localized to $I_{(-\infty, k]}$.

Suppose that $w(x, y)$ admits a decomposition
$w(x, y) = u(x) v(y)$, where $u$ has frequency support
in $I_\ell$, $\ell \ll k$. Use $u_0, v_0$ to denote $u(t = 0), v(t = 0)$.
If the derived sequence of $w$ is controlled with respect to
$B_{\tilde{\e}}$, then
\begin{equation}
\lVert u v \rVert_{L^2_{t,x}}^2
\lesssim
2^{\ell(d-1)} 2^{-j}
\left( \lVert u_0 \rVert_{L^2_x}^2 \lVert v_0 \rVert_{L^2_x}^2
+ B(w, \Lambda_w) + B_{\tilde{\e}}(w, \Lambda_w) \right)
\label{ABS}
\end{equation}
\label{C:ABS}
uniformly in $T$ and $j = k + O(1)$ provided $\eta$ is small enough.
\end{cor}
\begin{proof}
Taking into account that
\begin{equation*}
\lVert w_0 \rVert_{L^2_{x,y}}
=
\lVert u_0 \rVert_{L^2_x} \lVert v_0 \rVert_{L^2_x},
\end{equation*}
we apply Corollary \ref{C:Abstract LS} to $w$ at $(x, y) = 0$:
\begin{equation}
2^j \int_0^T E_{\tilde{\e}}(w) dt
\lesssim
\lVert u_0 \rVert_{L^2_x}^2 \lVert v_0 \rVert_{L^2_x}^2 +
B(w, \Lambda_w) +
B_{\tilde{\e}}(w, \Lambda_w).
\label{ls for w}
\end{equation}

We complete $(-\e, \e) / \sqrt{2}$ to a basis as follows:
\begin{equation*}
(-\e, \e)/ \sqrt{2}, 
(0, \e_1), \ldots, (0, \e_{d-1}),
(\e, \e)/ \sqrt{2}, 
(\e_1, 0), \ldots, (\e_{d-1}, 0).
\end{equation*}
On the one hand,
$E_{\tilde{\e}}(w)(0)$ is 
by definition (see (\ref{Ee def})) equal to
\begin{equation*}
\int_{\mathbf{R}}
\int_{\mathbf{R}^{2d-2}}
\lvert u(0 \cdot \e + r \e + x_j \e_j, t)
v(0 \cdot \e + r \e + y_j \e_j, t)
\rvert^2
dx^\prime dy^\prime dr.
\end{equation*}
We rewrite it as
\begin{equation}
\int_{\mathbf{R}}
\int_{\mathbf{R}^{d-1}} 
\lvert v(r \e + y_j \e_j, t) \rvert^2 dy^\prime
\int_{\mathbf{R}^{d-1}} 
\lvert u(r\e + x_j \e_j, t) \rvert^2 
dx^\prime
dr.
\label{Ee w}
\end{equation}
On the other hand,
\begin{equation*}
\begin{split}
\lVert u v \rVert_{L^2_y}^2
&= 
\int_{\Rd}
\lvert u(y, t) \rvert^2 \lvert v(y, t) \rvert^2 dy \\
&=
\int_{\mathbf{R}}
\int_{\mathbf{R}^{d-1}}
\lvert u(r \e + y_j \e_j) \rvert^2
\lvert v(r \e + y_j \e_j) \rvert^2
dy^\prime dr,
\end{split}
\end{equation*}
and by applying Bernstein to $u$ in the $y^\prime$ variables, we obtain
\begin{equation}
\lVert u v \rVert_{L^2_y}^2
\lesssim 2^{\ell(d - 1)}
\int_{\mathbf{R}}
\int_{\mathbf{R}^{d-1}}
\lvert v (r \e + y_j \e_j) \rvert^2 dy^\prime
\int_{\mathbf{R}^{d-1}}
\lvert u(r \e + x_j \e_j ) \rvert^2 dx^\prime
dr.
\label{uv Bernstein}
\end{equation}
Together (\ref{uv Bernstein}), (\ref{Ee w}), and (\ref{ls for w})
imply (\ref{ABS}).
\end{proof}

\subsection{Applying the abstract lemmas}
\label{SS:App}
We would like to apply the abstract estimates just developed
to the evolution equation (\ref{NLS}). We work in the caloric gauge
and adopt the magnetic potential decomposition introduced in
\S \ref{SS:DMP}. 
Throughout we take $\cV$ as defined in (\ref{calV def}).

Our starting point is the equation
\begin{equation}
(i \partial_t + \Delta) \psi_m
= \Blo + \Bhi + V_m.
\label{app start}
\end{equation}
Applying Fourier multipliers
$P_k$, $P_{j, \theta} P_k$, or variants thereof, we easily obtain corresponding
evolution equations for $P_k \psi_m$, $P_{j, \theta} P_k$, etc.
In rewriting a projection $P$ of (\ref{app start}) in the form (\ref{Abstract NLS}),
evidently $\Delta_\cA \psi_m$ should somehow come from 
$\Delta P \psi_m - P \Blo$, whereas $P \Bhi + P V_m$ ought to constitute the leading
part of the nonlinearity $\Lambda$. Fourier multipliers $P$, however, do not commute
with the connection coefficients $A$, and therefore in order to use the 
abstract machinery we must first track and control certain commutators.
Toward this end we adopt some notation from \cite{T2}.

Following \cite[\S 1]{T2}, we use
$L_O(f_1, \dots, f_m)(s, x, t)$ to denote any multi-linear
expression of the form
\begin{equation*}
\begin{split}
L_O(f_1, &\dots, f_m)(s, x, t) \\
&:= \int K(y_1, \ldots, y_{M(c)} f_1(s, x - y_1, t) \ldots
f_m(s, x - y_{M(c)}, t) dy_1 \ldots dy_{M(c)},
\end{split}
\end{equation*}
where the kernel $K$ is a measure with bounded mass
(and $K$ may change from line to line). Moreover, the kernel
of $L_O$ does not depend upon the index $\alpha$.
Also, we extend this notation to vector or matrices by making
$K$ into an appropriate tensor.
The expression $L_O(f_1, \dots, f_m)$ may be thought
of as a variant of $O(f_1, \dots, f_m)$. It obeys two key
properties. The first is simple consequence of Minkowski's
inequality (e.g., see \cite[Lemma 1]{T2}).
\begin{lem}
Let $X_1, \ldots, X_m, X$ be spatially translation-invariant
Banach spaces such that the product estimate
\begin{equation*}
\lVert f_1 \cdots f_m \rVert_{X}
\leq
C_0 \lVert f_1 \rVert_{X_1} \cdots \lVert f_m \rVert_{X_m}
\end{equation*}
holds for all scalar-valued $f_i \in X_i$ and for some constant
$C_0 > 0$. Then
\begin{equation*}
\lVert L_O(f_1, \ldots, f_m) \rVert_{X}
\lesssim
(Cd)^{Cm} C_0 \lVert f_1 \rVert_{X_1} \cdots \lVert f_m \rVert_{X_m}
\end{equation*}
holds for all $f_i \in X_i$ that are scalars, $d$-dimensional vectors,
or $d \times d$ matrices.
\end{lem}
The next lemma is an adapation of \cite[Lemma 2]{T2}.
\begin{lem}[Leibniz rule]
Let $P_k^\prime$ be a $C^\infty$ Fourier multiplier
whose frequency support lies in some compact subset
of $I_k(\mathbf{R}^d)$. The commutator identity
\begin{equation*}
P_k^\prime (f g) =
f P_k^\prime g + L_O(\partial_x f, 2^{-k} g)
\end{equation*}
holds.
\label{L:LeibnizRule}
\end{lem}
\begin{proof}
Rescale so that $k = 0$ and let $m(\xi)$ denote the symbol 
of $P_0^\prime$ so that
\begin{equation*}
\widehat{P_0^\prime h}(\xi) := m(\xi) \hat{h}(\xi).
\end{equation*}
By the Fundamental Theorem of Calculus, we have
\begin{equation*}
\begin{split}
\left( P_0^\prime (f g) - f P_0^\prime g \right)(s, x, t)
&= \int_{\mathbf{R}^d}
\check{m}(y) \left(f(s, x - y, t) - f(s, x, t) \right) g(s, x - y, t) dy\\
&= - \int_0^1 \int_{\mathbf{R}^d} \check{m}(y) y \cdot \partial_x f(s, x - r y, t) g(s, x - y, t) dy dr.
\end{split}
\end{equation*}
The conclusion follows from the rapid decay of $\hat{m}$.
\end{proof}

We are interested in controlling $P_{\theta, j} P_k \psi_m$ in $L^{\infty, 2}_\theta$
over all $\theta \in \mathbf{S}^1$ and $\lvert j - k \rvert \leq 20$. In the abstract
framework, however, we assumed a much tighter localization than $P_{\theta, j}$
provides. Therefore we decompose $P_{\theta, j}$ as a sum
\begin{equation}
P = \sum_{l = 1, \ldots, O\left((\eta^*)^{-1}\right)} P_{\theta, j, l},
\label{Psum}
\end{equation}
and it suffices by the triangle inequality to bound  $P_{\theta, j, l} P_k \psi_m$.
We note that this does not affect perturbative estimates since $\eta^*$ is universal
and in particular does not depend upon $\varepsilon_1, \varepsilon$.

For notational convenience set $P := P_{\theta, j, l} P_k$. Applying
$P$ to (\ref{app start}) yields
\begin{equation*}
(i \partial_t + \Delta) P \psi_m
= P \left( \Blo + \Bhi + V_m \right).
\end{equation*}
Now
\begin{equation*}
P \Blo
=
- i P \sum_{\lvert k_3 - k \rvert \leq 4}(\partial_\ell(\Alol P_{k_3} \psi_m)
+ \Alol \partial_\ell P_{k_3} \psi_m),
\end{equation*}
as $P$ localizes to a region of the annulus $I_k$.
Applying Lemma \ref{L:LeibnizRule}, we obtain
\begin{equation*}
P \Blo
=
- i (\partial_\ell(\Alol P \psi_m)
-i \Alol \partial_\ell P \psi_m)
+ R
\end{equation*}
where
\begin{equation}
R :=
\sum_{\lvert k_3 - k \rvert \leq 4} \left( L_O(\partial_x \partial_\ell \Alol,
 2^{-k} P_{k_3} \psi_m) +
 L_O( \partial_x \Alol, 2^{-k} P_{k_3} \partial_\ell \psi_m) \right).
\label{R def}
\end{equation}
Set
\begin{equation*}
\cA_m := \Alo.
\end{equation*}
Then
\begin{equation}
(i \partial_t + \Delta_{\cA}) P \psi_m
=
P(\Bhi + V_m) + \cA_x^2 P \psi_m + R.
\label{ls app prep}
\end{equation}
It is this equation 
that we shall show fits within the abstract local smoothing framework.

First we check that Lemmas \ref{L:Abstract energy} and \ref{L:LS prep}
apply.  The main condition to check is (\ref{A L-infinity}).
Key are (\ref{A_x Bound}) and Bernstein, which together 
with the fact that $\cA$ is frequency-localized to $I_{(-\infty, k]}$
provide the estimate
\begin{equation*}
\lVert \cA \rVert_{L^\infty_{t,x}} \lesssim 2^k.
\end{equation*}
To achieve the $\varepsilon_m$ gain, we adjust $\varpi$, which
forces a gap between $I_k$ and the frequency support of $\cA$,
i.e., we localize $\cA$ to $I_{(-\infty, k - \varpi]}$ instead.
Thus it suffices to set $\varpi \in \mathbf{Z}_+$ equal to a sufficiently large universal
constant.

There is more to check in showing that (\ref{ls app prep}) falls within the purview
of Theorem \ref{T:Abstract LS}.
Already we have $d = 2$, $\e = \theta$, $\varepsilon_m \sim 2^{-\varpi}$,
$\cA_m := \Alo$,
$v = P_{\theta, j, l} P_k \psi_m$,
and $\Lambda_v = P(\Bhi + V_m) + \cA_x^2 P \psi_m + R$.

Next we choose $\cP$ based upon the norms used in $N_k$, with the exception
of the local smoothing/maximal function estimates. To be precise, define
the new norms $\tN_k$ via
\begin{equation*}
\begin{split}
\lVert f \rVert_{\tN_k(T)} :=& \inf_{f = f_1 + f_2 + f_3 + f_4 + f_5}
\lVert f_1 \rVert_{L_{t,x}^{4/3}}
+ 2^{k/6} \lVert f_2 \rVert_{L_{\hat{\theta}_1}^{3/2, 6/5}}
+ 2^{k/6} \lVert f_3 \rVert_{L_{\hat{\theta}_2}^{3/2, 6/5}} \\
& + 2^{-k/6}
\lVert f_4 \rVert_{L^{6/5,3/2}_{\hat{\theta}_1}}
+ 2^{-k/6}
\lVert f_5 \rVert_{L^{6/5,3/2}_{\hat{\theta}_2}}
\end{split}
\end{equation*}
and similarly $\tG_k$ via
\begin{equation*}
\begin{split}
\lVert f \rVert_{\tG_k(T)} :=& \;
\lVert f \rVert_{L_t^\infty L_x^2} +
\lVert f \rVert_{L_{t,x}^4}
+ 2^{-k/2} \lVert f \rVert_{L_x^4 L_t^\infty} \\
&+ 2^{-k/6} \sup_{\theta \in \mathbf{S}^1}
\lVert f \rVert_{L^{3,6}_\theta}
+ 2^{k / 6} \sup_{\lvert j - k \rvert \leq 20} \sup_{\theta \in \mathbf{S}^1}
\lVert P_{j, \theta} f \rVert_{L_\theta^{6,3}}.
\end{split}
\end{equation*}
Set $\cP = \{2, 3, 3/2, 4, 4/3, 6, 5/6 \}$.
We define the form $B(\cdot, \cdot)$ via
\begin{equation}
B(f, g) := 
\lVert f \rVert_{\tG_k(T)} \lVert g \rVert_{\tN_k(T)}
\label{Bdef}
\end{equation}
and $B_\theta$ by
\begin{equation}
B_\theta(f, g) := B(f, \chi_{\{ x \cdot \theta \geq 0 \}} g)
\label{Btheta}
\end{equation}
as in Definition \ref{D:good}. That $B_\theta$
is adapted to $\cP$ is a direct consequence of the definition.

\begin{prop}
Let $\eta > 0$ be a parameter to be specified later.
Let
$d = 2$, $\e = \theta$, $\varepsilon_m \sim 2^{-\varpi}$,
$\cA_m := \Alo$,
$v = P_{\theta, j, l}^{(\eta)} P_k \psi_m$,
$\Lambda_v = P(\Bhi + V_m) + \cA_x^2 P \psi_m + R$,
and
$\cP = \{2, 3, 3/2, 4, 4/3, 6, 5/6 \}$.
Let $B, B_\theta$ be given by (\ref{Bdef}) and (\ref{Btheta})
respectively.
Then the conditions of Theorem \ref{T:Abstract LS} are satisfied
and the derived sequence of $v$ is controlled with respect to $B_\theta$
so that conclusion (\ref{Abstract LS}) holds for
$v = P_{\theta, j, l}^{(\eta)} P_k \psi_m$ given $\eta$ sufficiently small.
\label{P:Applying ALS}
\end{prop}
\begin{proof}
The only claim of Proposition \ref{P:Applying ALS} that remains to be
verified is that the derived sequence of 
$v = P_{\theta, j, l} P_k \psi_m$
is controlled with respect to $B_\theta$. In particular, we need
to show that for each $q \geq 1$ we have
\begin{equation*}
B_\theta(v^{(q)}, \Lambda_{v^{(q)}}) < \infty,
\end{equation*}
where $v^{(1)} := P_{\theta, j, l} P_k \psi_m$,
\begin{equation*}
v^{(q+1)} := v^{(q)} + i 2^{-j} \partial_\theta v^{(q)},
\end{equation*}
and
\begin{equation*}
\Lambda_{v^{(q+1)}} :=
(1 + i 2^{-j} \partial_\theta) \Lambda_{v^{(q)}}
+ i2^{-j}( i \partial_\theta \partial_\ell \cA_\ell - \partial_\theta \cA_\ell^2) v^{(q)}
- 2^{-j + 1} (\partial_\theta \cA_\ell) \partial_\ell v^{(q)}.
\end{equation*}

We first prove the following lemma.
\begin{lem}
Let $\sigma \in [0, \sigma_1 - 1]$.
The right hand side of (\ref{ls app prep}) satisfies
\begin{equation*}
\lVert P(\Bhi + V_m) + \cA_x^2 P \psi_m + R \rVert_{\tN_k(T)}
\lesssim
\cV 2^{-\sigma k} b_k(\sigma).
\end{equation*}
\label{L:ls app RHS}
\end{lem}
\begin{proof}
We will repeatedly use implicitly the fact that the multiplier $P_{\theta, j, l}$
is bounded on $L^p, 1 \leq p \leq \infty$, so that in particular $P$
obeys estimates that are at least as good as those obeyed by $P_k$.

From Corollaries \ref{V bound corollary}, \ref{B bound corollary} of 
\S\S \ref{SS:V}, \ref{SS:B} it follows that $P_k(\Bhi + V_m)$
is perturbative and bounded in $\tN_k(T)$ by 
$\cV 2^{-\sigma k} b_k(\sigma)$.
The $\tN_k(T)$ estimates on $P V_m$ immediately imply
the boundedness of $\cA_x^2 P \psi_m$.

To estimate $R$, we apply Lemma \ref{L:Trilin1} to bound
$P \Blo$ by
\begin{equation*}
\begin{split}
\sum_{(k_1, k_2, k_3) \in Z_1(k)}
&\int_0^\infty 2^{\max\{k_1, k_2\}} 2^{k_3 - k} C_{k, k_1, k_2, k_3}
\lVert P_{k_1} \psi_x(s) \rVert_{F_{k_1}} \times \\
&\times \lVert P_{k_2}(D_\ell \psi_\ell(s)) \rVert_{F_{k_2}}
\lVert P_{k_3} \psi_m(0) \rVert_{G_{k_3}} ds,
\end{split}
\end{equation*}
which, in view of (\ref{PsiF}), (\ref{APsiF}), and (\ref{Ik1k2}),
is controlled by
\begin{equation*}
\sum_{(k_1, k_2, k_3) \in Z_1(k)}
C_{k, k_1, k_2, k_3} b_{k_1} b_{k_2} 2^{-\sigma k_3} b_{k_3}(\sigma).
\end{equation*}
Summation is achieved thanks to Corollary \ref{C:MTE1}.
\end{proof}

We return to the proof of the proposition, and in particular to showing
that $B_\theta(v, \Lambda_v) < \infty$. 
With the important observation that
the spatial multiplier
$\chi_{x \cdot \theta \geq 0}$ is bounded on the spaces
$\tN_k(T)$, we may apply Lemma \ref{L:ls app RHS}
to control $\chi_{x \cdot \theta \geq 0} \Lambda_v$ in $\tN_k$.
Since by assumption $P \psi_m$ is bounded in $\tG_k(T)$
(even in $G_k(T)$), we conclude that
$B_\theta(v, \Lambda_v) < \infty$.

Next we need to show $B_\theta(v^{q}, \Lambda_{v^{q}}) < \infty$
for $q > 1$.
By Bernstein,
\begin{equation*}
\lVert v^{(q)} \rVert_{\tG_{k}(T)}
\lesssim
\lVert v^{(q-1)} \rVert_{\tG_{k}(T)}.
\end{equation*}
Similarly,
\begin{equation*}
\lVert (1 + i 2^{-j}) \partial_\theta \Lambda_{v^{(q)}} \rVert_{\tN_k(T)}
\lesssim
\lVert \Lambda_{v^{(q-1)}} \rVert_{\tN_k(T)}.
\end{equation*}
Thus it remains to control
$i2^{-j}( i \partial_\theta \partial_\ell \cA_\ell - \partial_\theta \cA_\ell^2) v^{(q)}$
and $2^{-j + 1} (\partial_\theta \cA_\ell) \partial_\ell v^{(q)}$ in $\tN_k$
for each $q > 1$.
Both are consequences of arguments in Lemma \ref{L:ls app RHS}:
Boundedness of $2^{-j} (\partial_\theta \partial_\ell \cA_\ell) v^{(q)}$
and $2^{-j+1}(\partial_\theta \cA_\ell) \partial_\ell v^{(q)}$
follows directly from the argument used to control $R$ 
and from Bernstein's
inequality, whereas boundedness of $2^{-j} (\partial_\theta \cA_\ell^2) v^{(q)}$
is a consequence of Bernstein and the estimates on $\cA_x^2 P \psi_m$
from \S \ref{SS:V}.
\end{proof}

Combining Lemma \ref{L:ls app RHS} and Proposition \ref{P:Applying ALS},
we conclude that Corollary \ref{C:Abstract LS} applies to
$v = P \psi_m$, with right hand side bounded by
$2^{-2 \sigma k} c_k(\sigma)^2 + \cV 2^{-2 \sigma k} b_k(\sigma)^2$.
In view of the decomposition (\ref{Psum}), we conclude
\begin{cor}
Assume $\sigma \in [0, \sigma_1 - 1]$.
The function $P_k \psi_m$
satisfies
\begin{equation*}
\sup_{\lvert j - k \rvert \leq 20} \sup_{\theta \in \mathbf{S}^1}
\lVert P_{j, \theta} P_k \psi_m \rVert_{L^{\infty, 2}_{\theta}}
\lesssim 2^{-k/2} ( 2^{-\sigma k} c_k(\sigma) + \cV^{1/2} 2^{-\sigma k} b_k(\sigma) ),
\end{equation*}
\label{C:ALS}
\end{cor}
thereby establishing Theorem \ref{ls thm}.

Our next objective is to apply Corollary \ref{C:ABS} to the case where
$w$ splits as a product $u(x) v(y)$ where $u, v$ are appropriate frequency
localizations of $\psi_m$ or $\overline{\psi_m}$.
First we must find function spaces suitable for defining an adapted form.
We start with $(i \partial_t + \Delta_{\cA}) w = \Lambda_w$
and observe how it behaves with respect to separation
of variables.
If $w(x, y) = u(x) v(y)$, then the left hand side may be rewritten as
$u \cdot (i \partial_t + \Delta_{\cA_y})v + v \cdot (i \partial_t + \Delta_{\cA_x})u$.
Let $\Lambda_u := (i \partial_t + \Delta_{\cA_x})u$ and
$\Lambda_v := (i \partial_t + \Delta_{\cA_y})v$. Then
\begin{equation*}
(i \partial_t + \Delta_{\cA}) (uv) = u \Lambda_v + v \Lambda_u.
\end{equation*}
We control
\begin{equation*}
\int_0^T \int_{\mathbf{R}^2 \times \mathbf{R}^2} u(x) v(y) 
\left( \Lambda_u(x) v(y) + u(x) \Lambda_v(y) \right) dx dy dt
\end{equation*}
as follows: in the case of the first term $u(x) v(y) \Lambda_u(x) v(y)$ we place each 
$v(y)$ in $L^\infty_t L^2_y$; we bound $u(x) \Lambda_u(x)$ by
placing $u(x)$ in $G_j$ and $\Lambda_u(x)$ in $\tN_j$. To control
$u(x) v(y) u(x) \Lambda_v(y)$, 
we simply reverse the roles of $u$ and $v$ (and of $x$ and $y$).
This leads us to the spaces $\oN_{k, \ell}$ defined by
\begin{equation}
\begin{split}
\lVert f \rVert_{\oN_{k, \ell}(T)} :=
\inf_{ \substack{ J \in \mathbf{Z}_+, 
f(x, y) = \\ \sum_{j = 1}^{2J} (g_j(x)h_j(y) + g_{j+1}(x)h_{j+1}(y))} }
&\left( \lVert g_j \rVert_{\tN_\ell(T)} \lVert h_j \rVert_{L^\infty_t L^2_y} \right. \\
&\;\left. + \lVert g_{j+1} \rVert_{L^\infty_t L^2_x} \lVert h_{j+1} \rVert_{\tN_k(T)} \right),
\end{split}
\label{oN def}
\end{equation}
and the spaces $\oG_{k, \ell}$ defined via
\begin{equation}
\lVert f \rVert_{\oG_{k, \ell}(T)}  :=
\lVert  \lVert f(x, y) \rVert_{\tG_k(T)(y)} \rVert_{\tG_{\ell}(T)(x)}.
\label{oG def}
\end{equation}
We use these spaces to define the form $\oB(\cdot, \cdot)$ by
\begin{equation}
\oB(f, g) := 
\lVert f \rVert_{\oG_{k, \ell}(T)} \lVert g \rVert_{\oN_{k, \ell}(T)}
\label{oBdef}
\end{equation}
and the form $\oB_\Theta$ by
\begin{equation}
\oB_\Theta(f, g) := \oB(f, \chi_{\{ (x, y) \cdot \Theta \geq 0 \}} g),
\label{oBtheta}
\end{equation}
where $\Theta := (-\theta, \theta)$.

\begin{prop}
Let $\eta > 0$ be a small parameter 
and $\varpi \in \mathbf{Z}_+$ a large parameter, both to be specified later.
Let $j, k, \ell \in \mathbf{Z}$, $j = k + O(1)$, $\ell \ll k$.
Let
$d = 2$, $\e = \theta$, $\varepsilon_m \sim 2^{-\varpi}$,
$\cA_x := \Alo$,
$v = P_{\theta, j, l}^{(\eta)} P_k \psi_m$,
$\Lambda_v = P(\Bhi + V_m) + \cA_x^2 P \psi_m + R$,
and
$\cP = \{2, 3, 3/2, 4, 4/3, 6, 5/6 \}$.
Here $R$ is given by (\ref{R def}).
Also, let $u = \overline{P_{\ell} \psi_{p}}$, $p \in \{1, 2\}$ and
$\Lambda_u = P_{\ell} (B_{p, \mathrm{hi \lor hi}} + V_p) + \cA_x^2 P_\ell \psi_p + R^\prime$,
where $R^\prime$ is given by (\ref{R def}), but defined in terms of derivative field
$\psi_\ell$ and frequency $\ell$ rather than $\psi_m$ and $k$.

Let $w(x, y) := u(x)v(y)$, $\cA := (\cA_x, \cA_y)$,
$\Lambda_w := \Lambda_u v + u \Lambda_v$.
Then, for $\varpi$ sufficiently large and $\eta$ sufficiently small,
the conditions of Corollary \ref{C:ABS} are satisfied
and (\ref{ABS}) applies to $u(x)v(x)$.
\label{P:ABSapp}
\end{prop}
\begin{proof}
The frequency support conditions on $\cA$
and $\Lambda_w$ are easily verified.
That the $L^\infty$ bound
on $\cA$ holds follows from (\ref{A_x Bound}) and Bernstein
provided $\varpi$ is large enough (cf.~discussion preceding
Proposition \ref{P:Applying ALS}).
In order to guarantee the frequency support conditions on $w$,
it is necessary to make the gap $\ell \ll k$ sufficiently large
with respect to $\eta$.

That $\oB_\Theta$ is adapted to $\cP$ is a 
straightforward consequence of its definition. To see that the derived
sequence of $w$ is controllable, we look to the proof of
Proposition \ref{P:Applying ALS} and the
definitions of the $\oN_{k, \ell}$, $\oG_{k, \ell}$ spaces.
\end{proof}

In a spirit similar to that of the proof of Corollary \ref{C:ALS}, 
we may combine Lemma \ref{L:ls app RHS} and
the proof of Proposition \ref{P:Applying ALS}
to control $B(w, \Lambda_w) + B_\Theta(w, \Lambda_w)$;
in fact, in measuring $\Lambda_w$ in the 
$\oN_{k, \ell}$ spaces, it suffices to take $J = 1$ (see (\ref{oN def})).
Then we obtain
$B(w, \Lambda_w) + B_\Theta(w, \Lambda_w) \lesssim \cV b_j 2^{-\sigma k} b_k(\sigma)$.
Using decomposition (\ref{Psum}) and the triangle inequality 
to bound $P_k \psi_m$ in terms of the bounds on
$P_{\theta, j, \ell}^{(\eta)} P_k \psi_m$, we obtain the bilinear Strichartz
analogue of Corollary \ref{C:ALS}. In our application, however,
the lower-frequency term will not simply be $\overline{P_j \psi_\ell}$,
but rather its heat flow evolution $\overline{P_j \psi_\ell}(s)$.

\begin{cor}[Improved Bilinear Strichartz]
Let $j, k \in \mathbf{Z}$, $j \ll k$, and
$u \in \{ P_j \psi_\ell, \overline{P_j \psi_\ell} : j \leq k - \varpi, \ell \in \{1, 2\} \}$.
Then for $s \geq 0$, $\sigma \in [0, \sigma_1 - 1]$,
\begin{equation}
\lVert u(s) P_k \psi_m(0) \rVert_{L^2_{t,x}}
\lesssim
2^{(j - k) / 2} (1 + s 2^{2j})^{-4} 2^{-\sigma k} \left(c_j c_k(\sigma) + \cV b_j b_k(\sigma) \right).
\label{ABSapp}
\end{equation}
\label{C:ABSapp}
\end{cor}
\begin{proof}
It only remains to prove (\ref{ABSapp}) when $s > 0$.
Let $v := P_k \psi_m$. Using the Duhamel
formula, we write
\begin{equation}
u(s) v
=
(e^{s \Delta}u(0))  v(0) + 
\int_0^s e^{(s - s^\prime)\Delta} U(s^\prime) \ds^\prime \cdot v(0),
\label{Uhamel}
\end{equation}
where $U$ is defined by (\ref{Heat Nonlinearity 0}) in terms
of $u$.

To control the nonlinear term 
$\int_0^s e^{(s - s^\prime) \Delta} U(s^\prime) \ds^\prime \cdot v(0)$ in $L^2$,
we apply local smoothing estimate (\ref{bilin5 improv}), which places
the nonlinear evolution in $F_j(T)$ and $v(0)$ in $G_k(T)$.
Using Lemma \ref{Nonlinear Evolution Bound} 
to bound the $F_j(T)$ norm, we conclude
\begin{equation}
\lVert \int_0^s e^{(s - s^\prime) \Delta} \widetilde{U}(s^\prime) \ds^\prime \cdot v(0) \rVert_{L^2_{t,x}}
\lesssim
\cV
2^{(j - k) / 2}
(1 + s 2^{2j})^{-4}  2^{-\sigma k} b_j b_k(\sigma).
\label{NonlinearBSterm}
\end{equation}
It remains to show
\begin{equation}
\lVert (e^{s \Delta} u) v \rVert_{L^2_{t,x}}
\lesssim
(1 + s 2^{2j})^{-4} 2^{(j- k)/2} 2^{-\sigma k}(c_j c_k(\sigma) + \cV b_j b_k(\sigma)),
\label{LinearBS}
\end{equation}
which is not a direct consequence of the time $s = 0$ bound.
Let $\cT_{a}$ denote the spatial translation operator
that acts on functions $f(x, t)$ according to
$\cT_{a} f(x, t) := f(x - a, t)$.
Then, if
\begin{equation}
\lVert (\cT_{x_1} u) (\cT_{x_2} v) \rVert_{L^2_{t,x}}
\lesssim
2^{(j- k)/2} 2^{-\sigma k}(c_j c_k(\sigma) + \cV b_j b_k(\sigma))
\label{translated}
\end{equation}
can be shown to hold for all $x_1, x_2 \in \mathbf{R}^2$,
then (\ref{LinearBS}) follows from Minkowski's and Young's inequalities.

Consider, then, a solution $w$ to
\begin{equation*}
(i \partial_t + \Delta_\cA(x, t)) w(x, t) = \Lambda_w(x, t)
\end{equation*}
satisfying the conditions of Theorem \ref{T:Abstract LS}.
The translate $\cT_{x_0} w(x, t)$ then satisfies
\begin{equation*}
(i \partial_t + \Delta_{\cT_{x_0}(\cA)(x, t)}) 
(\cT_{x_0}w)(x, t) = (\cT_{x_0}\Lambda_w)(x, t).
\end{equation*}
The operator $\cT_{x_0}$ clearly does not affect 
$L^\infty_{t,x}$ bounds or frequency support conditions.
The only possible obstruction to concluding (\ref{Abstract LS})
is this: whereas the derived
sequence of $w$ is controlled with respect to $B_\e$,
in the abstract setting it may no longer be the case
that the derived sequence of $\cT_{x_0} w$ is controlled.
This is due to the presence of the spatial multiplier
in the definition of $B_\e$.
Fortunately, as already alluded to in the proof of
Proposition (\ref{P:Applying ALS}), in our applications
we do enjoy uniform boundedness with respect to
\emph{any} spatial multipliers appearing in the second
argument of an adapted form $B_\e$.
Therefore Proposition \ref{P:ABSapp} holds 
for spatial translates of frequency projections of $\psi_m$,
from which we conclude (\ref{translated}).
\end{proof}
This establishes Theorem \ref{T:BilinearStrichartz}.

%% file: CaloricGauge.tex
\section{The caloric gauge} \label{S:The caloric gauge}

In \S \ref{SS:CalCon} we briefly recall from \cite{Sm09}
the construction of the caloric gauge and some useful
quantitative estimates.
In \S \ref{SS:FreqLocCalGau} we prove the frequency-localized estimates
stated in \S \ref{SS:Frequency localization}.

\subsection{Construction and basic results}\label{SS:CalCon}
In brief, the basic caloric gauge construction goes as follows.
Starting with $H^\infty_Q$-class data $\varphi_0 : \rr \to \mathbf{S}^2$ with energy $E(\varphi_0) < \Ec$,
evolve $\varphi_0$ in $s$ via the heat flow equation (\ref{Covariant Heat}).  At $s = \infty$
the map trivializes.  Place an arbitrary orthonormal frame $e(\infty)$ on
$T_{\varphi(s = \infty)} \mathbf{S}^2$.  Evolving this frame backward in time via parallel
transport in the $s$ direction yields a caloric gauge on $\varphi^* T_{\varphi(s = \infty)} \mathbf{S}^2$.

For energies $E(\varphi_0)$ sufficiently small, global existence and decay bounds
may be proven directly using Duhamel's formula.  In order to extend these results
to all energies less than $\Ec$, we employ in \cite{Sm09} an induction-on-energy argument
that exploits the symmetries of (\ref{Covariant Heat}) via concentration compactness.

In \cite{Sm09} the following energy densities play an important role in the quantitative arguments.
\begin{defin}
For each positive integer $k$, define the energy densities $\e_k$ of a heat flow $\varphi$
by
\begin{align}
\e_k &:= \lvert (\varphi^* \nabla)_x^{k-1} \partial_x \varphi \rvert^2 \nonumber \\
&:= \langle (\varphi^* \nabla)_{j_1} \cdots (\varphi^* \nabla)_{j_{k-1}} \partial_{j_k}\varphi,
(\varphi^* \nabla)_{j_1} \cdots (\varphi^* \nabla)_{j_{k-1}} \partial_{j_k}\varphi \rangle,
\label{ek def}
\end{align}
where $j_1, \ldots, j_k$ are summed over $1,2$ and $\nabla$ denotes the
Riemannian connection on the sphere, i.e., for vector fields $X, Y$ on the sphere
$\nabla_X Y$ denotes the orthogonal projection of $\partial_X Y$ onto the sphere.
\label{Energy Densities}
\end{defin}
A key quantitative result of \cite{Sm09} is the following
\begin{thm}
For any initial data $\varphi_0 \in H^\infty_Q$ 
with $E(\varphi_0) < \Ec$ we have that there exists a unique global 
smooth heat flow $\varphi$ with initial data $\varphi_0$.  
Moreover, $\varphi$ satisfies the following estimates
\begin{align}
\int_0^\infty \int_\rr s^{k-1} \e_{k+1}(s, x) \dx ds 
&\lesssim_{E_0, k} 1 \label{ek1} \\
\sup_{0 < s < \infty} s^{k-1} \int_\rr \e_k(s, x) \dx 
&\lesssim_{E_0, k} 1 \nonumber \\
\sup_{ \substack{ 0 < s < \infty \\ x \in \rr} } s^k \e_k(s, x) 
&\lesssim_{E_0, k} 1 \nonumber \\
\int_0^\infty s^{k-1} \sup_{x \in \rr} \e_k(s, x) \ds 
&\lesssim_{E_0, k} 1 \label{ek4}
\end{align}
for each $k \geq 1$, as well as the estimate
\begin{equation}
\int_0^\infty \int_\rr \e_1^2(s, x) \ds ds \lesssim_{E_0} 1.
\label{controlling norm}
\end{equation}
\end{thm}
We employ (\ref{ek1}), (\ref{ek4}), and (\ref{controlling norm}) below.

\subsection{Frequency-localized caloric gauge estimates}\label{SS:FreqLocCalGau}
The key estimates to establish are (\ref{Hard Envelope Bounds}) for $\varphi$; most
remaining estimates will be derived as a corollary.  Our strategy is to exploit energy dispersion 
so that we can apply
the Duhamel formula
to a frequency localization of 
the heat flow equation (\ref{Covariant Heat}), 
which for convenience we rewrite as
\begin{equation}
\partial_s \varphi = \Delta \varphi + \varphi \e_1.
\label{Covariant Heat Alt}
\end{equation}

\begin{proof}[Proof of (\ref{Hard Envelope Bounds}) for $\varphi$]

Let $\sigma_1 \in \mathbf{Z}_+$ be positive
and let $S^\prime \geq S \gg 0$.
Let $\mathcal{K} \in \mathbf{Z}_+$, $T \in (0, 2^{2\mathcal{K}}]$ be fixed. 
Define for each $t \in (-T, T)$ the quantity
\begin{equation}
\cC(S, t) := \sup_{\sigma \in [2 \delta, \sigma_1]} \sup_{s \in [0, S]} \sup_{k \in \mathbf{Z}}
(1 + s 2^{2k})^{\sigma_1} 2^{\sigma k} \gamma_k(\sigma)^{-1} 
\lVert P_k \varphi(s, \cdot, t) \rVert_{L_x^2(\rr)}.
\label{B(S)}
\end{equation}
For fixed $t$ the function $\cC(S, t): [0, S^\prime] \to (0, \infty)$ 
is well-defined, continuous, and non-decreasing.
Moreover, in view of the definition (\ref{gamma Envelope}) of $\gamma_k(\sigma)$, it follows that
$\lim_{S \to 0} \cC(S, t) \lesssim 1$.
A simple consequence of (\ref{B(S)}) is
\begin{equation}
\lVert P_k \varphi(s, \cdot, t) \rVert_{L_x^2(\rr)}
\leq
\cC(S, t)
(1 + s 2^{2k})^{-\sigma_1} 2^{-\sigma k} \gamma_k(\sigma)
\label{B(S,t) con}
\end{equation}
for $0 \leq s \leq S \leq S^\prime$.

Our goal is to show $\cC(S, t) \lesssim 1$ uniformly in $S$ and $t$ and our strategy
is to apply Duhamel's formula to (\ref{Covariant Heat Alt})
and run a bootstrap argument.
Beginning with the decomposition
\begin{align*}
P_k (\varphi \e_1) =& 
\sum_{\lvert k_2 - k \rvert \leq 4} P_k (P_{\leq k-5}\varphi \cdot P_{k_2} \e_1) + \\
&\sum_{\lvert k_1 - k \rvert \leq 4} P_k (P_{k_1}\varphi \cdot P_{\leq k - 5} \e_1) + \\
&\sum_{\substack{ k_1, k_2 \geq k - 4 \\ \lvert k_1 - k_2 \rvert \leq 8}}
P_k (P_{k_1} \varphi \cdot P_{k_2} \e_1),
\end{align*}
we proceed to place in $L_x^2$ each of the three terms on the right hand side;
we then integrate in $s$ and consider separately 
the low-high, high-low, and high-high frequency
interactions.

\nline
\noindent \textbf{Low-High interaction.}
By Duhamel and the triangle
inequality it suffices to bound
\begin{equation}
\mathrm{LH}(s, t) := \int_0^s e^{-(s - s^\prime)2^{2k-2}} \sum_{\lvert k_2 - k\rvert \leq 4}
\lVert P_k (P_{\leq k - 5} \varphi(s^\prime, \cdot, t) \cdot P_{k_2} \e_1(s^\prime, \cdot, t)) \rVert_{L_x^2}
ds^\prime.
\label{LH Def}
\end{equation}
By H\"older's inequality, $\lvert \varphi \rvert \equiv 1$, and $L^p$-boundedness
of the Littlewood-Paley multipliers,
\begin{align*}
\mathrm{LH}(s, t) &\lesssim \int_0^s e^{-(s - s^\prime)2^{2k-2}} \sum_{\lvert k_2 - k\rvert \leq 4}
\lVert P_{\leq k-5}\varphi \rVert_{L_x^\infty} \lVert P_{k_2} \e_1 \rVert_{L_x^2} ds^\prime \\
&\lesssim \int_0^s e^{-(s - s^\prime)2^{2k-2}} \sum_{\lvert k_2 - k\rvert \leq 4}
\lVert P_{k_2} \e_1(s^\prime, \cdot, t) \rVert_{L_x^2} ds^\prime.
\end{align*}
To control the sum
we further decompose $P_\ell \e_1 = P_\ell (\partial_x \varphi \cdot \partial_x \varphi)$ into low-high and high-high frequency interactions:
\begin{equation}
P_\ell \e_1 =
2 \sum_{\lvert \ell_1 - \ell \rvert \leq 4} P_\ell ( P_{\leq \ell - 5} \partial_x \varphi \cdot
P_{\ell_1} \partial_x \varphi ) +
\sum_{\substack{\ell_1, \ell_2 \geq \ell - 4 \\ \lvert \ell_1 - \ell_2\rvert \leq 8}}
P_\ell (P_{\ell_1}\partial_x \varphi \cdot P_{\ell_2} \partial_x \varphi).
\label{Further Decomposition}
\end{equation}

\nline
\noindent \textbf{Low-High interaction (i).}
We first attend to the low-high subcase.  For convenience set $\Xi_{lh}$
equal to the first term of the right hand side of (\ref{Further Decomposition}), i.e.,
\begin{equation*}
\Xi_{lh}(s, x, t) := \sum_{\lvert \ell_1 - \ell \rvert \leq 4} 
P_\ell ( P_{\leq \ell - 5} \partial_x \varphi(s, x, t) \cdot
P_{\ell_1} \partial_x \varphi(s, x, t) ).
\end{equation*}
By the triangle inequality, H\"older's inequality, Berstein's inequality, the
definition (\ref{ek def}) for $\e_1(s, \cdot, t)$, and (\ref{B(S,t) con}), it follows that
\begin{align*}
\lVert \Xi_{lh}(s, \cdot, t) \rVert_{L_x^2}
&\lesssim 
\sum_{\lvert \ell_1 - \ell \rvert \leq 4} \lVert P_\ell ( P_{\leq \ell - 5} \partial_x \varphi \cdot
P_{\ell_1} \partial_x \varphi )\rVert_{L_x^2} \\
&\lesssim
\sum_{\lvert \ell_1 - \ell \rvert \leq 4} \lVert P_{\leq \ell - 5} \partial_x \varphi \rVert_{L_x^\infty}
\lVert P_{\ell_1} \partial_x \varphi \rVert_{L_x^2} \\
&\lesssim 
\sum_{\lvert \ell_1 - \ell \rvert \leq 4} \lVert P_{\leq \ell - 5} \partial_x \varphi \rVert_{L_x^\infty}
2^{\ell_1} \lVert P_{\ell_1} \varphi \rVert_{L_x^2} \\
&\lesssim
\lVert \sqrt{\e_1} \rVert_{L_x^\infty} 2^\ell \sum_{\lvert \ell_1 - \ell \rvert \leq 4}
\lVert P_{\ell_1} \varphi \rVert_{L_x^2} \\
&\lesssim
\lVert \sqrt{\e_1}(s, \cdot, t) \rVert_{L_x^\infty} 2^\ell 2^{-\sigma \ell} \gamma_\ell(\sigma)
\cC(S, t) (1 + s 2^{2 \ell})^{-\sigma_1}.
\end{align*}
As we apply this inequality in the case where $\ell = k_2$, $\lvert k_2 - k \rvert \leq 4$, we have
\begin{align}
\int_0^s & e^{-(s - s^\prime)2^{2k-2}} 
\lVert \Xi_{lh}(s^\prime, \cdot, t) \rVert_{L_x^2} ds^\prime \nonumber \\
&\lesssim
2^k 2^{-\sigma k} \gamma_k(\sigma) \cC(S, t)
\int_0^s e^{-(s - s^\prime)2^{2k-2}}
\lVert \sqrt{\e_1}(s^\prime, \cdot , t) \rVert_{L_x^\infty} (1 + s^\prime 2^{2 k})^{-\sigma_1} ds^\prime.
\label{LHi}
\end{align}
Apply Cauchy-Schwarz.  Clearly
\begin{equation}
\left( \int_0^s \lVert \sqrt{\e_1}(s^\prime, \cdot, t) \rVert_{L_x^\infty}^2 \ds^\prime \right)^{1/2}
\leq \lVert \e_1(\cdot, \cdot, t) \rVert_{L_s^1 L_x^\infty}^{1/2}.
\label{1stCSfactor}
\end{equation}
We postpone applying (\ref{ek4}) with $k = 1$ to (\ref{1stCSfactor}).
As for the other factor, we have
\begin{equation}
\left( \int_0^s e^{-(s - s^\prime)2^{2k-1}} (1 + s^\prime 2^{2k})^{-2\sigma_1} ds^\prime \right)^{1/2}
\lesssim \left( s (1 + s 2^{2k - 1})^{-2 \sigma_1} (1 + s 2^{2k})^{-1} \right)^{1/2}
\label{2ndCSfactor}
\end{equation}
since
\begin{equation*}
\int_0^s e^{-(s - s^\prime)\lambda} (1 + s^\prime \lambda^\prime)^{-\alpha} ds^\prime
\lesssim s (1 + \lambda s)^{-\alpha} (1 + \lambda^\prime s)^{-1}
\end{equation*}
for $s \geq 0$, $0 \leq \lambda \leq \lambda^\prime$, and $\alpha > 1$.
Hence, applying Cauchy-Schwarz to (\ref{LHi}) and using (\ref{1stCSfactor}) and (\ref{2ndCSfactor}), we get
\begin{align*}
\int_0^s &e^{-(s - s^\prime)2^{2k-2}} 
\lVert \Xi_{lh}(s^\prime, \cdot, t) \rVert_{L_x^2}
ds^\prime \\
&\lesssim 2^{-\sigma k} \gamma_k(\sigma) \cC(S, t) 2^k s^{1/2}
(1 + s 2^{2k-1})^{-\sigma_1} (1 + s 2^{2k})^{-1/2} 
\lVert \e_1(t) \rVert_{L_s^1 L_x^\infty([0, s] \times \rr)}^{1/2}.
\end{align*}
Discarding $s^{1/2} 2^k (1 + s 2^{2k})^{1/2} \leq 1$, we conclude
\begin{align}
\int_0^s &e^{-(s - s^\prime)2^{2k-2}} \lVert \Xi_{lh}(s^\prime, \cdot, t) 
\rVert_{L_x^2} \; ds^\prime \nonumber \\
&\lesssim
2^{-\sigma k} \gamma_k(\sigma) \cC(S, t)
(1 + s 2^{2k-1})^{-\sigma_1}
\lVert \e_1(t) \rVert_{L_s^1 L_x^\infty([0, s] \times \rr)}^{1/2}.
\label{lowhigh subcase}
\end{align}


\nline
\noindent \textbf{Low-High interaction (ii).}
We now move on to the high-high interaction subcase, setting
$\Xi_{hh}$ equal to the second term of the right hand side of (\ref{Further Decomposition}):
\begin{equation*}
\Xi_{hh}(s, x, t) :=  \sum_{\substack{\ell_1, \ell_2 \geq \ell - 4 \\ \lvert \ell_1 - \ell_2\rvert \leq 8}}
P_\ell (P_{\ell_1} \partial_x \varphi(s, x, t) \cdot P_{\ell_2} \partial_x \varphi(s, x, t) ).
\end{equation*}
By the triangle inequality, Bernstein, and Cauchy-Schwarz,
\begin{align*}
\lVert \Xi_{hh} \rVert_{L_x^2} &\lesssim
\sum_{\substack{\ell_1, \ell_2 \geq \ell - 4 \\ \lvert \ell_1 - \ell_2\rvert \leq 8}}
\lVert
P_\ell (P_{\ell_1} \partial_x \varphi \cdot P_{\ell_2} \partial_x \varphi )
\rVert_{L_x^2} \\
&\lesssim
\sum_{\substack{\ell_1, \ell_2 \geq \ell - 4 \\ \lvert \ell_1 - \ell_2\rvert \leq 8}}
2^{\ell} \lVert P_{\ell_1} \partial_x \varphi 
\cdot P_{\ell_2} \partial_x \varphi \rVert_{L_x^1} \\
&\lesssim
\sum_{\substack{\ell_1, \ell_2 \geq \ell - 4 \\ \lvert \ell_1 - \ell_2\rvert \leq 8}}
2^{\ell}
\lVert P_{\ell_1} \partial_x \varphi \rVert_{L_x^2}
\lVert P_{\ell_2} \partial_x \varphi \rVert_{L_x^2}.
\end{align*}
At this stage we apply Bernstein twice, exploiting $\lvert \ell_1 - \ell_2 \rvert \leq 8$:
\begin{align*}
\lVert P_{\ell_1} \partial_x \varphi \rVert_{L_x^2}
\lVert P_{\ell_2} \partial_x \varphi \rVert_{L_x^2}
&\lesssim
2^{\ell_2} \lVert P_{\ell_1} \partial_x \varphi \rVert_{L_x^2}
\lVert P_{\ell_2} \varphi \rVert_{L_x^2} \\
&\lesssim
\lVert P_{\ell_1} \lvert \partial_x \rvert^2 \varphi \rVert_{L_x^2}
\lVert P_{\ell_2} \varphi \rVert_{L_x^2}.
\end{align*}
So
\begin{equation*}
\lVert \Xi_{hh} \rVert_{L_x^2} \lesssim
2^{\ell}
\sum_{\substack{\ell_1, \ell_2 \geq \ell - 4 \\ \lvert \ell_1 - \ell_2\rvert \leq 8}}
\lVert P_{\ell_1} \lvert \partial_x \rvert^2 \varphi \rVert_{L_x^2}
\lVert P_{\ell_2} \varphi \rVert_{L_x^2}.
\end{equation*}
Applying Cauchy-Schwarz yields
\begin{align}
\lVert \Xi_{hh} \rVert_{L_x^2} &\lesssim
2^{\ell}
\left( \sum_{\ell_1 \geq \ell - 4} \lVert P_{\ell_1} \lvert \partial_x \rvert^2
\varphi \rVert_{L_x^2}^2 \right)^{1/2}
\left( \sum_{\ell_2 \geq \ell - 4} \lVert P_{\ell_2} \varphi 
\rVert_{L_x^2}^2 \right)^{1/2} \nonumber \\
&\lesssim
\lVert \lvert \partial_x \rvert^2 \varphi \rVert_{L_x^2}
2^{\ell}
\left( \sum_{\ell_2 \geq \ell - 4} \lVert P_{\ell_2} \varphi 
\rVert_{L_x^2}^2 \right)^{1/2}.
\label{hhCS}
\end{align}
As $\varphi$ takes values in $\mathbf{S}^2$, which has constant curvature,
we readily estimate ordinary derivatives by covariant ones:
\begin{equation}
\lvert \partial_x^2 \varphi \rvert \lesssim
\sqrt{\e_2} + \e_1.
\label{ord by cov}
\end{equation}
Applying (\ref{ord by cov}) in (\ref{hhCS}) and
using (\ref{B(S,t) con}), we arrive at
\begin{align}
\lVert \Xi_{hh}(s, \cdot, t) \rVert_{L_x^2}
\lesssim& \;
\lVert \sqrt{\e_2} + \e_1 \rVert_{L_x^2}
2^{\ell}
\left(
\sum_{\ell_2 \geq \ell - 4} 
\lVert P_{\ell_2} \varphi \rVert_{L_x^2}^2 \right)^{1/2} \nonumber \\
\lesssim& \;
\lVert (\sqrt{\e_2} + \e_1)(s, \cdot, t) \rVert_{L_x^2}
2^{\ell}
\cC(S, t) \times \nonumber \\
&\quad \times \left(
\sum_{\ell_2 \geq \ell - 4} (1 + s 2^{2 \ell_2})^{- 2 \sigma_1} 2^{-2 \sigma \ell_2}
\gamma_{\ell_2}^2(\sigma) \right)^{1/2} \nonumber \\
\lesssim& \;
\lVert (\sqrt{\e_2} + \e_1)(s, \cdot, t) \rVert_{L_x^2}
2^{\ell}
\cC(S, t) (1 + s 2^{2 \ell})^{-\sigma_1} \times \nonumber \\
&\quad \times
\left( \sum_{\ell_2 \geq \ell - 4} 2^{-2 \sigma \ell_2} \gamma_{\ell_2}^2(\sigma) \right)^{1/2}.
\label{presum}
\end{align}
As $\sigma > \delta$ is bounded away from $\delta$ uniformly, we may apply
summation rule (\ref{Sum 2}) in (\ref{presum}).
Recalling $\ell = k_2$ where $\lvert k_2 - k \rvert \leq 4$, we conclude
\begin{equation*}
\lVert \Xi_{hh}(s, \cdot, t) \rVert_{L_x^2} \lesssim
\lVert (\sqrt{\e_2} + \e_1)(s, \cdot, t) \rVert_{L_x^2}
2^k 2^{-\sigma k} \gamma_k(\sigma)
\cC(S, t)
(1 + s 2^{2 k})^{-\sigma_1}.
\end{equation*}
Integrating in $s$ yields
\begin{align}
\int_0^s &e^{-(s - s^\prime)2^{2k-2}} 
\lVert \Xi_{hh}(s^\prime, \cdot, t) \rVert_{L_x^2}
ds^\prime
\lesssim \nonumber \\
&2^k 2^{-\sigma k} \gamma_k(\sigma) \cC(S, t)
\int_0^s e^{-(s - s^\prime)2^{2k-2}}
\lVert (\sqrt{\e_2} + \e_1)(s^\prime, \cdot, t)
\rVert_{L_x^2} (1 + s^\prime 2^{2 k})^{-\sigma_1} ds^\prime.
\label{presplit}
\end{align}
We use the triangle inequality to write
$\lVert \sqrt{\e_2} + \e_1 \rVert_{L_x^2} \leq
\lVert \sqrt{\e_2} \rVert_{L_x^2} + \lVert \e_1 \rVert_{L_x^2}$
and split the integral in (\ref{presplit}) into two pieces.
By Cauchy-Schwarz and (\ref{2ndCSfactor}),
\begin{align}
\int_0^s &e^{-(s - s^\prime)2^{2k-2}}
\lVert \e_1(s^\prime, \cdot, t)
\rVert_{L_x^2} (1 + s^\prime 2^{2 k})^{-\sigma_1} ds^\prime \nonumber \\
&\leq
\left( \int_0^s \lVert \e_1(s^\prime, \cdot, t) \rVert_{L_x^2}^2 \ds^\prime \right)^{1/2} 
\left( \int_0^s e^{-(s - s^\prime)2^{2k-1}} 
(1 + s^\prime 2^{2k})^{-2\sigma_1} ds^\prime \right)^{1/2} \nonumber \\
&\lesssim
\lVert \e_1(t) \rVert_{L_{s,x}^2}
\left( s(1 + s 2^{2k - 1})^{-2 \sigma_1} (1 + s 2^{2k})^{-1} \right)^{1/2}.
\label{1stIntegral}
\end{align}
To the remaining integral we also apply Cauchy-Schwarz
and (\ref{2ndCSfactor}):
\begin{align}
\int_0^s &e^{-(s - s^\prime)2^{2k-2}}
\lVert \sqrt{\e_2}(s^\prime, \cdot, t)
\rVert_{L_x^2} (1 + s^\prime 2^{2 k})^{-\sigma_1} ds^\prime \nonumber \\
&\leq
\left( \int_0^s \lVert \e_2(s^\prime, \cdot, t) \rVert_{L_x^1} \ds^\prime \right)^{1/2}
\left( \int_0^s e^{-(s - s^\prime)2^{2k-1}} 
(1 + s^\prime 2^{2k})^{-2\sigma_1} ds^\prime \right)^{1/2} \nonumber \\
&\lesssim
\lVert \e_2(t) \rVert_{L_{s,x}^1}^{1/2}
\left( s(1 + s 2^{2k - 1})^{-2 \sigma_1} (1 + s 2^{2k})^{-1} \right)^{1/2}
\label{2ndIntegral}.
\end{align}
Hence using Cauchy-Schwarz, (\ref{1stIntegral}), and (\ref{2ndIntegral})
in (\ref{presplit}),
we conclude
\begin{align}
\int_0^s &e^{-(s - s^\prime)2^{2k-2}} \lVert \Xi_{hh}(s^\prime, \cdot, t) 
\rVert_{L_x^2} \; ds^\prime \nonumber \\
&\lesssim
2^{-\sigma k} \gamma_k(\sigma) \cC(S, t) (1 + s 2^{2k - 1})^{-\sigma_1}
\left(\lVert \e_1(t) \rVert_{L_{s,x}^2} +
\lVert \e_2(t) \rVert_{L_{s,x}^1}^{1/2}\right).
\label{highhigh subcase}
\end{align}

\nline
\noindent \textbf{Low-High interaction conclusion.}
Combining (\ref{lowhigh subcase}) and (\ref{highhigh subcase}), we conclude
in view of (\ref{LH Def}) and the decomposition (\ref{Further Decomposition}) that
\begin{align}
\mathrm{LH}(s, t) 
\lesssim 
&\;
2^{-\sigma k} \gamma_k(\sigma) \cC(S, t) (1 + s 2^{2k - 1})^{-\sigma_1} \times \nonumber \\
&\left(\lVert \e_1(t) \rVert_{L_s^1 L_x^\infty}^{1/2} +
\lVert \e_1(t) \rVert_{L_{s,x}^2} +
\lVert \e_2(t) \rVert_{L_{s,x}^1}^{1/2}\right).
\label{LH Bound}
\end{align}

\nline
\noindent \textbf{High-Low interaction.}
We now go on to bound the high-low interaction.  By Duhamel
and the triangle inequality it suffices to bound
\begin{equation*}
\mathrm{HL}(s, t) :=
\int_0^s e^{-(s - s^\prime) 2^{2k-2}} \sum_{\lvert k_1 - k\rvert \leq 4}
\lVert P_k (P_{k_1}\varphi(s^\prime, \cdot, t) 
\cdot P_{\leq k - 5} \e_1(s^\prime, \cdot, t)) \rVert_{L_x^2} ds^\prime.
\end{equation*}
By H\"older's inequality, (\ref{B(S,t) con}), and Bernstein's inequality, we have
\begin{align*}
\sum_{\lvert k_1 - k\rvert \leq 4}&
\lVert P_k (P_{k_1}\varphi(s, \cdot, t) \cdot P_{\leq k - 5} \e_1(s, \cdot, t)) \rVert_{L_x^2} \\
&\lesssim
\sum_{\lvert k_1 - k\rvert \leq 4}
\lVert P_{k_1} \varphi \rVert_{L_x^2} \lVert P_{\leq k - 5} \e_1 \rVert_{L_x^\infty} \\
&\lesssim
\lVert P_{\leq k - 5} \e_1(s, \cdot, t) \rVert_{L_x^\infty}
\sum_{\lvert k_1 - k\rvert \leq 4}
(1 + s^\prime 2^{2k_1})^{-\sigma_1} 2^{-\sigma k_1} \gamma_{k_1}(\sigma) \cC(S, t)\\
&\lesssim
2^k \lVert P_{\leq k - 5} \e_1(s, \cdot, t) \rVert_{L_x^2}
2^{-\sigma k} \gamma_k(\sigma) \cC(S, t) (1 + s^\prime 2^{2k})^{-\sigma_1}.
\end{align*}
Hence
\begin{equation*}
\mathrm{HL}(s, t) \lesssim 2^k 2^{-\sigma k} \gamma_k(\sigma) \cC(S, t)
\int_0^s e^{-(s - s^\prime) 2^{2k-2}} (1 + s^\prime 2^{2k})^{-\sigma_1}
\lVert \e_1(s^\prime, \cdot, t) \rVert_{L_x^2} ds^\prime.
\end{equation*}
Bounding the integral as in (\ref{1stIntegral}),
we obtain
\begin{equation}
\mathrm{HL}(s, t) \lesssim 2^{-\sigma k} \gamma_k(\sigma) \cC(S, t)
(1 + s 2^{2k-1})^{-\sigma_1}
\lVert \e_1(t) \rVert_{L_{s,x}^2}.
\label{HL Bound}
\end{equation}

\nline
\noindent \textbf{High-High interaction.}
We conclude with the high-high interaction.  Set
\begin{equation*}
\mathrm{HH}(s, x, t) := \int_0^s e^{-(s - s^\prime)2^{2k - 2}}
\sum_{\substack{ k_1, k_2 \geq k - 4 \\ \lvert k_1 - k_2 \rvert \leq 8}}
\lVert P_k (P_{k_1} \varphi(s, x, t) 
\cdot P_{k_2} \e_1(s, x, t)) \rVert_{L_x^2} ds^\prime.
\end{equation*}
By Bernstein, Cauchy-Schwarz, and (\ref{B(S,t) con}),
\begin{align*}
\sum_{\substack{ k_1, k_2 \geq k - 4 \\ \lvert k_1 - k_2 \rvert \leq 8}}
&\lVert P_k (P_{k_1} \varphi \cdot P_{k_2} \e_1) \rVert_{L_x^2} \\
&\lesssim
\sum_{\substack{k_1, k_2 \geq k - 4 \\ \lvert k_1 - k_2 \rvert \leq 8}}
2^k \lVert P_{k_1} \varphi \rVert_{L_x^2} \lVert P_{k_2} \e_1 \rVert_{L_x^2} \\
&\lesssim
2^k \left( \sum_{k_1 \geq k - 4} \lVert P_{k_1} \varphi \rVert_{L_x^2}^2 \right)^{1/2}
\left( \sum_{k_2 \geq k - 4} \lVert P_{k_2} \e_1 \rVert_{L_x^2}^2 \right)^{1/2} \\
&\lesssim
2^k \left( \sum_{k_1 \geq k - 4} (1 + s^\prime 2^{2k_1})^{-2 \sigma_1} 2^{-2 \sigma k_1}
\gamma_{k_1}(\sigma)^2 \cC(S, t)^2 \right)^{1/2} \lVert \e_1(s, \cdot, t) \rVert_{L_x^2} \\
&= \lVert \e_1(s, \cdot, t)  \rVert_{L_x^2} 2^k \cC(S, t) 
\left( \sum_{k_1 \geq k - 4} (1 + s^\prime 2^{2 k_1})^{-2 \sigma_1} 2^{-2 \sigma k_1}
\gamma_{k_1}(\sigma)^2 \right)^{1/2}.
\end{align*}
We handle the sum as in (\ref{presum}), 
taking advantage of the frequency envelope summation rule (\ref{Sum 2}),
and conclude
\begin{equation}
\mathrm{HH}(s, t) \lesssim 2^{-\sigma k} \gamma_k(\sigma) \cC(S, t)
(1 + s 2^{2k-1})^{-\sigma_1}
\lVert \e_1(t) \rVert_{L_{s,x}^2}.
\label{HH Bound}
\end{equation}

\nline
\noindent \textbf{Wrapping up.}
For the linear term $e^{s \Delta} P_k \varphi$ we have
\begin{align}
\lVert e^{s \Delta} P_k \varphi_0 \rVert_{L_x^2} 
&\leq e^{-s 2^{2k - 2}} \lVert P_k \varphi_0 \rVert_{L_x^2} \nonumber \\
&\leq   e^{-s 2^{2k - 2}} 2^{-\sigma k} \gamma_k(\sigma).
\label{linear term}
\end{align}
Using (\ref{LH Bound}), (\ref{HL Bound}), (\ref{HH Bound}), and (\ref{linear term})
in Duhamel's formula applied to the covariant heat equation (\ref{Covariant Heat Alt}),
we have that for any $s \in [0, S]$, $t \in (-T, T)$,
\begin{align*}
2^{\sigma k} \lVert P_k \varphi(s, \cdot, t) \rVert_{L_x^2} &(1 + s 2^{2k})^{\sigma_1} \\
\lesssim &\;
\gamma_k(\sigma) + \mathrm{LL}(s, t) + \mathrm{LH}(s, t) + \mathrm{HH}(s, t) \\
\lesssim& \;
\gamma_k(\sigma) + \gamma_k(\sigma) \cC(S, t) 
\left( \lVert \e_1(t) \rVert_{L_s^1 L_x^\infty}^{1/2}
+ \lVert \e_2(t) \rVert_{L_{s,x}^1}^{1/2}
+ \lVert \e_1(t) \rVert_{L_{s,x}^2} \right). 
\end{align*}
In view of (\ref{ek4}) with $k = 1$, (\ref{ek1}) with $k = 1$, and (\ref{controlling norm}),
we may split up the $s$-time interval $[0, \infty)$ into $O_{E_0}(1)$ intervals $I_\rho$
on which
$\lVert \e_1(t) \rVert_{L_s^1 L_x^\infty(I_\rho \times \rr)}^{1/2}$,
$ \lVert \e_2(t) \rVert_{L_s^1 L_x^1(I_\rho \times \rr)}^{1/2}$, and
$ \lVert \e_1(t) \rVert_{L_s^2 L_x^2(I_\rho \times \rr)}$ are all simultaneously small uniformly in $t$.  
By iterating a bootstrap argument $O_{E_0}(1)$ times beginning with interval $I_1$, we
conclude that $\cC(s, t) \lesssim 1$ for all $s > 0$, uniformly in $t$. 
Therefore
\begin{equation}
\lVert P_k \varphi(s) \rVert_{L_t^\infty L_x^2} \lesssim
(1 + s 2^{2k})^{-\sigma_1} 2^{-\sigma k} \gamma_k(\sigma)
\label{gamma con}
\end{equation}
for $s \in [0, \infty)$ and $\sigma \geq 2\delta$.
\end{proof}

\begin{rem}
Having proven the quantitative bounds (\ref{Hard Envelope Bounds}) for $\varphi$, 
one may establish as a corollary 
the qualitative bounds (\ref{Soft Envelope Bounds}) for $\varphi$
by using an inductive argument as in the proof of \cite[Lemma 8.3]{BeIoKeTa11}.
We omit the proof, noting
in particular that the argument deriving (\ref{Soft Envelope Bounds})
from (\ref{Hard Envelope Bounds}) does not require a small-energy hypothesis.
\end{rem}

\begin{proof}[Proof of (\ref{Hard Envelope Bounds}) for $v, w$]
We begin by introducing the matrix-valued function
\begin{equation}
R(s, x, t) := \partial_s \varphi(s, x, t) \cdot \varphi(s, x, t)^\dagger - 
\varphi(s, x, t) \cdot \partial_s \varphi(s, x, t)^\dagger,
\label{R Def}
\end{equation}
where here $\varphi$ is thought of as a column vector.  The dagger ``$\dagger$''
denotes transpose.
Using the heat flow equation (\ref{Covariant Heat}) in (\ref{R Def}), we rewrite $R$ as
\begin{align}
R 
&= \Delta \varphi \cdot \varphi^\dagger - \varphi \cdot \Delta \varphi^\dagger \label{R deflap} \\
&= \partial_m ( \partial_m \varphi \cdot \varphi^\dagger - 
\varphi \cdot \partial_m \varphi^\dagger) \label{R defmod}
\end{align}
and proceed to bound its Littlewood-Paley projections $P_k R$ in $L_x^2$.
Noting that by Bernstein we have
\begin{equation}
\lVert P_k (\partial_m (\partial_m \varphi \cdot \varphi^\dagger)) \rVert_{L_x^2} \sim
2^k \lVert P_k (\partial_m \varphi \cdot \varphi^\dagger) \rVert_{L_x^2},
\label{R note}
\end{equation}
we further decompose the nonlinearity $P_k (\partial_m \varphi \cdot \varphi^\dagger)$:
\begin{align}
P_k (\partial_m \varphi \cdot \varphi^\dagger) 
=&
\sum_{\lvert k_2 - k \rvert \leq 4} P_{\leq k - 4} \partial_m \varphi \cdot
P_{k_2} \varphi^\dagger + \nonumber \\
&\sum_{\lvert k_1 - k \rvert \leq 4} P_{k_1} \partial_m \varphi \cdot
P_{\leq k - 4} \varphi^\dagger + \nonumber \\
&\sum_{\substack{ k_1, k_2 \geq k - 4 \\ \lvert k_1 - k_2 \rvert \leq 8}}
P_k ( P_{k_1} \partial_m \varphi \cdot P_{k_2} \varphi^\dagger ).
\label{R decomposition}
\end{align}
By H\"older's and Bernstein's inequalities and by $\lvert \varphi \rvert \equiv 1$ 
and (\ref{gamma con}) with Bernstein,
\begin{align}
\sum_{\lvert k_2 - k \rvert \leq 4} \lVert P_{\leq k - 4} \partial_m \varphi \cdot
P_{k_2} \varphi \rVert_{L_x^2}
&\lesssim
\sum_{\lvert k_2 - k \rvert \leq 4} 2^k
\lVert P_{\leq k - 4} \varphi \rVert_{L_{x}^\infty}
\lVert P_{k_2} \varphi \rVert_{L_x^2} \nonumber \\
&\lesssim 2^k (1 + s 2^{2k})^{-\sigma_1} 2^{-\sigma k} \gamma_k(\sigma).
\label{R LH}
\end{align}
Similarly,
\begin{align}
\sum_{\lvert k_1 - k \rvert \leq 4} \lVert P_{k_1} \partial_m \varphi \cdot
P_{\leq k - 4} \varphi \rVert_{L_x^2}
&\lesssim \sum_{\lvert k_1 - k \rvert \leq 4}
\lVert P_{k_1} \partial_m \varphi \rVert_{L_x^2}
\lVert P_{\leq k - 4} \varphi \rVert_{L_{x}^\infty} \nonumber \\
&\lesssim 2^k (1 + s 2^{2k})^{-\sigma_1} 2^{-\sigma k} \gamma_k(\sigma).
\label{R HL}
\end{align}
Finally, by Bernstein and Cauchy-Schwarz, energy decay, (\ref{gamma con}),
and frequency envelope summation rule (\ref{Sum 2}), we get
\begin{align}
\sum_{\substack{ k_1, k_2 \geq k - 4 \\ \lvert k_1 - k_2 \rvert \leq 8}}
\lVert P_k ( P_{k_1} \partial_m \varphi \cdot P_{k_2} \varphi ) \rVert_{L_x^2}
&\lesssim 
\sum_{\substack{ k_1, k_2 \geq k - 4 \\ \lvert k_1 - k_2 \rvert \leq 8}}
2^k \lVert P_{k_1} \partial_m \varphi \rVert_{L_x^2}
\lVert P_{k_2} \varphi \rVert_{L_x^2} \nonumber \\
&\lesssim 
2^k \sum_{k_2 \geq k - 4} \lVert P_{k_2} \varphi \rVert_{L_x^2} \nonumber \\
&\lesssim 
2^k \sum_{k_1 \geq k - 4} (1 + s 2^{2k_1})^{-\sigma_1}
2^{-\sigma k_1} \gamma_{k_1}(\sigma) \nonumber \\
&\lesssim 
2^k (1 + s 2^{2k})^{-\sigma_1} 2^{-\sigma k} \gamma_k(\sigma).
\label{R HH}
\end{align}
Using the decomposition (\ref{R decomposition}) and
combining the cases (\ref{R LH}), (\ref{R HL}), and (\ref{R HH})
to control (\ref{R note}), we conclude from the representation (\ref{R defmod})
of $R$ that for fixed $t \in (-T, T)$,
\begin{equation*}
2^{\sigma k} \lVert P_k R(s, \cdot, t) \rVert_{L_x^2}
\lesssim 2^{2k} (1 + s 2^{2k})^{-\sigma_1} \gamma_k(\sigma).
\end{equation*}
As this estimate is uniform in $T$, it follows that
\begin{equation}
2^{\sigma k} \lVert P_k R(s) \rVert_{L_t^\infty L_x^2}
\lesssim 2^{2k} (1 + s 2^{2k})^{-\sigma_1} \gamma_k(\sigma).
\label{RsDecay}
\end{equation}
By arguing as in \cite[Lemma 8.4]{BeIoKeTa11}, one may obtain the
qualitative estimate
\begin{align}
\sup_{s \geq 0} &\left( (1 + s)^{(\sigma + 2)/2}
\lVert \partial_x^\sigma \partial_t^\rho R(s) \rVert_{L_t L_x^2} \right. \nonumber \\
&\left. (1 + s)^{(\sigma + 3)/2} 
\lVert \partial_x^\sigma \partial_t^\rho R(s) \rVert_{L_{t,x}^\infty} \right)
< \infty.
\label{R qual}
\end{align}
From the Duhamel representation of $\varphi$ and the explicit formula
for the heat kernel, one can easily show that qualitative bound
\begin{footnote}{
We may alternately invoke (\ref{R qual}) as in \cite{BeIoKeTa11}.}
\end{footnote}
\begin{equation*}
\int_0^\infty \lVert R(s, \cdot, t) \rVert_{L_x^\infty} \ds
\lesssim_\varphi 1
\end{equation*}
as in \cite[\S 7]{Sm09}.
Hence we may define $v$ as the unique solution of the ODE
\begin{equation}
\partial_s v = R(s) \cdot v
\quad \quad \text{and} \quad \quad
v(\infty) = Q^\prime,
\label{R ode}
\end{equation}
where $Q^\prime \in \mathbf{S}^2$ is chosen so that $Q \cdot Q^\prime = 0$.
This indeed coincides with the definition given in \cite{Sm09}, since
(\ref{R ode}) is nothing other than the parallel transport condition
$(\varphi^* \nabla)_s v = 0$ written explicitly in the setting 
$\mathbf{S}^2 \hookrightarrow \mathbf{R}^3$.  Smoothness and basic
convergence properties follow as in \cite{Sm09},
to which we refer the reader for the precise results and proofs.  
Our goal here is to exploit
(\ref{R ode}) and (\ref{RsDecay}) to prove (\ref{Hard Envelope Bounds})
for $v$.

Using 
$\int_0^\infty \lVert \partial_x^\sigma \partial_t^\rho R(s)) \rVert_{L_{t,x}^\infty} \ds < \infty$
from (\ref{R qual}), we conclude
\begin{equation}
\sup_{s \geq 0} (1 + s)^{(\sigma + 1)/2}
\lVert \partial_x^\sigma \partial_t^\rho (v(s) - Q^\prime) \rVert_{L_{t,x}^\infty} < \infty
\label{v qual}
\end{equation}
for $\sigma, \rho \in \mathbf{Z}_+$.
Integrating (\ref{R ode}) in $s$ from infinity, we get
\begin{equation}
v(s) - Q^\prime +
\int_s^\infty R(s^\prime) \cdot Q^\prime \ds^\prime
=
- \int_s^\infty R(s^\prime) \cdot (v(s^\prime) - Q^\prime) \ds^\prime,
\end{equation}
which, combined with estimates (\ref{R qual}) and (\ref{v qual}), implies
\begin{equation}
\sup_{s \geq 0} \sup_{k \in \mathbf{Z}} (1 + s)^{\sigma / 2} 2^{\sigma k}
\lVert P_k \partial_t^\rho v(s) \rVert_{L_t^\infty L_x^2} < \infty,
\label{pkv qual}
\end{equation}
i.e., (\ref{Soft Envelope Bounds}) for $v$.
Projecting (\ref{R ode}) to frequencies $\sim 2^k$ and integrating in $s$,
we obtain
\begin{equation}
P_k(v(s)) = -\int_s^\infty P_k(R(s^\prime) \cdot v(s^\prime)) \; ds^\prime.
\label{Pkvs}
\end{equation}
Set
\begin{equation*}
\cC_1(S, t) := \sup_{\sigma \in [2\delta, \sigma_1]} \sup_{s \in [S, \infty)} \sup_{k \in \mathbf{Z}}
\gamma_k(\sigma)^{-1} (1 + s 2^{2k})^{\sigma_1 - 1} 2^{\sigma k}
\lVert P_k v(s, \cdot, t) \rVert_{L_x^2}.
\end{equation*}
That $\cC_1(S, t) < \infty$ follows from (\ref{pkv qual}) and 
$\sup_{k \in \mathbf{Z}} \gamma_k(\sigma)^{-1} 2^{- \delta \lvert k \rvert} < \infty$.
Consequently, for $s \in [S, \infty)$,
\begin{equation}
\lVert P_k v(s, \cdot, t) \rVert_{L_x^2}
\lesssim
\cC_1(S, t)
(1 + s 2^{2k})^{-\sigma_1 + 1} 2^{-\sigma k} \gamma_k(\sigma).
\label{B1 con}
\end{equation}

We perform the Littlewood-Paley decomposition
\begin{align}
P_k (R(s) v(s)) =& 
\sum_{\lvert k_2 - k \rvert \leq 4 } P_k( P_{\leq k - 4} R(s) P_{k_2} v(s) ) + \nonumber \\
&\sum_{ \lvert k_1 - k \rvert \leq 4}  P_k( P_{k_1} R(s) P_{\leq k - 4} v(s) ) + \nonumber \\
&\sum_{k_2 \geq k - 4} P_k( P_{\geq k - 4} R(s) P_{k_2} v(s) )
\label{R decomp}
\end{align}
and proceed to consider individually the various frequency interactions.
By H\"older's inequality, Bernstein's inequality, and (\ref{B1 con}),
\begin{align}
\sum_{\lvert k_2 - k \rvert \leq 4 }& \lVert  P_k( P_{\leq k - 4} R(s) P_{k_2} v(s) ) \rVert_{L_x^2}
\nonumber \\
&\lesssim
\sum_{\lvert k_2 - k \rvert \leq 4 }
\lVert P_{\leq k - 4} R(s) \rVert_{L_x^2}
\lVert P_{k_2} v(s) \rVert_{L_{x}^\infty} \nonumber \\
&\lesssim
 \lVert R(s) \rVert_{L_x^2}
\sum_{\lvert k_2 - k \rvert \leq 4} 2^{k_2} \lVert P_{k_2} v(s) \rVert_{L_x^2} \nonumber \\
&\lesssim
\lVert R(s) \rVert_{L_x^2} 2^k 2^{-\sigma k} \gamma_k(\sigma)
(1 + s 2^{2k})^{-\sigma_1 + 1} \cC_1(S, t).
\label{R LHb}
\end{align}
By H\"older's inequality, $\lvert v \rvert \equiv 1$, and (\ref{RsDecay}),
\begin{align}
\sum_{ \lvert k_1 - k \rvert \leq 4} \lVert  P_k( P_{k_1} R(s) P_{\leq k - 4} v(s) ) \rVert_{L_x^2}
&\lesssim \lVert P_{\leq k - 4} v(s) \rVert_{L_{x}^\infty}
\sum_{ \lvert k_1 - k \rvert \leq 4} \lVert P_{k_1} R(s) \rVert_{L_x^2} \nonumber \\
&\lesssim 2^{2k} (1 + s 2^{2k})^{-\sigma_1} 2^{-\sigma k} \gamma_k(\sigma).
\label{R HLb}
\end{align}
From Bernstein's inequality, Cauchy-Schwarz, (\ref{B1 con}), 
and $\sigma > 2 \delta$ with (\ref{Sum 2}), it follows that
\begin{align}
 \sum_{k_2 \geq k - 4}&
 \lVert  P_k( P_{\geq k - 4} R(s) P_{k_2} v(s) ) \rVert_{L_x^2} \nonumber \\
&\lesssim
\sum_{k_2 \geq k - 4} 2^k \lVert P_{\geq k - 4} R(s) P_{k_2} v(s) \rVert_{L_x^1} \nonumber \\
&\lesssim \lVert R(s) \rVert_{L_x^2} 2^k
\sum_{k_2 \geq k - 4} \lVert P_{k_2} v(s) \rVert_{L_x^2} \nonumber \\
&\lesssim \lVert R(s) \rVert_{L_x^2} 2^k
\sum_{k_2 \geq k - 4}
2^{-\sigma k_2} \gamma_{k_2}(\sigma) (1 + s 2^{2k_2})^{-\sigma_1 + 1} \cC_1(S, t) \nonumber \\
&\lesssim \lVert R(s) \rVert_{L_x^2}
2^k 2^{-\sigma k} \gamma_k(\sigma) (1 + s 2^{2k})^{-\sigma_1 + 1} \cC_1(S, t).
\label{R HHb}
\end{align}
Using the decomposition (\ref{R decomp}) in (\ref{Pkvs}) and combining the estimates
(\ref{R LHb}), (\ref{R HLb}), and (\ref{R HHb}) implies
\begin{align*}
2^{\sigma k} \lVert P_k v(s) \rVert_{L_x^2} \leq&
\int_s^\infty 2^{\sigma k} \lVert P_k (R(s^\prime) v(s^\prime)) \rVert_{L_x^2} \; ds^\prime \\
\lesssim& \;
\gamma_k(\sigma) \int_s^\infty 2^{2k} (1 + s^\prime 2^{2k})^{-\sigma_1} \; ds^\prime \\
&+ \cC_1(s, t) \gamma_k(\sigma) \int_s^\infty 
\lVert R(s^\prime) \rVert_{L_x^2} 2^k (1 + s^\prime 2^{2k})^{-\sigma_1 + 1} \; ds^\prime.
\end{align*}
Applying Cauchy-Schwarz in $s$, we obtain
\begin{align}
2^{\sigma k} \lVert P_k v(s) \rVert_{L_x^2} 
\lesssim& \;
\gamma_k(\sigma) \int_s^\infty 2^{2k} (1 + s^\prime 2^{2k})^{-\sigma_1} \; ds^\prime \nonumber \\
&+ \cC_1(s, t) \gamma_k(\sigma)
\left( \int_s^\infty \lVert R(s^\prime) \rVert_{L_x^2}^2 \; ds^\prime \right)^{1/2} \times \nonumber \\
&\times \left( \int_s^\infty 2^{2k} (1 + s^\prime 2^{2k})^{-2 \sigma_1 + 2} 
\ds^\prime \right)^{1/2} \nonumber \\
&\lesssim 
\gamma_k(\sigma) + \cC_1(s, t) \gamma_k(\sigma)
\left(\int_s^\infty \lVert R(s^\prime) \rVert_{L_x^2}^2 \ds^\prime \right)^{1/2}.
\label{R bootstrap prep}
\end{align}

As noted in (\ref{ord by cov}), it holds that
$\lvert \Delta \varphi \rvert \leq \sqrt{\e_2} + \e_1$, and so it follows
from the representation (\ref{R deflap}) of $R$ that
\begin{equation}
\lvert R(s, x, t) \rvert \leq \lvert \e_1(s, x, t) \rvert + 
\lvert \sqrt{\e_2}(s, x, t) \rvert.
\label{R B}
\end{equation}
As (\ref{R B}) implies
\begin{equation*}
\int_0^\infty \lVert R(s) \rVert_{L_x^2}^2 \ds \lesssim
\lVert \e_2 \rVert_{L_{s,x}^1} +
\lVert \e_1 \rVert_{L_{s,x}^2}^2,
\end{equation*}
we therefore  in view of (\ref{ek1}) with $k = 1$ and (\ref{controlling norm})
may choose $S$ large so that the integral of the $R$ term 
in (\ref{R bootstrap prep}) is small, say $\leq \varepsilon$.  Then  
\begin{equation*}
\cC_1(S, t) \lesssim 1 + \varepsilon \cC_1(S, t)
\end{equation*}
so that $\cC_1(S) \lesssim 1$ for such $S$.
In fact, together (\ref{ek1}) and (\ref{controlling norm}) imply that
we may divide the time interval $[0, \infty)$
into $O_{E_0}(1)$ subintervals $I_\rho$ so that on each such subinterval
\begin{equation*}
\int_{I_\rho} \lVert R(s) \rVert_{L_x^2}^2 \ds \leq \varepsilon^2.
\end{equation*}
Hence by a simple iterative bootstrap argument we conclude 
\begin{equation}
\cC_1(0, t) \lesssim 1.
\label{B1t}
\end{equation}
As (\ref{B1t}) is uniform in $t$, we have
\begin{equation}
\lVert P_k v(s, \cdot, t) \rVert_{L_x^2}
\lesssim
(1 + s 2^{2k})^{-\sigma_1 + 1} 2^{-\sigma k} \gamma_k(\sigma).
\label{Pkv bound}
\end{equation}
By repeating the above argument with $w$ in place of $v$
(and appropriately modifying the boundary condition at $\infty$ in
(\ref{R ode})), we get
\begin{equation}
\lVert P_k w(s, \cdot, t) \rVert_{L_x^2}
\lesssim
(1 + s 2^{2k})^{-\sigma_1 + 1} 2^{-\sigma k} \gamma_k(\sigma).
\label{Pkw bound}
\end{equation}
and
\begin{equation*}
\sup_{s \geq 0} \sup_{k \in \mathbf{Z}} (1 + s)^{\sigma / 2} 2^{\sigma k}
\lVert P_k \partial_t^\rho w(s) \rVert_{L_t^\infty L_x^2} < \infty,
\end{equation*}
i.e., (\ref{Hard Envelope Bounds}) and (\ref{Soft Envelope Bounds}) respectively
for $w$.
\end{proof}

\begin{proof}[Proof of (\ref{Hard Field Bounds})]
Recall that
\begin{align}
\psi_m &= v \cdot \partial_m \varphi + i w \cdot \partial_m \varphi \nonumber \\
&= - \partial_m v \cdot \varphi - i \partial_m w \cdot \varphi.
\label{psi dmv}
\end{align}
Our first aim is to control $\lVert P_k \psi_x \rVert_{L_t^\infty L_x^2}$.
Toward this end, we perform a Littlewood-Paley decomposition of $\partial_m v \cdot \varphi$:
\begin{align}
P_k ( \partial_m v \cdot \varphi ) =&
\sum_{\lvert k_2 - k \rvert \leq 4}
P_k \left( P_{\leq k - 5} \partial_m v \cdot P_{k_2} \varphi \right) + \nonumber \\
&\sum_{\lvert k_1 - k \rvert \leq 4}
P_k \left( P_{k_1} \partial_m v \cdot P_{\leq k - 5} \varphi \right) + \nonumber \\
&\sum_{\substack{ k_1, k_2 \geq k - 4 \\ \lvert k_1 - k_2 \rvert \leq 8}}
P_k \left( P_{k_1} \partial_m v \cdot P_{k_2} \varphi \right).
\label{dmv phi lp}
\end{align}
To control the low-high frequency term we apply
H\"older's inequality, energy decay, and (\ref{gamma con}) with Bernstein's inequality:
\begin{align}
\sum_{\lvert k_2 - k \rvert \leq 4}
\lVert P_k (P_{\leq k - 5} \partial_m v \cdot P_{k_2} \varphi) \Vert_{L_x^2} 
&\lesssim
\sum_{\lvert k_2 - k \rvert \leq 4}
\lVert P_{\leq k - 5} \partial_m v \rVert_{L_x^2} \lVert P_{k_2} \varphi \rVert_{L_x^\infty} \nonumber \\
&\lesssim
(1 + s 2^{2k})^{-\sigma_1} 2^k 2^{-\sigma k} \gamma_k(\sigma).
\label{dmv lh}
\end{align}
We control the high-low frequency term by using
H\"older's inequality, $\lvert \varphi \rvert \equiv 1$, and (\ref{Pkv bound}):
\begin{align}
\sum_{\lvert k_1 - k \rvert \leq 4}
\lVert P_k( P_{k_1} \partial_m v \cdot P_{\leq k - 5} \varphi) \rVert_{L_x^2}
&\lesssim
\sum_{\lvert k_1 - k \rvert \leq 4}
\lVert P_{k_1} \partial_m v \rVert_{L_x^2} \lVert P_{\leq k - 5} \varphi \rVert_{L_x^\infty} \nonumber \\
&\lesssim
(1 + s 2^{2k})^{-\sigma_1} 2^k 2^{-\sigma k} \gamma_k(\sigma).
\label{dmv hl}
\end{align}
To control the high-high frequency term, we use Bernstein's inequality and Cauchy-Schwarz,
energy conservation and (\ref{gamma con}), and (\ref{Sum 2}):
\begin{align}
\sum_{\substack{ k_1, k_2 \geq k - 4 \\ \lvert k_1 - k_2 \rvert \leq 8}}
\lVert P_k (P_{k_1} \partial_m v \cdot P_{k_2} \varphi ) \rVert_{L_x^2}
&\lesssim 
\sum_{\substack{ k_1, k_2 \geq k - 4 \\ \lvert k_1 - k_2 \rvert \leq 8}}
2^k \lVert P_{k_1} \partial_m v \rVert_{L_x^2} \lVert P_{k_2} \varphi \rVert_{L_x^2} \nonumber \\
&\lesssim
2^k \sum_{ k_2 \geq k - 4 }
(1 + s 2^{2 k_2})^{-\sigma_1} 2^{-\sigma k_2} \gamma_{k_2}(\sigma) \nonumber \\
&\lesssim
(1 + s 2^{2k})^{-\sigma_1} 2^k 2^{-\sigma k} \gamma_{k}(\sigma)
\label{dmv hh}
\end{align}
We conclude using (\ref{dmv lh}), (\ref{dmv hl}), and (\ref{dmv hh}) 
in representation (\ref{dmv phi lp}) that
\begin{equation}
\lVert P_k (\partial_m v \cdot \varphi) \rVert_{L_x^2}
\lesssim
(1 + s 2^{2k})^{-\sigma_1} 2^k 2^{-\sigma k} \gamma_k(\sigma).
\label{bound for dmv phi}
\end{equation}
By repeating the argument with $w$ in place of $v$, it follows that
(\ref{bound for dmv phi}) also holds with $w$ in place of $v$.  Therefore,
referring back to (\ref{psi dmv}), we conclude
\begin{equation*}
\lVert P_k \psi_m \rVert_{L_x^2}
\lesssim
(1 + s 2^{2k})^{-\sigma_1} 2^k 2^{-\sigma k} \gamma_k(\sigma).
\end{equation*}
As this bound is uniform in $t$, (\ref{Hard Field Bounds}) holds for $\psi_m$.

Recalling that
\begin{equation*}
A_m = \partial_m v \cdot w,
\end{equation*}
and repeating the above argument with $w$ in place of $\varphi$
and (\ref{Pkw bound}) in place of (\ref{gamma con}),
we conclude
\begin{equation*}
\lVert P_k A_x (s) \rVert_{L_t^\infty L_x^2} \lesssim
(1 + s 2^{2k})^{-\sigma_1 + 1} 2^k 2^{-\sigma k} \gamma_k(\sigma).
\end{equation*}
\end{proof}

%% file: ParabolicEstimates.tex
\section{Proofs of parabolic estimates} \label{S:Parabolic Estimates}

The purpose of this section is to prove the 
parabolic heat-time estimates stated in \S \ref{SS:ParabolicEstimates}.
Many of these estimates have counterparts in \cite{BeIoKeTa11}.
Nevertheless, our proofs are more involved since
we only require energy dispersion, which is weaker
than the small-energy assumption made in \cite{BeIoKeTa11}.
Some of the $L^p$ estimates in \S \ref{SS:CCC} are new.

Throughout we assume $\varepsilon_1$ energy dispersion on the initial data
as stated in (\ref{iED})
and we assume that the bootstrap hypothesis
(\ref{Main Bootstrap}) holds.
Let $\sigma_1 \in \mathbf{Z}_+$ be positive and fixed.
We work exclusively with $\sigma \in [0, \sigma_1 - 1]$, even if this is not always explicitly stated.
Set $\varepsilon =  \varepsilon_1^{7/5}$ for short.

In this section we extensively use the spaces defined via (\ref{SomegaDef}).
They provide a crucial gain in high-high frequency interactions,
which is
captured in Lemmas \ref{BIKT Lemma 5.1} and \ref{BIKT Lemma 5.4}

\begin{lem}
Let $f \in L^2_{k_1}(T)$, where $\lvert k_1 - k \rvert \leq 20$,
let $0 \leq \omega^\prime\leq 1/2$, and let $h \in L^2_k(T)$.
Then
\begin{align*}
\lVert P_k (fg) \rVert_{F_k(T)} &\lesssim \lVert f \rVert_{F_{k_1}(T)} \lVert g \rVert_{L_{t,x}^\infty} \\
\lVert P_k (fg) \rVert_{S_k^{\omega^\prime}(T)} &\lesssim \lVert f \rVert_{F_{k_1}(T)}
2^{k \omega^\prime} \lVert g \rVert_{L_x^{2 / \omega^\prime} L_t^\infty} \\
\lVert h \rVert_{L_{t,x}^\infty} + 2^{k \omega^\prime} 
\lVert h \rVert_{L_x^{2 / \omega^\prime} L_t^\infty} &\lesssim 2^k \lVert h \rVert_{F_k(T)}.
\end{align*}
Moreover, for $f_{k_1}, g_{k_2}$ belonging to
$L^2_{k_1}(T), L^2_{k_2}(T)$
respectively, and with
$\lvert k_1 - k_2 \rvert \leq 8$, we have
\begin{equation*}
\lVert P_k (f_{k_1} g_{k_2}) \rVert_{F_k(T) \cap S_k^{1/2}(T)} \lesssim
2^k 2^{(k_2 - k)(1 - \omega)} \lVert f_{k_1} \rVert_{S_{k_1}^\omega(T)}
\lVert g_{k_2} \rVert_{S_{k_2}^0(T)}.
\end{equation*}
\label{Parabolic Bilinear Estimates}
\end{lem}
\begin{proof}
For the proofs, see \cite[\S 3]{BeIoKeTa11}.
\end{proof}

\begin{lem}
Assume that $T \in (0, 2^{2 \mathcal{K}}]$, $f, g \in H^{\infty, \infty}(T)$, 
$P_k f \in F_k(T) \cap S_k^\omega(T)$, $P_k g \in F_k(T)$ for some $\omega \in [0, 1/2]$
and all $k \in \mathbf{Z}$, and
\begin{equation*}
\alpha_k = \sum_{\lvert j - k \rvert \leq 20} \lVert P_j f \rVert_{F_j(T) \cap S_j^\omega(T)},
\quad
\beta_k = \sum_{\lvert j - k \rvert \leq 20} \lVert P_j g \rVert_{F_j(T)}.
\end{equation*}
Then, for any $k \in \mathbf{Z}$,
\begin{equation*}
\lVert P_k (fg) \rVert_{F_k(T) \cap S_k^{1/2}(T)} \lesssim
\sum_{j \leq k} 2^j (\beta_k \alpha_j + \alpha_k \beta_j) +
2^k \sum_{j \geq k} 2^{(j - k)(1 - \omega)} \alpha_j \beta_j.
\end{equation*}
\label{BIKT Lemma 5.1}
\end{lem}
\begin{proof}
For the proof, see \cite[\S 5]{BeIoKeTa11}.
\end{proof}

\subsection{Derivative field control} \label{SS:DFC}
The main purpose of this subsection is to establish the estimate
(\ref{PsiF}), which states
\begin{equation*}
\lVert P_k \psi_m(s) \rVert_{F_k(T)} \lesssim (1 + s 2^{2k})^{-4} 2^{-\sigma k} b_k(\sigma).
\end{equation*}
In the course of the proof we shall also establish auxiliary estimates useful elsewhere.
Estimate (\ref{PsiF}) plays a key role in controlling the nonlinear paradifferential flow,
allowing us to gain regularity by integrating in heat time.
The proof uses a bootstrap argument and exploits the Duhamel formula.

Recall that the fields $\psi_\alpha$, $A_\alpha$, $\alpha = 1, 2, 3$, ($\psi_3 \equiv \psi_t,
A_3 \equiv A_t$)
satisfy (\ref{genPsiEQ}), which states
\begin{equation*}
(\partial_s - \Delta) \psi_\alpha = U_\alpha.
\end{equation*}
We use representation (\ref{Heat Nonlinearity}) of the heat nonlinearity:
\begin{equation*}
U_\alpha := 2i A_\ell \partial_\ell \psi_\alpha + i (\partial_\ell A_\ell) \psi_\alpha
- A_x^2 \psi_\alpha + i \Im(\psi_\alpha \overline{\psi_\ell}) \psi_\ell.
\end{equation*}
Hence $\psi_\alpha$ admits the representation
\begin{equation}
\psi_\alpha(s) = e^{s \Delta} \psi_\alpha(s_0) +
\int_{s_0}^s e^{(s - s^\prime)\Delta} U_\alpha(s^\prime) \; ds^\prime
\label{Psi Duhamel}
\end{equation}
for any $s \geq s_0 \geq 0$.

For each $k \in \mathbf{Z}$, set
\begin{equation*}
a(k) := \sup_{s \in [0, \infty)} (1 + s 2^{2k})^4
\sum_{m = 1, 2} \lVert P_k \psi_m(s) \rVert_{F_k(T)},
\end{equation*}
and for $\sigma \in [0, \sigma_1 - 1]$ introduce the frequency envelopes
\begin{equation}
a_k(\sigma) = \sup_{j \in \mathbf{Z}} 2^{- \delta \lvert k - j \rvert} 2^{\sigma j} a(j).
\label{a Envelope}
\end{equation}
The frequency envelopes $a_k(\sigma)$ are finite 
and in $\ell^2$ by (\ref{Soft Field Space Bounds})
and (\ref{FsoftBound}).

Our goal is to show
$a_k(\sigma) \lesssim b_k(\sigma)$, which in particular implies
(\ref{PsiF}).  

\begin{lem}
Suppose that $\psi_x$ satisfies the bootstrap condition
\begin{equation}
\lVert P_k \psi_x(s) \rVert_{F_k(T) \cap S_k^{1/2}(T)} \leq
\varepsilon_p^{-1/2} b_k (1 + s 2^{2k})^{-4}.
\label{Parabolic  Bootstrap Assumption}
\end{equation}
Then (\ref{PsiF}) holds.
\label{Parabolic Bootstrap}
\end{lem}
We can take $\varepsilon_p = \varepsilon_1^{1/10}$, for instance.
As in \cite{BeIoKeTa11}, this result may be strengthened to
\begin{cor}
The estimate (\ref{PsiF}) holds even when the bootstrap hypothesis
(\ref{Parabolic Bootstrap Assumption}) is dropped.
\end{cor}
\begin{proof}
Directly apply the argument of \cite[Corollary 4.4]{BeIoKeTa11}, which we omit.
\end{proof}

The sequence of lemmas we prove in order to establish Lemma \ref{Parabolic Bootstrap}
culminates in Lemma \ref{Nonlinear Evolution Bound}, which controls the nonlinear
term of the Duhamel formula (\ref{Psi Duhamel}) by $2^{-\sigma k} a_k(\sigma)$
along with suitable decay and an epsilon-gain arising from energy dispersion.
Its immediate predecessor, Lemma \ref{U Bound}, controls $P_k U_m$ in $F_k(T)$.

Referring back to (\ref{Heat Nonlinearity}) and
seeing as how $U_m$ contains the term $2 i A_\ell \partial_\ell \psi_m$,
we see that in order to apply the parabolic estimates of Lemma \ref{Parabolic Bilinear Estimates}
toward controlling $P_k U_m$, it is necessary that we first control
 $P_k A_m$ in $F_k(T)$ in terms of the
frequency envelopes $\{ a_\ell(\sigma) \}$, and it is to this that we now turn.

For $k, k_0 \in \mathbf{Z}$ and $s \in [2^{2 k_0 - 1}, 2^{2 k_0 + 1})$, set
\begin{displaymath}
b_{k, s}(\sigma) = \left\{
\begin{array}{ll}
\sum_{j = k}^{-k_0} a_j a_j(\sigma)
& k + k_0 \leq 0 \\ \\
2^{k + k_0} a_{-k_0} a_k(\sigma)
& k + k_0 \geq 0.
\end{array}
\right.
\end{displaymath}
Let $\cC$ be the smallest number in $[1, \infty)$ such that
\begin{equation}
\lVert P_k A_m(s) \rVert_{F_k(T) \cap S_k^{1/2}(T)} \leq \cC (1 + s 2^{2k})^{-4} 2^{-\sigma k} b_{k,s}(\sigma)
\label{A Bootstrap}
\end{equation}
for all $s \in [0, \infty)$, $k \in \mathbf{Z}$, $m = 1, 2$, and $\sigma \in [0, \sigma_1 - 1]$.
While this constant is indeed finite, it is not a priori controlled by energy.
To show that $\cC$ is indeed controlled by energy, we use the integral representation
\begin{equation}
A_m(s) = -\sum_{\ell = 1, 2} \int_s^\infty \Im (\overline{\psi_m}( \partial_\ell \psi_\ell
+ i A_\ell \psi_\ell))(r) \; dr
\label{A Integral}
\end{equation}
and seek to control the Littlewood-Paley projection of the integrand in $F_k(T) \cap S_k^{1/2}(T)$.
We treat differently the two types of terms in (\ref{A Integral}) that need to be controlled. 
In Lemma \ref{Psi Quadratic} we bound terms of the sort
$P_k (\psi_x \overline{\psi_x})$ and $P_k (\psi_x \partial_x \overline{\psi_x})$ in 
$F_k(T) \cap S_k^{1/2}(T)$.  In Lemma \ref{Psi A Cubic} we combine the estimate on
$P_k (\psi_x \overline{\psi_x})$ with (\ref{A Bootstrap}) to obtain control on
$P_k (\psi_x \overline{\psi_x} A_x)$, gaining an epsilon from energy dispersion.
Using (\ref{A Integral}) and exploiting the epsilon gain from energy dispersion
will lead us to the conclusion of Lemma \ref{A Bound}: $\cC \lesssim 1$.

We use the following bracket notation in the sequel:
\begin{equation*}
\langle f \rangle := (1 + f^2)^{1/2}.
\end{equation*}

\begin{lem}
For any $f, g \in \{ \psi_m, \overline{\psi_m} : m = 1,2\}$, $r \in [2^{2j - 2}, 2^{2j + 2}]$,
$j \in \mathbf{Z}$, $i = 1, 2$, and $\sigma \in [0, \sigma_1 - 1]$, we have the bounds
\begin{equation}
\lVert P_k (f(r) g(r) )\rVert_{F_k(T) \cap S_k^{1/2}(T)} 
\lesssim
\langle 2^{j + k} \rangle^{-8} 2^{-\sigma k} 2^{-j} a_{-j} a_{\max(k, -j)}(\sigma).
\label{Psi Psi Bilinear Parabolic}
\end{equation}
and
\begin{equation}
\lVert P_k (f(r) \partial_i g(r) ) \rVert_{F_k(T) \cap S_k^{1/2}(T)} 
\lesssim
\langle 2^{j + k} \rangle^{-8} 2^{-\sigma k} 2^{-j} a_{-j} (2^k a_k(\sigma) + 2^{-j} a_{-j}(\sigma)).
\label{Psi dPsi Bilinear Parabolic}
\end{equation}
\label{Psi Quadratic}
\end{lem}
\begin{proof}
By Lemma \ref{BIKT Lemma 5.1} with $\omega = 0$ we have
\begin{equation}
\lVert P_k (fg) \rVert_{F_k(T) \cap S_k^{1/2}(T)} 
\lesssim
\sum_{\ell \leq k} 2^\ell \alpha_k \beta_\ell
+ \sum_{\ell \geq k} 2^\ell \alpha_\ell \beta_\ell,
\label{PPBP 5.1}
\end{equation}
where, due to the definition (\ref{a Envelope}), $\alpha_k$ and $\beta_k$ satisfy
\begin{equation}
\alpha_k \lesssim \langle 2^{j + k} \rangle^{-8} 2^{-\sigma k} a_k(\sigma),
\quad \quad
\beta_k \lesssim \langle 2^{j + k} \rangle^{-8} a_k.
\label{PQ ab}
\end{equation}
Turning to the high-low frequency interaction first, we have
using (\ref{PQ ab}) and the 
frequency envelope property (\ref{Slowly Varying}) that
\begin{equation}
\sum_{\ell \leq k} 2^\ell \alpha_k \beta_\ell 
\lesssim 
\langle 2^{j + k} \rangle^{-8} 2^{-\sigma k} 2^{-j} a_{-j}
\sum_{\ell \leq k} \langle 2^{j + \ell} \rangle^{-8}
2^{j + \ell} 2^{\delta \lvert j + \ell \rvert} a_k(\sigma).
\label{PPBP HL}
\end{equation}
Thus it remains to show that
\begin{equation}
\sum_{\ell \leq k} \langle 2^{j + \ell} \rangle^{-8}
2^{j + \ell} 2^{\delta \lvert j + \ell \rvert} a_k(\sigma)
\lesssim 
a_{\max(k, -j)}(\sigma)
\label{PPBP RTS1},
\end{equation}
which follows from pulling out
a factor of $a_k(\sigma)$ or $a_{-j}(\sigma)$, according to
whether $k + j \geq 0$ or $k + j < 0$, and then summing the remaining geometric series.
In case $k + j < 0$ we pull out a factor of $a_{-j}(\sigma)$ via (\ref{Slowly Varying}).

Turning to the high-high frequency interaction term, we have
\begin{equation}
\sum_{\ell \geq k} 2^\ell \alpha_\ell \beta_\ell 
\lesssim
\langle 2^{j + k} \rangle^{-8} 2^{-\sigma k} 2^{-j} a_{-j} 
\sum_{\ell \geq k} \langle 2^{j + \ell} \rangle^{-8} 2^{j + \ell} 
2^{\delta \lvert j + \ell \rvert} a_\ell(\sigma),
\label{PPBP HH}
\end{equation}
and so it remains to show that
\begin{equation}
\sum_{\ell \geq k} \langle 2^{j + \ell} \rangle^{-8} 2^{j + \ell} 
2^{\delta \lvert j + \ell \rvert} a_\ell(\sigma)
\lesssim a_{\max(k, -j)}(\sigma).
\label{PPBP RTS2}
\end{equation}
When $k + j \geq 0$, we have using (\ref{Sum 2})
\begin{equation*}
\sum_{\ell \geq k} \langle 2^{j + \ell} \rangle^{-8} 2^{j + \ell} 
2^{\delta \lvert j + \ell \rvert} a_\ell(\sigma)
\lesssim
a_k(\sigma) \sum_{\ell \geq k} 2^{(2 \delta - 1)(j + \ell)} 
\lesssim a_k(\sigma). 
\label{PPBP S2 C1}
\end{equation*}
If $k + j \leq 0$, then
we control the sum with (\ref{Sum 1}) if $\ell + j < 0$ and
with (\ref{Sum 2}) if $\ell + j \geq 0$.
Hence (\ref{PPBP RTS2}) holds.

Together (\ref{PPBP 5.1})--(\ref{PPBP RTS2}) 
imply (\ref{Psi Psi Bilinear Parabolic}).

To establish (\ref{Psi dPsi Bilinear Parabolic}) we follow a similar strategy.
By Lemma \ref{BIKT Lemma 5.1} with $\omega = 0$ we have
\begin{equation}
\lVert P_k (f \partial_i g) \rVert_{F_k(T) \cap S_k^{1/2}(T)} \lesssim
\sum_{\ell \leq k} 2^\ell \alpha_\ell \beta_k +
\sum_{\ell \geq k} 2^\ell \alpha_k \beta_\ell + \sum_{\ell \geq k} 2^\ell \alpha_\ell \beta_\ell,
\label{PdPBP 5.1}
\end{equation}
where
\begin{equation}
\alpha_k \lesssim \langle 2^{j + k} \rangle^{-8} 2^{-\sigma k} a_k(\sigma)
\label{Alpha Bound PdP}
\end{equation}
for any $\sigma \in [0, \sigma_1 - 1]$ and
\begin{equation}
\beta_k \lesssim \langle 2^{j + k} \rangle^{-8} 2^k 2^{-\sigma k} a_k(\sigma)
\label{Beta Bound PdP}
\end{equation}
for any $\sigma \in [0, \sigma_1 - 1]$.

Beginning with the low-high frequency interaction, we have
\begin{equation}
\sum_{\ell \leq k} 2^\ell \alpha_\ell \beta_k
\lesssim
\langle 2^{j + k} \rangle^{-8} 2^{-\sigma k} 2^k a_k(\sigma) \sum_{\ell \leq k} \langle 2^{j + \ell} \rangle^{-8} 2^\ell a_\ell,
\label{PdPBP LH}
\end{equation}
and so it remains to show that
\begin{equation}
\sum_{\ell \leq k} \langle 2^{j + \ell} \rangle^{-8}  2^\ell a_\ell \lesssim 2^{-j} a_{-j}.
\label{PdPBP LH RTS}
\end{equation}
If $k + j \leq 0$, then (\ref{PdPBP LH RTS}) holds due to (\ref{Sum 1}).
If $k + j \geq 0$, then
we apply (\ref{Sum 1}) and (\ref{Sum 2}) according to whether
$\ell + j \leq 0$ or $\ell + j > 0$.

Turning now to the high-low frequency interaction, we have
\begin{equation}
\sum_{\ell \leq k} 2^\ell \alpha_k \beta_\ell
\lesssim 
 \langle 2^{j + k} \rangle^{-8} 2^{-\sigma k} 2^{-j} a_{-j} 2^k a_k(\sigma)
\sum_{\ell \leq k} \langle 2^{j + \ell} \rangle^{-8} 2^{\ell - k} 2^{\ell + j} 2^{\delta \lvert \ell + j \rvert}.
\label{PdPBP HL}
\end{equation}
We need only check
\begin{equation*}
\sum_{\ell \leq k} \langle 2^{j + \ell} \rangle^{-8} 2^{\ell - k} 2^{\ell + j} 2^{\delta \lvert \ell + j \rvert}
\lesssim 1,
\label{PdPBP HL RTS}
\end{equation*}
which can be seen to hold by breaking into cases $k + j \leq 0$ and $k + j \geq 0$.

We conclude with the high-high frequency interaction:
\begin{align}
\sum_{\ell \geq k} 2^\ell \alpha_\ell \beta_\ell
&\lesssim
\langle 2^{j + k} \rangle^{-8} 2^{- \sigma k} \sum_{\ell \geq k} \langle 2^{j + \ell} \rangle^{-8}
2^{2 \ell} a_\ell(\sigma) a_\ell \nonumber \\
&\lesssim
\langle 2^{j + k} \rangle^{-8} 2^{- \sigma k} 2^{-2j} a_j a_j(\sigma)
\sum_{\ell \geq k} \langle 2^{j + \ell} \rangle^{-8} 2^{2 \ell + 2 j} 2^{2 \delta \lvert \ell + j \rvert }.
\label{PdPBP HH}
\end{align}
Here
\begin{equation*}
\sum_{\ell \geq k} \langle 2^{j + \ell} \rangle^{-8} 2^{2 \ell + 2 j} 2^{2 \delta \lvert \ell + j \rvert } \lesssim 1,
\label{PdPBP HH RTS}
\end{equation*}
which
is seen to hold by considering separately the cases $k+j \geq 0$, $k + j < 0$.

Combining (\ref{PdPBP LH})--(\ref{PdPBP HH RTS}), we conclude 
(\ref{Psi dPsi Bilinear Parabolic}).
\end{proof}

\begin{lem}
Let
\begin{equation*}
f(r) \in \{ \overline{\psi_m}(r) \psi_\ell(r) : m, \ell = 1, 2\},
\quad \quad
g(r) \in \{ A_m(r) : m = 1,2\},
\end{equation*}
and 
$r \in [2^{2j - 2}, 2^{2j + 2}]$.  Then
\begin{displaymath}
\lVert P_k (fg)(r) \rVert_{F_k(T) \cap S_k^{1/2}(T)} \lesssim \left\{
\begin{array}{ll}
\varepsilon \cC 2^{-\sigma k} 2^{-2j} a_{-j} a_{-j}(\sigma)
& k + j \leq 0 \\ \\
\varepsilon \cC \langle 2^{j + k} \rangle^{-8} 2^{-\sigma k}  2^{-2j} b_{k,r}(\sigma)
& k + j \geq 0.
\end{array}
\right.
\end{displaymath}
\label{Psi A Cubic}
\end{lem}
\begin{proof}
We apply Lemma \ref{BIKT Lemma 5.1}.
By (\ref{Psi Psi Bilinear Parabolic})
and (\ref{A Bootstrap})
\begin{equation}
\alpha_k(r) \lesssim 2^{-\sigma k} \langle 2^{j + k} \rangle^{-8} 2^{-j} a_{-j} a_{\max(k, -j)}(\sigma)
\label{Alpha Bound}
\end{equation}
and
\begin{equation}
\beta_k(r) \lesssim \cC 2^{-\sigma k} \langle 2^{j + k} \rangle^{-8} b_{k,r}(\sigma)
\label{Beta Bound}
\end{equation}
hold for any $\sigma \in [0, \sigma_1 - 1]$.

We consider six cases, treating separately the low-high, high-low, 
and high-high frequency interactions, which we further divide according to 
whether $k + j \geq 0$ or $k + j \leq 0$.

\nline
\noindent
\textbf{Low-High frequency interaction with $k + j \geq 0$:}

\nline
\noindent
Using (\ref{Alpha Bound}) and (\ref{Beta Bound}), we have
\begin{equation}
\sum_{\ell \leq k} 2^{\ell} \alpha_\ell \beta_k 
\lesssim
\cC \langle 2^{j + k} \rangle^{-8} 2^{- \sigma k} 2^{-2j} b_{k, r}(\sigma)
\sum_{\ell \leq k} 2^{\ell} 2^{2j} \alpha_\ell,
\label{PAC LH1}
\end{equation}
and so it remains to verify that
\begin{equation}
\sum_{\ell \leq k} 2^{\ell} 2^{2j} \alpha_\ell \lesssim \varepsilon.
\label{PAC LH1 RTS}
\end{equation}
Taking $\sigma = 0$ in the bounds (\ref{Alpha Bound}) for $\alpha_\ell$ 
and using (\ref{Slowly Varying}), (\ref{Sum 2}) yields
\begin{align*}
\sum_{\ell \leq k} 2^{\ell} 2^{2j} \alpha_\ell 
&\lesssim
\sum_{\ell \leq k} \langle 2^{j + \ell} \rangle^{-8}
2^{\ell} 2^{2j} 2^{-j} a_{-j} a_{\max(\ell, -j)} \\
&= \sum_{\ell \leq -j} 2^{\ell + j} a_{-j}^2 +
\sum_{-j < \ell \leq k} \langle 2^{j + \ell} \rangle^{-8} 2^{\ell + j} a_{-j} a_{\ell} \\
&\lesssim a_{-j}^2 + a_{-j}^2 \sum_{-j < \ell \leq k} \langle 2^{j + \ell} \rangle^{-8}
2^{(1 + \delta)(\ell + j)} \lesssim \varepsilon,
\end{align*}
which proves (\ref{PAC LH1 RTS}).

\nline
\noindent
\textbf{High-Low frequency interaction with $k + j \geq 0$:}

\nline
\noindent
Taking $\sigma = 0$ in the bounds for $b_{\ell, r}$, we have
\begin{equation}
\sum_{\ell \leq k} 2^{\ell} \alpha_k \beta_\ell 
\lesssim
\cC \langle 2^{j + k} \rangle^{-8} 2^{-\sigma k} 2^{-2j} b_{k, r}(\sigma)
\sum_{\ell \leq k} \langle 2^{j + \ell} \rangle^{-8} 2^{\ell - k} b_{\ell, r},
\label{PAC HL1}
\end{equation}
and so it remains to show that
\begin{equation}
\sum_{\ell \leq k} \langle 2^{j + \ell} \rangle^{-8} 2^{\ell - k} b_{\ell, r}
\lesssim \varepsilon.
\label{PAC HL1 RTS}
\end{equation}
Splitting the sum as follows,
\begin{equation*}
\sum_{\ell \leq k} \langle 2^{j + \ell} \rangle^{-8} 2^{\ell - k} b_{\ell, r} =
\sum_{\ell \leq -j} \langle 2^{j + \ell} \rangle^{-8} 2^{\ell - k}
\sum_{q = \ell}^{-j} a_q^2
+
\sum_{-j < \ell \leq k} \langle 2^{j + \ell} \rangle^{-8} 2^{\ell - k} 2^{\ell + j}
a_{-j} a_\ell,
\end{equation*}
we note that
the first summand is controlled by
\begin{align*}
\sum_{\ell \leq -j} \langle 2^{j + \ell} \rangle^{-8} 2^{\ell - k}
\sum_{q = \ell}^{-j} a_q^2
&\lesssim
a_{-j}^2 \sum_{\ell \leq -j} 2^{\ell - k}
\sum_{q = \ell}^{-j} 2^{- 2\delta ( j + q )} \\
&\lesssim
a_{-j}^2 \lesssim \varepsilon.
\end{align*}
The second summand by
may be handled similarly,
thus proving (\ref{PAC HL1 RTS}).

\nline
\noindent
\textbf{High-High frequency interaction with $k + j \geq 0$:}

\nline
\noindent
Taking $\sigma = 0$ in the bound (\ref{Beta Bound}) for $\beta_\ell$, we have
\begin{align}
\sum_{\ell \geq k} 2^\ell \alpha_\ell \beta_\ell 
\lesssim& \;
\langle 2^{j + \ell} \rangle^{-8}
2^k \sum_{\ell \geq k} 2^{\ell - k}
2^{-\sigma \ell} 2^{-j} a_{-j} a_{\ell}(\sigma)
\cC 2^{\ell + j} a_{-j} a_\ell \nonumber \\
\lesssim &\;
\cC \langle 2^{j + k} \rangle^{-8} 2^{-\sigma k} 2^{-2j} b_{k,r}(\sigma) \times \nonumber \\
&\times \sum_{\ell \geq k} \langle 2^{j + \ell} \rangle^{-8} 2^{\ell - k}
2^{\delta(\ell - k)} 2^{\ell + j} a_{-j} a_\ell 
\label{PAC HH1}
\end{align}
and so it remains to show that
\begin{equation}
\sum_{\ell \geq k} \langle 2^{j + \ell} \rangle^{-8} 2^{\ell - k}
2^{\delta(\ell - k)} 2^{\ell + j} a_{-j} a_\ell 
\lesssim \varepsilon,
\label{PAC HH1 RTS}
\end{equation}
which follows, for instance, from
pulling out $a_{-j}^2$ via (\ref{Slowly Varying}) and summing.

In view of (\ref{PAC LH1})--(\ref{PAC HH1 RTS}), 
it follows from Lemma \ref{BIKT Lemma 5.1},
with $\omega = 0$ that
\begin{equation}
\lVert P_k (fg)(r) \rVert_{F_k(T) \cap S_k^{1/2}(T)} \lesssim
\varepsilon \cC \langle 2^{j + k} \rangle^{-8} 2^{-\sigma k} 2^{-2j} b_{k,r}(\sigma)
\quad \quad \text{for $k + j \geq 0$}
\label{PAC1}
\end{equation}
as required.

\nline
\noindent
\textbf{Low-High frequency interaction with $k + j \leq 0$:}

\nline
\noindent

In this case it follows from (\ref{Beta Bound}) that
\begin{equation*}
\beta_k \lesssim \cC 2^{-\sigma k} \sum_{p = k}^{-j} a_p a_p(\sigma)
\end{equation*}
so that
\begin{align}
\sum_{\ell \leq k} 2^{\ell} \alpha_\ell \beta_k 
&\lesssim
\cC 2^{-\sigma k} 2^{-j} a_{-j}
a_{-j} \sum_{p = k}^{-j} a_p a_p(\sigma)
\sum_{\ell \leq k} \langle 2^{j + \ell} \rangle^{-8} 2^{\ell} \nonumber \\
&\lesssim
\cC  2^{-\sigma k} 2^{-2j} a_{-j} a_{-j}(\sigma) \cdot
a_{-j} \sum_{p = k}^{-j} a_p 2^{-\delta( j + p)}
\sum_{\ell \leq k} 2^{\ell + j}.
\label{PAC LH2}
\end{align}
It remains to show
\begin{equation*}
a_{-j} \sum_{p = k}^{-j} a_p 2^{-\delta( j + p)}
\sum_{\ell \leq k} 2^{\ell + j} \lesssim \varepsilon,
\label{PAC LH2 RTS}
\end{equation*}
which follows from
pulling out $a_p$ as an $a_{-j}$ via (\ref{Slowly Varying}) and summing.

\nline
\noindent
\textbf{High-Low frequency interaction with $k + j \leq 0$:}

\nline
\noindent

In this case
\begin{equation}
\sum_{\ell \leq k} 2^{\ell} \alpha_k \beta_\ell 
\lesssim
\cC 2^{-2j} a_{-j} a_{-j}(\sigma)
\sum_{\ell \leq k} 2^{\ell + j} \sum_{p = \ell}^{-j} a_p^2,
\label{PAC HL2}
\end{equation}
and so we need to show
\begin{equation*}
\sum_{\ell \leq k} 2^{\ell + j} \sum_{p = \ell}^{-j} a_p^2 \lesssim \varepsilon,
\label{PAC HL2 RTS}
\end{equation*}
which follows by pulling out $a_{-j}^2$ and summing.

\nline
\noindent
\textbf{High-High frequency interaction with $k + j \leq 0$:}

\nline
\noindent

As a first step we write
\begin{equation}
2^k \sum_{\ell \geq k} 2^{(\ell - k)/2} \alpha_\ell \beta_\ell =
2^k \sum_{k \leq \ell < -j} 2^{(\ell - k)/2} \alpha_\ell \beta_\ell +
2^k \sum_{\ell \geq -j} 2^{(\ell - k)/2} \alpha_\ell \beta_\ell.
\label{PAC HH2 Split}
\end{equation}
The first summand is controlled by
\begin{align}
2^k \sum_{k \leq \ell < -j} 2^{(\ell - k)/2} \alpha_\ell \beta_\ell 
\lesssim& \; \cC 2^{-\sigma k} 2^{-2j} a_{-j} a_{-j}(\sigma) \times \nonumber \\
&\times \sum_{k \leq \ell < -j} 2^{(\ell - k)/2} 2^{k + j} 2^{-\sigma (\ell - k)}
\sum_{p = \ell}^{-j} a_p^2.
\label{PAC HH2 S1}
\end{align}
We have
\begin{align*}
\sum_{k \leq \ell < -j} 2^{(\ell - k)/2} 2^{k + j} 2^{-\sigma (\ell - k)}
\sum_{p = \ell}^{-j} a_p^2
&\lesssim
a_{-j}^2 2^{(k + j)/2} \sum_{k \leq \ell < -j} 2^{-2\delta(j + \ell)} \\
&\lesssim \varepsilon,
\end{align*}
which establishes the desired control on the first summand.

The second summand is controlled by
\begin{align}
2^k \sum_{\ell \geq -j} 2^{(\ell - k)/2} \alpha_\ell \beta_\ell 
\lesssim& \;
2^k \sum_{\ell \geq -j} 2^{(\ell - k)/2} \langle 2^{j + \ell} \rangle^{-8} 2^{-\sigma \ell}
2^{-j} a_{-j} a_\ell (\sigma) \times \nonumber \\
&\times \cC \langle 2^{j + \ell} \rangle^{-8} 2^{\ell + j} a_{-j} a_\ell \nonumber \\
\lesssim&\; 
\cC 2^{-\sigma k} 2^{-2j} a_{-j} a_{-j}(\sigma) \times \nonumber \\
&\times
\sum_{\ell \geq -j} 2^{(\ell - k)/2} 2^{k + j} 2^{(1 + \delta)(\ell + j)} a_{-j} a_\ell,
\label{PAC HH2 S2}
\end{align}
and so it remains to show that
\begin{equation}
\sum_{\ell \geq -j} 2^{(\ell - k)/2} 2^{k + j} 2^{(1 + \delta)(\ell + j)} a_{-j} a_\ell
\lesssim \varepsilon,
\label{PAC HH2 S2 RTS}
\end{equation}
which follows from pulling out $a_{-j}^2$ and summing.

Combining (\ref{PAC LH2})--(\ref{PAC HH2 S2 RTS}),
 we conclude from applying
Lemma \ref{BIKT Lemma 5.1} with $\omega = 1/2$ that
\begin{equation*}
\lVert P_k(fg)(r) \rVert_{F_k(T) \cap S_k^{1/2}(T)} \lesssim
\varepsilon \cC 2^{-\sigma k} 2^{-2j} a_{-j} a_{-j}(\sigma)
\quad \quad \text{for $k + j \leq 0$},
\end{equation*}
which, combined with (\ref{PAC1}) completes the proof of the lemma.
\end{proof}

\begin{lem}
For any $k \in \mathbf{Z}$ and $s \in [0, \infty)$ we have
\begin{equation*}
\lVert P_k A_m(s) \rVert_{F_k(T) \cap S_k^{1/2}(T)} \lesssim
(1 + s2^{2k})^{-4} 2^{-\sigma k} b_{k,s}(\sigma).
\end{equation*}
\label{A Bound}
\end{lem}
\begin{proof}
From representation (\ref{A Integral}) for $A_m$ it follows that
\begin{align}
\lVert P_k A_m(s) \rVert_{F_k(T) \cap S_k^{1/2}(T)} \lesssim&
\int_s^\infty \lVert P_k( \overline{\psi_m}(r) \partial_\ell \psi_\ell(r))
\rVert_{F_k(T) \cap S_k^{1/2}(T)} dr \nonumber \\
&+ \int_s^\infty \lVert P_k( \overline{\psi_m}(r) \psi_\ell A_\ell(r))
\rVert_{F_k(T) \cap S_k^{1/2}(T)} dr.
\label{A Bound Split}
\end{align}
Taking $k_0 \in \mathbf{Z}$
so that $s \in [2^{2k_0 - 1}, 2^{2k_0 + 1})$ 
and using (\ref{Psi dPsi Bilinear Parabolic}),
we have that
the first term is dominated by
\begin{align}
&\sum_{j \geq k_0} \int_{2^{2j - 1}}^{2^{2j + 1}}
\lVert P_k (\overline{\psi_m}(r) \partial_\ell \psi_\ell(r)) \rVert_{F_k(T) \cap S_k^{1/2}(T)} dr \nonumber \\
&\lesssim 2^{-\sigma k} \sum_{j \geq k_0}
\langle 2^{j + k} \rangle^{-8} (2^{j + k} a_{-j} a_k(\sigma) + a_{-j} a_{-j}(\sigma)).
\label{A Bound S1}
\end{align}
We claim that
\begin{equation}
\sum_{j \geq k_0}
\langle 2^{j + k} \rangle^{-8} (2^{j + k} a_{-j} a_k(\sigma) + a_{-j} a_{-j}(\sigma))
\lesssim
(1 + s 2^{2k})^{-4} b_{k,s}(\sigma).
\label{bks summation}
\end{equation}
When $k + k_0 \geq 0$, 
it follows from (\ref{Slowly Varying}) that
the left hand side of (\ref{bks summation}) is bounded by
\begin{align}
2^{k_0 + k} a_{-k_0} a_k(\sigma)
&\sum_{j \geq k_0} \langle 2^{j + k} \rangle^{-8}
\left( 2^{j - k_0} 2^{\delta(j - k_0)} + 2^{-k_0 - k} 2^{\delta(j - k_0)} 2^{\delta (k + j)} \right) 
\nonumber \\
&\lesssim
b_{k, s}(\sigma) \sum_{j \geq k_0} \langle 2^{j + k} \rangle^{-8}
( 2^{(1 + \delta)(j - k_0)} + 
2^{(\delta - 1)(k_0 + k)}
2^{2\delta(j - k_0)}),
\label{A Bound S1 RTS1}
\end{align}
and so it suffices to show
\begin{equation}
\sum_{j \geq k_0} \langle 2^{j + k} \rangle^{-8} 2^{2(j - k_0)}
\lesssim
\langle 2^{j + k_0} \rangle^{-8},
\label{A Bound S1 RTS2}
\end{equation}
which follows from series comparison, for instance.

Together (\ref{A Bound S1 RTS2}) and (\ref{A Bound S1 RTS1}),
show that (\ref{bks summation}) holds for $k + k_0 \geq 0$.

If, on the other hand,
$k + k_0 \leq 0$, then we split the sum in (\ref{bks summation}) according to whether
$j + k \leq 0$ or $j + k > 0$.
In the first case,
\begin{align}
\sum_{k_0 \leq j \leq -k} \langle 2^{j + k} 
&
\rangle^{-8}
(2^{j + k}a_{-j} a_k(\sigma) +
a_{-j} a_{-j}(\sigma))  \nonumber
\\
& \lesssim
\langle 2^{k_0 + k} \rangle^{-8}
b_{k, s}(\sigma)
+ \sum_{k_0 \leq j \leq -k} \langle 2^{j + k} \rangle^{-8} 2^{j + k}a_{-j} a_k(\sigma).
\label{A Bound S1 RTS3}
\end{align}
Then
\begin{align}
\sum_{k_0 \leq j \leq -k} \langle 2^{j + k} \rangle^{-8} 2^{j + k}a_{-j} a_k(\sigma)
&\lesssim
\sum_{k_0 \leq j \leq -k} \langle 2^{j + k} \rangle^{-8} 2^{j + k}a_{-j} a_{-j}(\sigma) 2^{-\delta(j + k)} \nonumber \\
&\sim (1 + s2^{2k})^{-4}b_{k,s}(\sigma).
\label{A Bound S1 RTS4}
\end{align}

When $j + k > 0$ we have
\begin{align}
\sum_{j > -k} \langle 2^{j + k} \rangle^{-8}
&
(2^{j + k}a_{-j} a_k(\sigma) + a_{-j} a_{-j}(\sigma)) \nonumber \\
&\lesssim
a_k a_k(\sigma) \sum_{j > -k} \langle 2^{j + k} \rangle^{-8}
(2^{j + k} 2^{\delta(j + k)} + 2^{2 \delta(j + k)}) \nonumber \\
&\lesssim b_{k, s}(\sigma).
\label{A Bound S1 RTS5}
\end{align}

Therefore (\ref{A Bound S1 RTS3}) and (\ref{A Bound S1 RTS4}) imply
(\ref{bks summation}) holds when $k + k_0 \leq 0$ and $j + k \leq 0$
and (\ref{A Bound S1 RTS5}) implies it holds when both $k + k_0 \leq 0$
and $j + k > 0$.

Having shown (\ref{bks summation}), we combine it with (\ref{A Bound S1}), concluding
\begin{equation}
\int_s^\infty \lVert P_k (\overline{\psi_m}(r) \partial_\ell \psi_\ell(r))
\rVert_{F_k(T) \cap S_k^{1/2}(T)} dr \lesssim
(1 + s 2^{2k})^{-4} 2^{-\sigma k} b_{k,s}(\sigma).
\label{A Bound 1}
\end{equation}

We now move on to control the second term in (\ref{A Bound Split}).
By Lemma \ref{Psi A Cubic} and (\ref{bks summation}),
this term is bounded by
\begin{align}
&\sum_{j \geq k_0} \int_{2^{2j-1}}^{2^{2j+1}}
\lVert P_k( \overline{\psi_x}(r) \psi_x(r) A_x(r)) \rVert_{F_k(T) \cap S_k^{1/2}(T)} dr \nonumber \\
&\lesssim \cC 2^{-\sigma k} \varepsilon \sum_{j \geq k_0}
\langle 2^{j + k} \rangle^{-8} (\mathbf{1}_-(k + j) a_{-j} a_{-j}(\sigma) +
\mathbf{1}_+(k + j) b_{k, 2^{2j}}(\sigma)) \nonumber \\
&\lesssim 
\cC 2^{-\sigma k} \varepsilon \langle 2^{k_0 + k} \rangle^{-8} b_{k, 2^{2k_0}}(\sigma).
\label{A Bound 2}
\end{align}

Together (\ref{A Bound Split}), (\ref{A Bound 1}), and (\ref{A Bound 2}) imply
\begin{equation*}
\lVert P_k A_m(s) \rVert_{F_k(T) \cap S_k^{1/2}(T)} \lesssim
2^{-\sigma k} (1 + s 2^{2k})^{-4} b_{k,s}(\sigma)(1 + \cC \varepsilon),
\end{equation*}
from which it follows that $\cC \lesssim 1 + \cC \varepsilon$ and hence $\cC \lesssim 1$,
proving the lemma.
\end{proof}

\begin{lem}
It holds that
\begin{displaymath}
\lVert P_k A_\ell^2(r) \rVert_{F_k(T) \cap S_k^{1/2}(T)}
\lesssim \left\{
\begin{array}{ll}
\varepsilon 2^{-\sigma k} 2^{-j} a_{-j} a_{-j}(\sigma)
& \textrm{if $k + j \leq 0$} \\ \\
 \varepsilon 2^{-\sigma k} 2^{-j} b_{k, 2^{2j}}(\sigma)
& \textrm{if $k + j \geq 0$}.
\end{array} \right.
\end{displaymath}
\label{A Quadratic}
\end{lem}
\begin{proof}
We apply Lemma \ref{BIKT Lemma 5.1} with $f = g = A_\ell$ 
and $\omega = 0$ so that
\begin{equation*}
\lVert P_k (A_\ell^2(r)) \rVert_{F_k(T) \cap S_k^{1/2}(T)}
\lesssim
\sum_{\ell \leq k} 2^\ell \alpha_k \beta_\ell +
\sum_{\ell \geq k} 2^\ell \alpha_\ell \beta_\ell,
\end{equation*}
where
\begin{equation*}
\alpha_k \lesssim 2^{-\sigma k} \langle 2^{j + k} \rangle^{-8} b_{k,s}(\sigma),
\quad\quad
\beta_k \lesssim \langle 2^{j + k} \rangle^{-8} b_{k,s}.
\end{equation*}

\nline
\noindent
\textbf{Case $k + j \leq 0$:}

\nline
\noindent

We first consider the case $k + j \leq 0$ and
proceed to control the high-low frequency interaction.  We have
\begin{align}
\sum_{\ell \leq k} 2^\ell \alpha_k \beta_\ell &\lesssim
2^{-\sigma k} \sum_{\ell \leq k} 2^\ell b_{k, 2^{2j}}(\sigma) b_{\ell, 2^{2j}} \nonumber \\
&\lesssim 2^{-\sigma k}
\sum_{p = k}^{-j} a_p a_p(\sigma) 2^\ell
\sum_{\ell \leq k} \sum_{q = \ell}^{-j} a_q^2 \nonumber \\
&\lesssim
2^{-\sigma k}
a_{-j} a_{-j}(\sigma) \sum_{p = k}^{-j} 2^{-2\delta(j + p)}
\sum_{\ell \leq k} 2^\ell a_{-j}^2 \sum_{q = \ell}^{-j} 2^{-2 \delta(j + q)}.
\label{AQ HL1}
\end{align}
It remains to show
\begin{equation}
\sum_{p = k}^{-j} 2^{-2\delta(j + p)}
\sum_{\ell \leq k} 2^\ell a_{-j}^2 \sum_{q = \ell}^{-j} 2^{-2 \delta(j + q)}
\lesssim \varepsilon,
\label{AQ HL1 RTS}
\end{equation}
which follows from bounding $a_{-j}^2$ by $\varepsilon$ and summing.
To control the high-high interaction term
we first split the sum as
\begin{equation}
\sum_{\ell \geq k} 2^\ell \alpha_\ell \beta_\ell \lesssim
\sum_{k \leq \ell < -j} 2^\ell \alpha_\ell \beta_\ell +
\sum_{\ell \geq -j} 2^\ell \alpha_\ell \beta_\ell.
\label{AQ HH1 Split}
\end{equation}
The first summand is controlled by
\begin{align*}
\sum_{k \leq \ell < -j}2^{\ell} \alpha_\ell \beta_\ell
&\lesssim
2^{-\sigma k} \sum_{k \leq \ell < -j}2^{\ell} b_{\ell, 2^{2j}}(\sigma) b_{\ell, 2^{2j}} \nonumber \\
&\lesssim
2^{-\sigma k} 2^{-j} \sum_{k \leq \ell < -j} 2^{j + \ell}
\sum_{p = \ell}^{-j} a_p a_p(\sigma)
\sum_{q = \ell}^{-j} a_q^2.
\end{align*}
Pulling out $a_{-j}^3 a_{-j}(\sigma)$ and summing
implies
\begin{equation}
\sum_{k \leq \ell < -j}2^{\ell} \alpha_\ell \beta_\ell
\lesssim
\varepsilon 2^{-\sigma k} 2^{-j} a_{-j} a_{-j}(\sigma).
\label{AQ HH1 S1}
\end{equation}
The second summand is controlled by
\begin{align}
\sum_{\ell \geq -j} 2^{\ell} \alpha_\ell \beta_\ell
&\lesssim
2^{-\sigma k} \sum_{\ell \geq -j} 2^\ell \langle 2^{j + \ell} \rangle^{-8}
b_{\ell, 2^{2j}}(\sigma) b_{\ell, 2^{2j}} \nonumber \\
&\lesssim 2^{-\sigma k} \sum_{\ell \geq -j} 2^\ell\langle 2^{j + \ell} \rangle^{-8}
2^{2(\ell + j)} a_{-j}^2 a_\ell a_\ell(\sigma) \nonumber \\
&\lesssim \varepsilon 2^{-\sigma k} 2^{-j} a_{-j} a_{-j}(\sigma).
\label{AQ HH1 S2}
\end{align}

Combining (\ref{AQ HL1})--(\ref{AQ HH1 S2}), we conclude
\begin{equation}
\lVert P_k A_\ell^2 (r) \rVert_{F_k(T) \cap S_k^{1/2}(T)}
\lesssim
\varepsilon 2^{- \sigma k} 2^{-j} a_{-j} a_{-j}(\sigma)
\quad \quad \text{for $k + j \leq 0$}.
\label{AQ Bound 1}
\end{equation}

\nline
\noindent
\textbf{Case $k + j \geq 0$:}

\nline
\noindent

We now consider the case $k + j \geq 0$ and turn to the high-low frequency interaction, splitting
it into two pieces:
\begin{equation}
\sum_{\ell \leq k} 2^\ell \alpha_k \beta_\ell \leq
\sum_{\ell \leq -j} 2^\ell \alpha_k \beta_\ell + 
\sum_{-j < \ell \leq k} 2^{\ell} \alpha_k \beta_\ell.
\label{AQ HL2 Split}
\end{equation}
The first summand is controlled by
\begin{equation}
\sum_{\ell \leq -j} 2^\ell \alpha_k \beta_\ell \lesssim
2^{-\sigma k} 2^{-j} b_{k, 2^{2j}}(\sigma)
\sum_{\ell \leq -j} 2^{\ell + j} \langle 2^{j + k} \rangle^{-8} b_{\ell, 2^{2j}},
\label{AQ HL2 S1}
\end{equation}
and so we need to show
\begin{equation}
\sum_{\ell \leq -j} 2^{\ell + j} \langle 2^{j + k} \rangle^{-8}  b_{\ell, 2^{2j}} \lesssim \varepsilon,
\label{AQ HL2 S1 RTS}
\end{equation}
which follows from
\begin{align*}
\sum_{\ell \leq -j} 2^{\ell + j} b_{\ell, 2^{2j}} 
&\lesssim
\sum_{\ell \leq -j} 2^{\ell + j} \sum_{p = \ell}^{-j} a_p^2 \\
&\lesssim
a_{-j}^2
\sum_{\ell \leq -j} 2^{(1 - 2 \delta)(\ell + j)} \lesssim \varepsilon.
\end{align*}
The second summand in (\ref{AQ HL2 Split}) is controlled by
\begin{equation}
\sum_{-j < \ell \leq k} 2^{\ell} \alpha_k \beta_\ell \lesssim
2^{-\sigma k} 2^{-j} b_{k, 2^{2j}}(\sigma) \sum_{j < \ell \leq k}
\langle 2^{j + \ell} \rangle^{-8}
2^{\ell + j} \langle 2^{j + k} \rangle^{-8} 2^{\ell + j} a_{-j} a_\ell,
\label{AQ HL2 S2}
\end{equation}
where we note
\begin{align}
\sum_{-j < \ell \leq k} \langle 2^{j + \ell} \rangle^{-8} 2^{2\ell + 2j} a_{-j} a_\ell &\lesssim
a_{-j}^2
\sum_{-j < \ell \leq k} \langle 2^{j + \ell} \rangle^{-8} 2^{(2 + \delta)(\ell + j)} \nonumber \\
&\lesssim \varepsilon.
\label{AQ HL2 S2 RTS}
\end{align}
We now turn to the high-high frequency interaction.  We have
\begin{align}
\sum_{\ell \geq k} 2^{\ell} \alpha_\ell \beta_\ell &\lesssim
\sum_{\ell \geq k} 2^\ell 2^{-\sigma \ell} \langle 2^{j + \ell} \rangle^{-8}
2^{2(\ell + j)} a_{-j}^2 a_\ell a_\ell(\sigma) \nonumber \\
&\lesssim 
2^{-\sigma k} 2^{-j} 2^{k + j} a_{-j} \sum_{\ell \geq k}
2^{\ell - k} 2^{-\sigma (\ell - k)}
\langle 2^{j + \ell} \rangle^{-8} 2^{2 (\ell + j)} a_{-j} a_\ell a_\ell(\sigma) \nonumber \\
&\lesssim
2^{-\sigma k} 2^{-j} b_{k, 2^{2j}}(\sigma)
\sum_{\ell \geq k} 
\langle 2^{j + \ell} \rangle^{-8} 2^{(1 + \delta)(\ell - k)} 2^{2(\ell + j)}
a_{-j} a_\ell.
\label{AQ HH2}
\end{align}
It remains to show that
\begin{equation}
\sum_{\ell \geq k} 
\langle 2^{j + \ell} \rangle^{-8} 2^{(1 + \delta)(\ell - k)} 2^{2(\ell + j)}
a_{-j} a_\ell
\lesssim \varepsilon,
\label{AQ HH2 RTS}
\end{equation}
which follows from
bounding $a_{-j} a_\ell$ by $\varepsilon$ and summing.

Together (\ref{AQ HL2 Split})--(\ref{AQ HH2 RTS}) imply
\begin{equation*}
\lVert P_k A_\ell^2(r) \rVert_{F_k(T) \cap S_k^{1/2}(T)} \lesssim
\varepsilon 2^{-\sigma k} 2^{-j} b_{k, 2^{2j}}(\sigma)
\quad \quad \text {for $k + j \geq 0$},
\end{equation*}
which, combined with (\ref{AQ Bound 1}) implies the lemma.
\end{proof}

Set
\begin{equation}
c_{k, j}(\sigma) = \left\{
\begin{array}{ll}
2^{-j} a_{-j} a_{-j}(\sigma) & \textrm{if $k + j \leq 0$} \\ \\
2^{2k + j} a_{-j} a_k(\sigma) & \textrm{if $k + j \geq 0$}.
\end{array} \right.
\label{ckj Definition}
\end{equation}

\begin{lem}
Let $r \in [2^{2j - 2}, 2^{2j + 2} ]$ and let
\begin{equation*}
F \in \{ A_\ell^2, \partial_\ell A_\ell, fg : \ell = 1,2;
f, g \in \{\psi_m, \overline{\psi_m} : m = 1,2\} \}.
\end{equation*}
Then
\begin{equation}
\lVert P_k F(r) \rVert_{F_k(T) \cap S_k^{1/2}(T)} \lesssim
\langle 2^{j + k} \rangle^{-8} 2^{- \sigma k} c_{k,j}(\sigma).
\label{Quadratic Parabolic}
\end{equation}
\label{Quadratic-type Terms}
\end{lem}
\begin{proof}
If $F = A_\ell^2$, then (\ref{Quadratic Parabolic}) is an immediate consequence
of Lemma \ref{A Quadratic} when $k + j \leq 0$.  If $k + j \geq 0$,
then Lemma \ref{A Quadratic} implies
\begin{equation*}
\lVert P_k A_\ell^2 (r) \rVert_{F_k(T) \cap S_k^{1/2}(T)}
\lesssim \varepsilon 2^{-\sigma k} 2^{-j} 2^{k + j} a_{-j} a_{-j}(\sigma),
\end{equation*}
and multiplying the right hand side by $2^{k + j}$ yields the desired estimate.

Consider now the case where $F = \partial_\ell A_\ell$.
By Lemma \ref{A Bound}, we have
\begin{equation}
\lVert P_k (\partial_\ell A_\ell)(r) \rVert_{F_k(T) \cap S_k^{1/2}(T)}
\lesssim 2^k \langle 2^{j + k} \rangle^{-8} 2^{- \sigma k} b_{k, 2^{2j}}(\sigma).
\label{QtT dA}
\end{equation}
When $k + j \geq 0$, we rewrite (\ref{QtT dA}) as
\begin{equation*}
\lVert P_k (\partial_\ell A_\ell)(r) \rVert_{F_k(T) \cap S_k^{1/2}(T)}
\lesssim \langle 2^{j + k} \rangle^{-8} 2^{-\sigma k} 2^k 2^{k + j} a_{-j} a_k(\sigma),
\end{equation*}
which is the desired bound (\ref{Quadratic Parabolic}).  
If $k + j \leq 0$, then (\ref{QtT dA}) becomes
\begin{align*}
\lVert P_k (\partial_\ell A_\ell)(r) \rVert_{F_k(T) \cap S_k^{1/2}(T)}
&\lesssim \langle 2^{j + k} \rangle^{-8} 2^{-\sigma k} 2^k
\sum_{p = k}^{-j} a_p a_p(\sigma) \\
&\lesssim  \langle 2^{j + k} \rangle^{-8} 2^{-\sigma k} 2^{-j} a_{-j} a_{-j}(\sigma)
= \langle 2^{j + k} \rangle^{-8} 2^{-\sigma k} c_{k, j}(\sigma).
\end{align*}

If $F = fg$, $fg$ as in the statement of the lemma, then
(\ref{Quadratic Parabolic}) follows directly from
(\ref{Psi Psi Bilinear Parabolic}) when $k + j \leq 0$.  If $k + j \geq 0$, 
then to get (\ref{Quadratic Parabolic}) we
multiply the right hand side of (\ref{Psi Psi Bilinear Parabolic})
by $2^{2j + 2k}$.
\end{proof}

Set
\begin{equation}
d_{k,j} := \varepsilon \langle 2^{j + k} \rangle^{-8} 2^{-\sigma k} 2^{2k}
( a_k(\sigma) + 2^{-3(k+j)/2} a_{-j}(\sigma)).
\label{dkj Definition}
\end{equation}

\begin{lem}
It holds that
\begin{equation*}
\lVert P_k U_m(r) \rVert_{F_k(T) \cap S_k^{1/2}(T)} \lesssim
\varepsilon \langle 2^{j + k} \rangle^{-8} 2^{-\sigma k} 2^{2k} 
( a_k(\sigma) + 2^{-3(k+j)/2} a_{-j}(\sigma))
=: d_{k,j}.
\end{equation*}
\label{U Bound}
\end{lem}
\begin{proof}
Using now (\ref{Heat Nonlinearity 0}) instead of (\ref{Heat Nonlinearity}), i.e.,
taking now
\begin{equation*}
U_\alpha = i A_\ell \partial_\ell \psi_\alpha + i \partial_\ell (A_\ell \psi_\alpha)
- A_x^2 \psi_\alpha + i \Im(\psi_\alpha \psi_\ell)\psi_\ell,
\end{equation*}
we have that it suffices to prove
\begin{equation*}
\lVert P_k (F(r) f(r)) \rVert_{F_k(T) \cap S_k^{1/2}(T)} 
+ 2^k \lVert P_k (A_\ell(r) f(r) ) \rVert_{F_k(T) \cap S_k^{1/2}(T)}
\lesssim d_{k, j},
\end{equation*}
where
\begin{equation*}
F \in \{ A_\ell^2, \partial_\ell A_\ell, gh : \ell = 1,2;
f, h \in \{\psi_m, \overline{\psi_m} : m = 1,2\} \}
\end{equation*}
and $f \in \{ \psi_m, \overline{\psi_m} : m = 1, 2\}$.
We consider the terms $P_k(Ff)$, and $P_k(Af)$ separately.

\nline
\noindent
\textbf{Controlling $P_k(Ff)$:}
\nline

We apply Lemma \ref{BIKT Lemma 5.1} to $P_k (F f)$, handling the different
frequency interactions separately and according to cases.
We record 
\begin{equation*}
\alpha_k \lesssim \langle 2^{j + k} \rangle^{-8} 2^{-\sigma k} c_{k, j}(\sigma),
\end{equation*}
a consequence of  (\ref{Quadratic Parabolic}).

Let us begin by assuming $k + j \leq 0$.  
For the low-high frequency interaction, we have
\begin{align}
\sum_{\ell \leq k} 2^\ell \alpha_\ell \beta_k
&\lesssim
2^{-\sigma k} a_k(\sigma) \sum_{\ell \leq k} 2^\ell c_{\ell, j} \nonumber \\
&\lesssim
2^{-\sigma k} a_k(\sigma) \sum_{\ell \leq k} 2^{\ell - j} a_{-j}^2 \nonumber \\
&\lesssim \varepsilon 2^{-\sigma k} 2^{k - j} 2^{- \delta(k + j)} a_{-j}(\sigma).
\label{U1 LH1}
\end{align}
In a similar manner we control
the high-low frequency interaction by
\begin{align}
\sum_{\ell \leq k} 2^\ell \alpha_k \beta_\ell
&\lesssim 2^{-\sigma k} c_{k, j}(\sigma)
\sum_{\ell \leq k} 2^\ell a_\ell \nonumber \\
&\lesssim 2^{-\sigma k} 2^{-j} a_{-j} a_{-j}(\sigma)
\sum_{\ell \leq k} 2^\ell a_\ell \nonumber \\
&\lesssim
\varepsilon 2^{-\sigma k} 2^{k - j} a_{-j}(\sigma). \label{U1 HL1}
\end{align}

The high-high frequency interaction we split into two cases:
\begin{equation}
2^k \sum_{\ell \geq k} 2^{(\ell - k)/2} \alpha_\ell \beta_\ell
\lesssim
2^k \sum_{k \leq \ell < -j} 2^{(\ell - k)/2} \alpha_\ell \beta_\ell +
2^k \sum_{\ell \geq -j} 2^{(\ell - k)/2} \alpha_\ell \beta_\ell.
\label{U1 HH1 Split}
\end{equation}
We control the first summand using the definition 
(\ref{ckj Definition})
of $c_{k ,j}(\sigma)$,
the frequency envelope properties (\ref{Slowly Varying}), (\ref{Sum 1}), and energy dispersion:
\begin{align}
2^k \sum_{k \leq \ell < -j} 2^{(\ell - k)/2} \alpha_\ell \beta_\ell
&\lesssim
2^k \sum_{k \leq \ell < -j} 2^{(\ell - k)/2} 2^{-\sigma \ell}
c_{\ell, j}(\sigma) a_\ell \nonumber \\
&\lesssim
2^{-\sigma k} 2^{k - j} a_{-j}(\sigma) a_{-j}
\sum_{k \leq \ell < -j} 2^{(\ell - k)/2} a_\ell \nonumber \\
&\lesssim
2^{-\sigma k} 2^{k - j} 2^{-(k + j)/2} a_{-j}(\sigma)
a_{-j} \sum_{k \leq \ell < -j} 2^{(\ell + j)/2} a_\ell \nonumber \\
&\lesssim
\varepsilon 2^{-\sigma k} 2^{k - j} 2^{-(k + j)/2} a_{-j}(\sigma).
\label{U1 HH1 S1}
\end{align}

In like manner we control the second summand:
\begin{align}
2^k \sum_{\ell \geq -j} 2^{(\ell - k)/2} \alpha_\ell \beta_\ell
&\lesssim
2^k \sum_{\ell \geq -j} \langle 2^{j + \ell} \rangle^{-8}
2^{(\ell - k)/2} 2^{-\sigma \ell} c_{\ell, j}(\sigma) a_\ell \nonumber \\
&\lesssim
2^k \sum_{\ell \geq -j} \langle 2^{j + \ell} \rangle^{-8}
2^{(\ell - k)/2} 2^{-\sigma \ell} 2^{2 \ell + j} a_{-j} a_\ell(\sigma) a_\ell \nonumber \\
&\lesssim
\varepsilon 2^{-\sigma k} 2^{k - j} 2^{-(k + j)/2} a_{-j}(\sigma)
\label{U1 HH1 S2}
\end{align}

Combining (\ref{U1 LH1})--(\ref{U1 HH1 S2}), we conclude
\begin{equation}
\lVert P_k (F(r) f(r)) \rVert_{F_k(T) \cap S_k^{1/2}(T)} \lesssim \varepsilon
2^{-\sigma k} 2^{k - j} 2^{-(k + j)/2} a_{-j}(\sigma),
\quad \quad k + j \leq 0.
\label{F Psi High}
\end{equation}

We now turn to the case $k + j \geq 0$. 
In the low-high frequency interaction case, we have
\begin{align}
\sum_{\ell \leq k} 2^\ell \alpha_\ell \beta_k
\lesssim &\;
\langle 2^{j + k} \rangle^{-8} 2^{-\sigma k} a_k(\sigma)
\sum_{\ell \leq k} \langle 2^{j + \ell} \rangle^{-8} 2^\ell c_{\ell, j} \nonumber \\
\lesssim &\;
\langle 2^{j + k} \rangle^{-8} 2^{-\sigma k} 2^{2k} a_k(\sigma) \times \nonumber \\
&
\left( \sum_{\ell \leq -j} 2^{\ell - 2k} 2^{-j} a_{-j}^2 +
\sum_{-j < \ell \leq k} \langle 2^{j + \ell} \rangle^{-8} 2^{\ell - 2k} 2^{\ell + j} a_{-j} a_\ell \right).
\label{U1 LH2 Split}
\end{align}
To estimate the first term we use
\begin{equation}
a_{-j}^2 \sum_{\ell \leq -j} 2^{\ell - k} 2^{-j - k}
\lesssim \varepsilon 2^{-(j + k)} \cdot 2^{-(j + k)}
\leq \varepsilon
\label{U1 LH2 S1}
\end{equation}
and for the second
\begin{align}
a_{-j} &\sum_{-j < \ell \leq k} \langle 2^{j + \ell} \rangle^{-8} 2^{3 \ell + j - 2k} a_\ell \nonumber \\
&= a_{-j} \sum_{-j < \ell \leq k} \langle 2^{j + \ell} \rangle^{-8} 2^{\ell + j} 2^{2 \ell - 2k} a_\ell \nonumber \\
&\lesssim a_{-j} a_k \sum_{-j < \ell \leq k} \langle 2^{j + \ell} \rangle^{-8} 2^{\ell + j}
2^{(2 - \delta)(\ell - k)} \lesssim \varepsilon.
\label{U1 LH2 S2}
\end{align}

In the high-low frequency interaction case, we have
\begin{align}
\sum_{\ell \leq k} 2^\ell \alpha_k \beta_\ell
&\lesssim
\langle 2^{j + k} \rangle^{-8}
2^{-\sigma k} c_{k, j}(\sigma) \sum_{\ell \leq k}
\langle 2^{j + \ell} \rangle^{-8} 2^\ell a_\ell \nonumber \\
&\lesssim
\langle 2^{j + k} \rangle^{-8} 2^{-\sigma k} 2^{2k + j} a_{-j} a_k(\sigma)
\sum_{\ell \leq k} \langle 2^{j + \ell} \rangle^{-8} 2^\ell a_\ell \nonumber \\
&\lesssim
\langle 2^{j + k} \rangle^{-8} 2^{-\sigma k} 2^{2k} a_k(\sigma)
a_{-j}^2.
\label{U1 HL2}
\end{align}

In the high-high frequency interaction case we have
\begin{align}
\sum_{\ell \geq k} 2^\ell \alpha_\ell \beta_\ell
&\lesssim
\sum_{\ell \geq k}  \langle 2^{j + \ell} \rangle^{-8} 2^\ell 2^{-\sigma \ell} a_\ell(\sigma) c_{\ell, j} \nonumber \\
&\lesssim
\langle 2^{j + k} \rangle^{-8} 2^{-\sigma k} \sum_{\ell \geq k} \langle 2^{j + \ell} \rangle^{-8}
2^\ell a_\ell(\sigma) 2^{2\ell + j} a_{-j} a_\ell \nonumber \\
&\lesssim
\langle 2^{j + k} \rangle^{-8} 2^{-\sigma k} 2^{2k} a_k(\sigma) a_{-j}^2.
\label{U1 HH2}
\end{align}

From (\ref{U1 LH2 Split})--(\ref{U1 HH2}) we conclude
\begin{equation}
\lVert P_k (F(r) f(r)) \rVert_{F_k(T) \cap S_k^{1/2}(T)} \lesssim \varepsilon
\langle 2^{j + k} \rangle^{-8} 2^{-\sigma k} 2^{2k} a_k(\sigma),
\quad \quad k + j \geq 0.
\label{F Psi Low}
\end{equation}

\nline
\noindent
\textbf{Controlling $2^k P_k(Af)$:}
\nline

We now apply Lemma \ref{BIKT Lemma 5.1} to $P_k(A_\ell f)$.
Note that
\begin{equation*}
\alpha_k \lesssim \langle 2^{j + k} \rangle^{-8} 2^{-\sigma k} b_{k, r}(\sigma)
\end{equation*}
because of Lemma \ref{A Bound}
and that
\begin{equation*}
\beta_k \lesssim \langle 2^{j + k} \rangle^{-8} 2^{-\sigma k} a_k(\sigma).
\end{equation*}

We begin by assuming $k + j \leq 0$.
The low-high frequency interaction is controlled by
\begin{align}
\sum_{\ell \leq k} 2^\ell \alpha_\ell \beta_k
&\lesssim
2^{-\sigma k} a_k(\sigma) \sum_{\ell \leq k} 2^\ell \sum_{p = \ell}^{-j} a_\ell^2 \nonumber \\
&\lesssim
2^{-\sigma k} 2^{-\delta(k + j)} a_{-j}^2 a_{-j}(\sigma)
\sum_{\ell \leq k} 2^\ell \sum_{p = \ell}^{-j} 2^{-2(j + p)}.
\nonumber
\end{align}
Summing yields
\begin{equation}
2^k \sum_{\ell \leq k} 2^\ell \alpha_\ell \beta_k
\lesssim
2^{2k} 2^{-\sigma k} 2^{-(k + j)/2} a_{-j}^2 a_{-j}(\sigma).
\label{U2 LH1}
\end{equation}

Control over the high-low frequency interaction follows from
\begin{align}
\sum_{\ell \leq k} 2^\ell \alpha_k \beta_\ell
&\lesssim
2^{-\sigma k} \sum_{p = k}^{-j} a_p a_p(\sigma)
\sum_{\ell \leq k} 2^\ell a_\ell \nonumber \\
&\lesssim
2^k 2^{-\sigma k} 2^{-2 \delta(k + j)} a_{-j} a_k a_{-j}(\sigma).
\label{U2 HL1}
\end{align}

We now turn to the high-high frequency interaction.
We begin by splitting the sum:
\begin{equation}
2^k \sum_{\ell \geq k} 2^{(\ell - k)/2} \alpha_\ell \beta_\ell
\lesssim 2^k \sum_{k \leq \ell < -j} 2^{(\ell - k)/2} \alpha_\ell \beta_\ell
+ 2^k \sum_{\ell \geq -j} 2^{(\ell - k)/2} \alpha_\ell \beta_\ell.
\label{U2 HH1 Split}
\end{equation}
Then
\begin{align}
2^k \sum_{k \leq \ell < -j} 2^{(\ell - k)/2} \alpha_\ell \beta_\ell
&\lesssim
2^k 2^{-\sigma k} a_{-j}(\sigma) \sum_{k \leq \ell < - j} 2^{(\ell - k)/2}
2^{-\delta(j + \ell)} \sum_{p = \ell}^{-j} a_p^2 \nonumber \\
&\lesssim
2^k 2^{-\sigma k} 2^{-(k + j)/2} a_{-j}^2 a_{-j}(\sigma).
\label{U2 HH1 S1}
\end{align}
As for the second summand, we have
\begin{align}
2^k \sum_{\ell \geq -j} 2^{(\ell - k)/2} \alpha_\ell \beta_\ell
&\lesssim
2^k \sum_{\ell \geq -j} \langle 2^{j + \ell} \rangle^{-8} 2^{(\ell - k)/2}
2^{\ell + j} a_{-j} a_\ell(\sigma) 2^{-\sigma \ell} a_\ell \nonumber \\
&\lesssim
2^k 2^{-\sigma k} 2^{-(k + j)/2} a_{-j}^2 a_{-j}(\sigma).
\label{U2 HH1 S2}
\end{align}

Combining (\ref{U2 LH1})--
(\ref{U2 HH1 S2}) yields
\begin{equation}
2^k \lVert P_k (A_\ell(r) f(r) )\rVert_{F_k(T) \cap S_k^{1/2}(T)} \lesssim \varepsilon
2^{2k} 2^{-\sigma k} 2^{-(k + j)/2} a_{-j}(\sigma)
\quad \quad k + j \leq 0.
\label{A Psi High}
\end{equation}

Now let us assume that $k + j \geq 0$.
The low-high frequency interaction we first split into two pieces.
\begin{equation}
\sum_{\ell \leq k} 2^\ell \alpha_\ell \beta_k
\lesssim
\sum_{\ell \leq -j} 2^\ell \alpha_\ell \beta_k +
\sum_{-j < \ell \leq k} 2^\ell \alpha_\ell \beta_k
\label{U2 LH2 Split}
\end{equation}
For the first term, we have
\begin{align}
\sum_{\ell \leq -j} 2^\ell \alpha_\ell \beta_k
&\lesssim \langle 2^{j + k} \rangle^{-8} 2^{-\sigma k} a_k(\sigma)
\sum_{\ell \leq -j} \sum_{p = \ell}^{-j} a_p^2 \nonumber \\
&\lesssim
\langle 2^{j + k} \rangle^{-8} 2^{-\sigma k} a_{-j}^2 a_k(\sigma) \sum_{\ell \leq -j} 2^\ell
\sum_{p = \ell}^{-j} 2^{-2 \delta (j + p)}.
\label{U2 LH2 S1}
\end{align}
Then
\begin{equation}
\sum_{\ell \leq -j} 2^\ell \sum_{p = \ell}^{-j} 2^{-2 \delta (j + p)}
\lesssim
\sum_{\ell \leq -j} 2^\ell 2^{-2 \delta (j + \ell)}
\lesssim 2^{-j} \leq 2^k.
\label{U2 LH2 S1 RTS}
\end{equation}
As for the second summand,
\begin{align}
\sum_{-j < \ell \leq k} 2^\ell \alpha_\ell \beta_k
&\lesssim
\langle 2^{j + k} \rangle^{-8} 2^{-\sigma k} a_k(\sigma)
\sum_{-j < \ell \leq k} \langle 2^{j + \ell} \rangle^{-8} 2^\ell 2^{\ell + j} a_{-j} a_\ell \nonumber \\
&\lesssim
\langle 2^{j + k} \rangle^{-8} 2^{-\sigma k} 2^k a_{-j}^2 a_k(\sigma) .
\label{U2 LH2 S2}
\end{align}

The high-low frequency interaction is controlled by
\begin{align}
\sum_{\ell \leq k} 2^\ell \alpha_k \beta_\ell
&\lesssim
\langle 2^{j + k} \rangle^{-8} 2^{-\sigma k} 2^{k + j} a_{-j} a_k(\sigma)
\sum_{\ell \leq k} \langle 2^{j + \ell} \rangle^{-8} 2^\ell a_\ell \nonumber \\
&\lesssim
\langle 2^{j + k} \rangle^{-8} 2^{-\sigma k} 2^k a_{-j}^2 a_k(\sigma).
\label{U2 HL2}
\end{align}

Finally, the high-high frequency interaction is controlled by
\begin{align}
\sum_{\ell \geq k} 2^\ell \alpha_\ell \beta_\ell
&\lesssim
\sum_{\ell \geq k} \langle 2^{j + \ell} \rangle^{-8} 2^\ell
2^{\ell + j} a_{-j} a_\ell 2^{-\sigma \ell} a_\ell(\sigma) \nonumber \\
&\lesssim
\langle 2^{j + k} \rangle^{-8} 2^{-\sigma k} 2^k a_{-j} a_k a_{k}(\sigma).
\label{U2 HH2}
\end{align}

Thus, in view of (\ref{U2 LH2 Split})--(\ref{U2 HH2}), 
we have shown that
\begin{equation}
2^k \lVert P_k ( A_\ell(r) f(r) ) \rVert_{F_k(T) \cap S_k^{1/2}(T)}
\lesssim \varepsilon
\langle 2^{j + k} \rangle^{-8} 2^{-\sigma k} 2^{2k} a_k(\sigma)
\quad \quad
k + j \geq 0.
\label{A Psi Low}
\end{equation}

Combining (\ref{F Psi High}), (\ref{F Psi Low}),
(\ref{A Psi High}), and (\ref{A Psi Low}) proves the lemma.
\end{proof}

\begin{lem}
It holds that
\begin{equation*}
\lVert \int_0^s
e^{(s - s^\prime) \Delta}
P_k U_m (s^\prime) ds^\prime \rVert_{F_k(T) \cap S_k^{1/2}(T)}
\lesssim \varepsilon
(1 + s 2^{2k} )^{-4} 2^{-\sigma k} a_k(\sigma).
\end{equation*}
\label{Nonlinear Evolution Bound}
\end{lem}
\begin{proof}
Let $k_0 \in \mathbf{Z}$ be such that $s \in [2^{2k_0 - 1}, 2^{2k_0 + 1})$.
If $k + k_0 \leq 0$, then it follows from Lemma \ref{U Bound} that
\begin{align*}
\lVert \int_0^s e^{(s - r)\Delta} P_k U_m(r) dr \rVert_{F_k(T) \cap S_k^{1/2}(T)}
&\lesssim
\sum_{j \leq k_0} \int_{2^{2j - 1}}^{2^{2j + 1}} \lVert P_k U_m(r) \rVert_{F_k(T) \cap S_k^{1/2}(T)} dr \\
&\lesssim
\sum_{j \leq k_0} 2^{2j} \varepsilon 2^{-\sigma k} 2^{2k} (a_k(\sigma) + 2^{-3(k + j)/2} a_{-j}(\sigma)) \\
&\lesssim
\varepsilon 2^{-\sigma k} a_k (\sigma) \sum_{j \leq k_0} 2^{2k + 2j} (1 + 2^{-3(k + j)/2}
2^{-\delta (k + j) }) \\
&\lesssim
\varepsilon 2^{-\sigma k} a_k(\sigma).
\end{align*}
On the other hand, if $k + k_0 > 0$, then
\begin{align}
\lVert \int_0^s e^{(s - r)\Delta} P_k U_m(r) dr \rVert_{F_k(T) \cap S_k^{1/2}(T)}
\lesssim &\;
\int_0^{s/2} \lVert e^{(s - r)\Delta} P_k U_m(r) \rVert_{F_k(T) \cap S_k^{1/2}(T)} dr + \nonumber \\
&\int_{s/2}^s \lVert e^{(s - r)\Delta} P_k U_m(r) \rVert_{F_k(T) \cap S_k^{1/2}(T)} dr \nonumber \\
\lesssim &\;
\sum_{j \leq k_0} 2^{-20(k + k_0)} 2^{2j} d_{k,j} + 2^{2k_0}d_{k,k_0} \nonumber \\
\lesssim &\;
2^{-20(k_0 + k)} \sum_{j \leq k_0} 2^{2j} d_{k,j} + 2^{-2k}d_{k,k_0}.
\label{LEB 1}
\end{align}
By Lemma \ref{U Bound} and the fact that $k + k_0 > 0$, it holds that
\begin{equation*}
2^{-2k} d_{k, k_0} \lesssim \varepsilon \langle 2^{k_0 + k} \rangle^{-8} 2^{-\sigma k} a_k(\sigma)
\end{equation*}
and
\begin{align*}
2^{-20(k_0 + k)}&
\sum_{j \leq k_0} 2^{2j} d_{k,j} \nonumber \\
\lesssim &\;
2^{-20(k_0 + k)}
\sum_{j \leq k_0} \varepsilon \langle 2^{j + k} \rangle^{-8} 2^{-\sigma k} 2^{2k}
\left(2^{2j} a_k(\sigma) + 2^{j/2} 2^{-3k/2} a_{-j}(\sigma)\right) \\
\lesssim &\;
\varepsilon 2^{-\sigma k} a_k(\sigma)
2^{-20(k_0 + k)}
\sum_{j \leq k_0} \langle 2^{j + k} \rangle^{-8} \left(2^{2j + 2k} + 2^{(j + k)/2} 2^{\delta \lvert j + k \rvert} \right) \\
\lesssim &\;
\varepsilon \langle 2^{k_0 + k} \rangle^{-8} 2^{-\sigma k} a_k(\sigma),
\end{align*}
which, combined with (\ref{LEB 1}), completes the proof of the lemma.
\end{proof}

\begin{lem}
The bound (\ref{PsiF}) holds:
\begin{equation*}
\lVert P_k \psi_m(s) \rVert_{F_k(T) \cap S_k^{1/2}(T)} \lesssim
(1 + s 2^{2k})^{-4} 2^{-\sigma k} b_k(\sigma).
\end{equation*}
\label{Psi Bound}
\end{lem}
\begin{proof}
In view of (\ref{Psi Duhamel}), we have
\begin{equation*}
P_k \psi_m (s) =
e^{s \Delta} P_k \psi_m(0) +
\int_0^s e^{(s - r) \Delta} P_k U_m(r) \; dr.
\end{equation*}
Then it follows from Lemma \ref{Nonlinear Evolution Bound} that
\begin{equation*}
\lVert P_k \psi_m(s) \rVert_{F_k(T) \cap S_k^{1/2}(T)} \lesssim
2^{-\sigma k} (1 + s 2^{2k})^{-4} (b_k(\sigma) + \varepsilon a_k(\sigma)),
\quad \quad 0 \leq \sigma \leq \sigma_1 - 1 .
\end{equation*}
Therefore $a_k(\sigma) \lesssim b_k(\sigma) + \varepsilon a_k(\sigma)$ and hence
\begin{equation}
a_k(\sigma) \lesssim b_k(\sigma).
\label{a lies under b}
\end{equation}
\end{proof}

\subsection{Connection coefficient control} \label{SS:CCC}
The main results of this subsection are the $L^2_{t,x}$ bounds
(\ref{AxL2}) and (\ref{AtL2}), respectively
proven in Corollary \ref{C:AxL4} and Lemma \ref{L:AtL4}, and the frequency-localized
$L^2_{t,x}$ bounds (\ref{PkAxL2}) and (\ref{PkAtL2}), respectively proven in
Corollaries \ref{C:PkAxL2} and \ref{C:PkAtL2}.

\begin{lem}
Let $s \in [2^{2j - 2}, 2^{2j + 2}]$.  Then it holds that
\begin{equation*}
\lVert P_k (A_\ell (s) \psi_m(s) ) \rVert_{F_k(T) \cap S_k^{1/2}(T)} \lesssim
\varepsilon (1 + s 2^{2k})^{-3} (s 2^{2k})^{-3/8} 2^k 2^{-\sigma k} b_k(\sigma).
\end{equation*}
\label{A Psi Bound}
\end{lem}
\begin{proof}
Using (\ref{A Psi High}) and (\ref{Slowly Varying}), we have
\begin{equation}
2^k \lVert P_k (A_\ell (s) \psi_m(s) ) \rVert_{F_k(T) \cap S_k^{1/2}(T)}
\lesssim
\varepsilon 2^{2k} 2^{-\sigma k} 2^{- (1/2 + \delta) (k + j) } a_k(\sigma).
\label{A Psi High k}
\end{equation}
Combining (\ref{A Psi High k}), (\ref{A Psi Low}), and (\ref{a lies under b}) then yields
\begin{displaymath}
\lVert P_k (A_\ell (s) \psi_m(s) ) \rVert_{F_k(T) \cap S_k^{1/2}(T)} \lesssim
\left\{
\begin{array}{ll}
\varepsilon (s 2^{2k})^{-3/8}
2^k 2^{-\sigma k} b_k(\sigma) 
& \textrm{if $k + j \leq 0$} \\ \\
\varepsilon (1 + s 2^{2k})^{-4} 2^k 2^{-\sigma k} b_k(\sigma) 
& \textrm{if $k + j \geq 0$},
\end{array} \right.
\end{displaymath}
which proves the lemma.
\end{proof}

\begin{lem}
Assume that $T \in (0, 2^{2 \mathcal{K}}]$, $f, g \in H^{\infty, \infty}(T)$, $P_k f \in S_k^\omega(T)$,
and
$P_k g \in L_{t,x}^4$ for some $\omega \in [0, 1/2]$ and all $k \in \mathbf{Z}$.
Set
\begin{equation*}
\mu_k := \sum_{\lvert j - k \rvert \leq 20}
\lVert P_{j} f \rVert_{S_{k^\prime}^\omega(T)},
\quad\quad
\nu_k := \sum_{\lvert j - k \rvert \leq 20}
\lVert P_{j} g \rVert_{L_{t,x}^4}.
\end{equation*}
Then, for any $k \in \mathbf{Z}$,
\begin{equation*}
\lVert P_k(fg) \rVert_{L_{t,x}^4}
\lesssim
\sum_{j \leq k} 2^j \mu_j \nu_k + 
\sum_{j \leq k} 2^{(k + j)/2} \mu_k \nu_j +
2^k \sum_{j \geq k} 2^{-\omega (j - k) } \mu_j \nu_j.
\end{equation*}
\label{BIKT Lemma 5.4}
\end{lem}
\begin{proof}
For the proof, see \cite[\S 5]{BeIoKeTa11}.
\end{proof}

\begin{lem}
It holds that
\begin{equation*}
\lVert P_k \psi_s(0) \rVert_{L_{t,x}^4} +
\lVert P_k \psi_t(0) \rVert_{L_{t,x}^4} 
\lesssim 2^k \tilde{b}_k
(1 + \sum_j b_j^2).
\end{equation*}
\label{Psi_t Bound}
\end{lem}
\begin{proof}
We only treat $\psi_t(0)$ since $\psi_s(0)$ and $\psi_t(0)$ differ only by
a factor of $i$.
As $\psi_t(0) = i D_\ell(0) \psi_\ell (0)$, we have
\begin{equation*}
\psi_t(0) = i \partial_\ell \psi_\ell(0) - A_\ell(0) \psi_\ell(0).
\end{equation*}
Clearly
\begin{equation*}
\lVert P_k \partial_\ell \psi_\ell(0) \rVert_{L_{t,x}^4}
\lesssim
2^k \lVert P_k \psi_x(0) \rVert_{L_{t,x}^4}
\lesssim
2^k \tilde{b}_k.
\end{equation*}
For the remaining term, we apply Lemma \ref{BIKT Lemma 5.4}, bounding
$P_j A_\ell(0)$ in $S_j^{1/2}$ by $\sum_{p} b_p^2$, which follows from
Lemma \ref{A Bound}.  We get
\begin{align*}
\lVert P_k(A_\ell(0) \psi_\ell(0)) \rVert_{L_{t,x}^4} 
\lesssim& \;
\sum_{j \leq k} 2^j (\sum_p b_p^2) \tilde{b}_k +
\sum_{j \leq k} 2^{(k + j)/2} (\sum_p b_p^2) \tilde{b}_j + \\
&\; 2^k \sum_{j \geq k} 2^{-(j - k)/2} (\sum_p b_p^2) \tilde{b}_j.
\end{align*}
Therefore
\begin{equation*}
\lVert P_k (A_\ell \psi_\ell(0) ) \rVert_{L_{t,x}^4} \lesssim 2^k \tilde{b}_k (\sum_j b_j^2).
\end{equation*}
\end{proof}

\begin{cor}\label{C:AltL4Psit}
It holds that
\begin{equation*}
\lVert P_k \psi_s(0) \rVert_{L_{t,x}^4} +
\lVert P_k \psi_t(0) \rVert_{L_{t,x}^4} 
\lesssim 2^k 2^{-\sigma k} b_k(\sigma)
(1 + \sum_j b_j^2).
\end{equation*}
\end{cor}
\begin{proof}
Without loss of generality, we prove the bound only for $\psi_t$.
We have
\[
\lVert P_k \partial_\ell \psi_\ell(0) \rVert_{L^4_{t,x}}
\lesssim 
2^k \lVert P_k \psi_x(0) \rVert_{L^4_{t,x}}
\lesssim
2^k 2^{-\sigma k} b_k(\sigma).
\]
It remains to control $P_k(A_\ell(0) \psi_\ell(0))$ in $L^4_{t,x}$.
The obstruction to applying Lemma \ref{BIKT Lemma 5.4} as we did in
Lemma \ref{Psi_t Bound} is the high-low
interaction, for which summation can be achieved only for small $\sigma$. 
If we restrict the range of $\sigma$ to
$\sigma < 1/2 - 2 \delta$, then we ensure the constant remains bounded
and can apply Lemma \ref{BIKT Lemma 5.4} as in Lemma \ref{Psi_t Bound}.

For $\sigma \geq 1/2 - 2 \delta$, we can still apply the bounds of Lemma \ref{BIKT Lemma 5.4}
to the low-high and high-high
interactions. For the remaining high-low interaction, we
bound $A_\ell(0)$ in $L^4_{t,x}$ and $\psi_\ell(0)$ in $L^\infty_{t,x}$.
In particular, we have, thanks to (\ref{Pk A4}) and Bernstein, that
\begin{align*}
\sum_{ \substack{ | j_1 - k | \leq 4 \\ j_2 \leq k + 4} }
\lVert P_k ( P_{j_1} A_\ell(0) P_{j_2} \psi_\ell(0)) \rVert_{L^4_{t,x}}
&\lesssim \;
\sum_{ \substack{ | j_1 - k | \leq 4 \\ j_2 \leq k + 4} }
\lVert P_{j_1} A_\ell(0) \rVert_{L^4_{t,x}}
\lVert P_{j_2} \psi_\ell(0) \rVert_{L^\infty_{t,x}} \\
&\lesssim \;
\sum_{ \substack{ | j_1 - k | \leq 4 \\ j_2 \leq k + 4} }
2^{-\sigma j_1} b_{j_1} b_{j_1}(\sigma)
2^{j_2} \lVert P_{j_2} \psi_\ell(0) \rVert_{L^\infty_t L^2_x} \\
&\lesssim \;
\sum_{j_2 \leq k +4}
2^{-\sigma k} b_k b_k(\sigma)
2^{j_2} b_{j_2} \\
&\lesssim \;
2^{-\sigma k} b_k^2 b_k(\sigma)
\sum_{j_2 \leq k + 4} 2^k 2^{(j_2 - k) + (k - j_2) \delta} \\
&\lesssim \;
2^{-\sigma k} 2^k b_k^2 b_k(\sigma).
\end{align*}
\end{proof}

\begin{lem}
It holds that
\begin{equation*}
\lVert P_k \psi_s(s) \rVert_{L_{t,x}^4} +
\lVert P_k \psi_t(s) \rVert_{L_{t,x}^4} 
\lesssim
(1 + s 2^{2k})^{-2} 2^k \tilde{b}_k
(1 + \sum_j b_j^2).
\end{equation*}
\label{L:Psi_t Flow Bound}
\end{lem}
\begin{proof}
We treat only $\psi_t(s)$ since the proof for $\psi_s(s)$ is analogous.
From (\ref{Psi Duhamel}) we have
\begin{equation*}
\psi_t(s) = e^{s \Delta} \psi_t(0) + \int_0^s e^{(s - r)\Delta} U_t(r) \; dr.
\end{equation*}
We claim that
\begin{equation}
\lVert \int_0^s e^{(s - r) \Delta} P_k U_t(r) \; dr \rVert_{L_{t,x}^4}
\lesssim \varepsilon (1 + s 2^{2k})^{-2} 2^k \tilde{b}_k (1 + \sum_j b_j^2),
\label{U L4 Bound}
\end{equation}
which combined with Lemma \ref{Psi_t Bound} and a standard iteration argument
proves the lemma.

As in the proof of Lemma \ref{Nonlinear Evolution Bound}, we take
\begin{equation*}
F \in \{ A_\ell^2, \partial_\ell A_\ell, fg : \ell = 1, 2;
f, g \in \{ \psi_m, \overline{\psi_m}: m = 1,2\} \}.
\end{equation*}
By (\ref{Quadratic Parabolic}) and (\ref{a lies under b}) we have
\begin{equation}\label{PkFSk}
\lVert P_k F(r) \rVert_{S_k^{1/2}(T)} \lesssim
\varepsilon^{1/2}
(1 + s 2^{2k})^{-2} (s 2^{2k})^{-5/8} 
2^k b_k.
\end{equation}
Moreover, by Lemma \ref{A Bound}
\begin{equation}\label{PkASk}
\lVert P_k A_\ell(r) \rVert_{S_k^{1/2}(T)} \lesssim
\varepsilon^{1/2}
(1 + s 2^{2k})^{-3} (s 2^{2k})^{-1/8} b_k.
\end{equation}
Applying Lemma \ref{BIKT Lemma 5.4} with $\omega = 1/2$ yields
\begin{equation}
\lVert P_k( F(r) \psi_t(r) ) \rVert_{L_{t,x}^4} +
2^k \lVert P_k( A_\ell(r) \psi_t(r) \rVert_{L_{t,x}^4}
\lesssim
\varepsilon
(1 + s 2^{2k})^{-2} (s 2^{2k})^{-7/8}
2^k \tilde{b}_k (1 + \sum_j b_j^2).
\label{Prelim U L4 Bound}
\end{equation}
Integrating with respect to $s$ yields
\begin{equation*}
\int_0^s (1 + (s - r)2^{2k})^{-N} (1 + r 2^{2k})^{-2} (r 2^{2k})^{-7/8} dr
\lesssim
2^{- 2k} (1 + s 2^{2k})^{-2},
\end{equation*}
which, together with (\ref{Prelim U L4 Bound}), implies (\ref{U L4 Bound}).
\end{proof}

\begin{lem}
It holds that
\begin{equation}
\lVert P_k A_m(0) \rVert_{L_{t,x}^4} \lesssim
2^{-\sigma k} b_k b_k(\sigma).
\label{Pk A4}
\end{equation}
\end{lem}
\begin{proof}
We have
\begin{equation*}
\lVert P_k \psi_m(s) \rVert_{S_k^0} \lesssim (1 + s 2^{2k})^{-4} 2^{-\sigma k} b_k(\sigma)
\end{equation*}
and
\begin{equation*}
\lVert P_k (D_\ell \psi_\ell) (s) \rVert_{L_{t,x}^4} \lesssim
(1 + s 2^{2k})^{-3} (s 2^{2k})^{-3/8} 2^k 2^{-\sigma k} b_k(\sigma).
\end{equation*}
Applying Lemma \ref{BIKT Lemma 5.4} with $\omega = 0$, we get
\begin{align*}
\lVert P_k A_m(0) \rVert_{L_{t,x}^4} \lesssim& \;
\sum_{\ell = 1,2} \int_0^\infty \lVert P_k (\overline{\psi_m(s)} D_\ell \psi_\ell(s)) \rVert_{L_{t,x}^4} \ds \\
\lesssim& \;
2^{-\sigma k} \sum_{j \leq k} b_j b_k(\sigma) 2^{j + k}
\int_0^\infty (1 + s 2^{2k})^{-3} (s 2^{2k})^{-3/8} \ds \\
&+ 2^{-\sigma k} \sum_{j \leq k} b_k(\sigma) b_j 2^{(k + j)/2} 2^{j}
\int_0^\infty (1 + s 2^{2k})^{-4} (s 2^{2j})^{-3/8} \ds \\
&+ \sum_{j \geq k} 2^{-\sigma j} b_j(\sigma) b_j 2^{k - j} 2^{2j}
\int_0^\infty (1 + s 2^{2j})^{-7} (s 2^{2j})^{-3/8} \ds.
\end{align*}
Call the integrals $I_1$, $I_2$, and $I_3$, respectively.
Clearly $I_1$ and $I_3$ satisfy $I_1 \lesssim 2^{-2k}$
and $I_3 \lesssim 2^{-2j}$.  By Cauchy-Schwarz, $I_2$ satisfies
\begin{align*}
I_2 &\lesssim \left( \int_0^\infty (1 + s 2^{2k})^{-8} (1 + s 2^{2j})^{4} \ds \right)^{1/2}
\left( \int_0^\infty (1 + s 2^{2j})^{-4} (s 2^{2j})^{-3/8} \ds \right)^{1/2} \\
&\lesssim
2^{-j - k}.
\end{align*}
Therefore
\begin{align*}
\lVert P_k A_m(0) \rVert_{L_{t,x}^4}
&\lesssim
2^{-\sigma k} b_k(\sigma) \sum_{j \leq k} \left( b_j 2^{j - k} + b_j  2^{(j - k)/2} \right)
+ 2^{-\sigma k} \sum_{j \geq k} b_j(\sigma) b_j 2^{k - j} \\
&\lesssim 2^{-\sigma k} b_k b_k(\sigma).
\end{align*}
\end{proof}

\begin{cor}
It holds that
\begin{equation*}
\lVert A_x^2(0) \rVert_{L^2_{t,x}}
\lesssim
\sup_{j \in \mathbf{Z}} b_j^2 \cdot \sum_{k \in \mathbf{Z}} b_k^2.
\end{equation*}
\label{C:AxL4}
\end{cor}
\begin{proof}
We have
\begin{align*}
\lVert A_x^2(0) \rVert_{L^2_{t,x}}
&\lesssim
\lVert A_x(0) \rVert_{L^4_{t,x}}^2 \\
&\lesssim
\sum_{k \in \mathbf{Z}} \lVert P_k A_x(0) \rVert_{L^4_{t,x}}^2 \\
&\lesssim
\sup_{j \in \mathbf{Z}} b_j^2 \cdot \sum_{k \in \mathbf{Z}} b_k^2.
\end{align*}
\end{proof}

\begin{cor}
Let $\sigma \geq 2 \delta$. Then it holds that
\begin{equation*}
\lVert P_k A_x^2(0) \rVert_{L^2_{t,x}} 
\lesssim
2^{-\sigma k} b_k(\sigma) \cdot \sup_{j} b_j \cdot \sum_{\ell \in \mathbf{Z}} b_\ell^2.
\end{equation*}
\label{C:PkAxL2}
\end{cor}
\begin{proof}
We perform a Littlewood-Paley decomposition and invoke (\ref{C:AxL4}).

Consider first the high-low interactions:
\begin{align*}
\sum_{\substack{ |j_2 - k| \leq 4 \\ j_1 \leq k - 5}}
\lVert P_{k} (P_{j_1} A_x P_{j_2} A_x ) \rVert_{L^2}
&\lesssim \;
\sum_{\substack{ |j_2 - k| \leq 4 \\ j_1 \leq k - 5}}
\lVert P_{j_1} A_x \rVert_{L^4} \lVert P_{j_2} A_x \rVert_{L^4} \\
&\lesssim \;
2^{-\sigma k} b_{k} b_{k}(\sigma) 
\sum_{j_1 \leq k - 5} 
b_{j_1}^2 .
\end{align*}
Next consider the high-high interactions:
\begin{align*}
\sum_{\substack{j_1, j_2 \geq k - 4 \\ |j_1 - j_2| \leq 8}}
\lVert P_{k} (P_{j_1} A_x P_{j_2} A_x ) \rVert_{L^2}
&\lesssim \;
\sum_{\substack{j_1, j_2 \geq k - 4 \\ |j_1 - j_2| \leq 8}}
\lVert P_{j_1} A_x \rVert_{L^4} \lVert P_{j_2} A_x \rVert_{L^4} \\
&\lesssim \;
\sum_{j \geq k - 4}
2^{-\sigma j}b_j(\sigma) b_j^3.
\end{align*}
Using the frequency envelope property, we bound this last sum by
\begin{align*}
\sum_{j \geq k - 4}
2^{-\sigma j}b_j(\sigma) b_j^3
&\lesssim \;
2^{-\sigma k} b_k(\sigma)
\sum_{j \geq k - 4}
2^{-\sigma (j - k)} 2^{\delta (j - k)} b_j^3 \\
&\lesssim \;
2^{-\sigma k} b_k(\sigma)
\sup_{j  \geq k - 4} b_j \cdot 
\sum_{j \geq k - 4} b_j^2
\end{align*}
It is in controlling this last sum that we use $\sigma > \delta+$.
\end{proof}

\begin{lem}
It holds that
\begin{equation*}
\lVert A_t(0) \rVert_{L_{t,x}^2}
\lesssim
(1 + \sum_j b_j^2)^2 \sum_k \lVert P_k \psi_x(0) \rVert_{L_{t,x}^4}^2.
\end{equation*}
\label{L:AtL4}
\end{lem}
\begin{proof}
We begin with
\begin{equation}
\lVert A_t(0) \rVert_{L_{t,x}^2} 
\lesssim
\int_0^\infty \lVert (\overline{\psi}_t \cdot D_\ell \psi_\ell)(s) \rVert_{L_{t,x}^2} \ds.
\label{A_t Starting Point}
\end{equation}
If we define
\begin{equation}
\mu_k(s) := \sup_{k^\prime \in \mathbf{Z}} 2^{-\delta \lvert k - k^\prime \rvert}
\lVert P_k \psi_t(s) \rVert_{L_{t,x}^4}
\quad \text{and} \quad
\nu_k(s) := \sup_{k^\prime \in \mathbf{Z}} 2^{-\delta \lvert k - k^\prime \rvert}
\lVert P_k (D_\ell \psi_\ell)(s) \rVert_{L_{t,x}^4},
\label{munu}
\end{equation}
then
\begin{equation}
\lVert (\overline{\psi}_t \cdot D_\ell \psi_\ell)(s) \rVert_{L_{t,x}^2} 
\lesssim
\sum_k \mu_k(s) \sum_{j \leq k} \nu_j(s) +
\sum_k \nu_k(s) \sum_{j \leq k} \mu_j(s).
\label{L2  Envelope Bound}
\end{equation}
From Lemmas \ref{Psi_t Bound}, \ref{Psi Bound}, and \ref{A Psi Bound}, it follows that
\begin{equation}
\mu_k(s), \nu_k(s) \lesssim (1 + s 2^{2k})^{-2} 2^k \tilde{b}_k (1 + \sum_p b_p^2).
\label{mu nu Bound}
\end{equation}
Combining (\ref{A_t Starting Point}), (\ref{L2 Envelope Bound}), and (\ref{mu nu Bound}),
we have
\begin{align*}
\lVert A_t(0) \rVert_{L_{t,x}^2}
&\lesssim 
\sum_k \mu_k(s) \sum_{j \leq k} \nu_j(s) \\
&\lesssim 
(1 + \sum_p b_p^2)^2 \sum_k 2^k \tilde{b}_k
\sum_{j \leq k} 2^j \tilde{b}_j
\int_0^\infty (1 + s 2^{2j})^{-2} (1 + s 2^{2k})^{-2} \ds \\
&\lesssim
(1 + \sum_p b_p^2)^2 \sum_k 2^k \tilde{b}_k
\sum_{j \leq k} 2^j \tilde{b}_j
\int_0^\infty (1 + s 2^{2k})^{-2} \ds \\
&\lesssim
(1 + \sum_p b_p^2)^2 \sum_k 2^{2k} \tilde{b}_k^2
\int_0^\infty (1 + s 2^{2k})^{-2} \ds \\
&\lesssim
(1 + \sum_p b_p^2)^2 \sum_k \tilde{b}_k^2.
\end{align*}
\end{proof}

As a corollary of the proof, we also obtain
\begin{cor}
Let $\sigma \geq 2 \delta$.
It holds that
\begin{equation*}
\lVert P_k A_t \rVert_{L^2_{t,x}}
\lesssim
(1 + \sum_p b_p^2) \tilde{b}_k 2^{-\sigma k} b_k(\sigma).
\end{equation*}
\label{C:PkAtL2}
\end{cor}
\begin{proof}
We start by modifying the proof of Lemma \ref{L:AtL4}, taking $\mu_k$ and $\nu_k$
as in (\ref{munu}). Then
\begin{align*}
\lVert P_k A_t \rVert_{L^2} &\lesssim \;
\int_0^\infty \lVert P_k (\overline{\psi}_t \cdot D_\ell \psi_\ell)(s) \rVert_{L^2_{t,x}} ds \\
&\lesssim \;
\int_0^\infty \left( \mu_k(s) \sum_{j \leq k} \nu_j(s) + \nu_k \sum_{j \leq k} \mu_j(s)
+ \sum_{j \geq k} \mu_j(s) \nu_j(s) \right) ds.
\end{align*}
Combining Lemmas \ref{Psi Bound} and \ref{A Psi Bound} gives a bound on $\nu_k$ of
\begin{equation}
\lVert \nu_k(s) \rVert_{L^4} \lesssim (1 + s 2^{2k})^{-3} (s 2^{2k})^{-3/8} 2^k 2^{-\sigma k} b_k(\sigma),
\label{nu bound}
\end{equation}
which leads to
\[
\int_0^\infty \nu_k \sum_{j \leq k} \mu_j(s) ds
\lesssim
(1 + \sum_p b_p^2) \tilde{b}_k 2^{-\sigma k} b_k(\sigma).
\]
Also, by using (\ref{mu nu Bound}) for $\mu_k$ and (\ref{nu bound}) for $\nu_k$ yields
\begin{align*}
\int_0^\infty \sum_{j \geq k} \mu_j(s) \nu_j(s) ds
&\lesssim \;
(1 + \sum_p b_p^2) \sum_{j \geq k} 2^{2j} 2^{-\sigma j} b_j(\sigma) \tilde{b}_j 
\int_0^\infty (1 + s 2^{2j})^{-3} (s 2^{2j})^{-3/8} ds \\
&\lesssim \;
(1 + \sum_p b_p^2) \sum_{j \geq k} 2^{-\sigma j} b_j(\sigma) \tilde{b}_j \\
&\lesssim \;
(1 + \sum_p b_p^2) 2^{-\sigma k} b_k(\sigma)
\sum_{j \geq k} 2^{(\delta -\sigma) (j - k)} \tilde{b}_j \\
&\lesssim \;
(1 + \sum_p b_p^2) 2^{-\sigma k} \tilde{b}_k b_k(\sigma).
\end{align*}
Here we have used $\sigma \geq 2 \delta$.
It remains to consider
\[
\int_0^\infty \mu_k(s) \sum_{j \leq k} \nu_j(s) ds.
\]
Suppose that
\begin{equation}
\mu_k(s) \lesssim (1 + s2^{2k})^{-2} 2^k 2^{-\sigma k} b_k(\sigma) (1 + \sum_p b_p^2).
\label{mymu}
\end{equation}
Then
\begin{align*}
\int_0^\infty &\mu_k(s) \sum_{j \leq k} \nu_j(s) ds \\
&\lesssim \;
(1 + \sum_p b_p^2)^2  2^{-\sigma k} b_k(\sigma)
2^k \sum_{j \leq k} \int_0^\infty (1 + s 2^{2k})^{-2} (1 + s 2^{2j})^{-2} 2^j \tilde{b}_j ds \\
&\lesssim \;
(1 + \sum_p b_p^2)^2  2^{-\sigma k} b_k(\sigma)
2^k \sum_{j \leq k} 2^j \tilde{b}_j \int_0^\infty (1 + s 2^{2k})^{-2} ds \\
&\lesssim \;
(1 + \sum_p b_p^2)^2  2^{-\sigma k} b_k(\sigma)
2^{2k} \tilde{b}_k \cdot 2^{-2k} \\
&=
(1 + \sum_p b_p^2)^2  2^{-\sigma k} b_k(\sigma) \tilde{b}_k.
\end{align*}
Hence it remains to establish (\ref{mymu}).

By Corollary \ref{C:AltL4Psit}, (\ref{mymu}) holds when $s = 0$.
To extend this estimate to $s > 0$, we proceed as in the proof of Lemma
\ref{L:Psi_t Flow Bound}, replacing bounds (\ref{PkFSk}) and (\ref{PkASk})
with their $\sigma > 0$ analogues as needed; that these analogues hold
follows from the bounds referenced in establishing (\ref{PkFSk}) and (\ref{PkASk}).
To obtain the analogue of (\ref{Prelim U L4 Bound}), we apply
Lemma \ref{BIKT Lemma 5.4}, choosing to use $\sigma > 0$ bounds only over
the high frequency ranges.
\end{proof}